\documentclass{amsart}

\usepackage{wasysym}
\usepackage{amsmath}
\usepackage{amssymb}
\usepackage{amsthm}
\usepackage{tikz}
\usetikzlibrary{matrix,arrows,backgrounds,shapes.misc,shapes.geometric,patterns,calc,positioning}
\usetikzlibrary{calc,shapes}
\usepackage{wrapfig}
\usepackage{epsfig}  
\usepackage{color}	 %
\input{xy}
\xyoption{poly}
\xyoption{2cell}
\xyoption{all}

\usepackage{lscape}

\usetikzlibrary{calc,shapes}
\usetikzlibrary{matrix,arrows,backgrounds,shapes.misc,shapes.geometric,patterns,calc,positioning}
\usetikzlibrary{calc,shapes}

\newcommand{\cals}{\mathcal{S}}

\newcommand{\calf}{\mathcal{F}}
\newcommand{\cala}{\mathcal{A}}
\newcommand{\caly}{\mathcal{Y}}
\newcommand{\calx}{\mathcal{X}}
\newcommand{\calg}{\mathcal{G}}
\newcommand{\calh}{\mathcal{H}}
\newcommand{\call}{\mathcal{L}}
\newcommand{\calr}{\mathcal{R}}
\newcommand{\callr}{\mathcal{LR}}
\newcommand{\calxs}{\mathcal{XS}}
\newcommand{\calxr}{\mathcal{XR}}	
\newcommand{\calls}{\mathcal{LS}}
\newcommand{\ocalg}{\overline{\mathcal{G}}}

\newcommand{\brac}{\textup{Brac}}
\newcommand{\band}{\calg^\circ}

\newcommand{\tcalg}{\widetilde {\mathcal{G}}}
\newcommand{\calgSW}{{}_{SW}\calg}
\newcommand{\calgNE}{\calg^{N\!E}}
\newcommand{\GSW}{{}_{SW}G}
\newcommand{\GNE}{G^{N\!E}}

\makeatletter
\def\Ddots{\mathinner{\mkern1mu\raise\p@
\vbox{\kern7\p@\hbox{.}}\mkern2mu
\raise4\p@\hbox{.}\mkern2mu\raise7\p@\hbox{.}\mkern1mu}}
\makeatother

\newtheorem{thm}{Theorem}[section]
\newtheorem{prop}[thm]{Proposition}
\newtheorem{lem}[thm]{Lemma}
\newtheorem{cor}[thm]{Corollary}

\theoremstyle{defn}
\newtheorem{defn}[thm]{Definition}
\newtheorem{example}[thm]{Example}

\theoremstyle{remark}

\newtheorem{remark}[thm]{Remark}

\numberwithin{equation}{section}

\DeclareMathSizes{5}{3}{2}{2}

\newcommand{\gi}[3]{G_{#1}, G_{#2}, \ldots, G_{#3}}
\newcommand{\gii}[3]{G'_{#1}, G'_{#2}, \ldots, G'_{#3}}

\newcommand{\gr}[3] {\calg_1 = (G_{#1}, G_{#2}, \ldots, G_{#3})}
\newcommand{\grr}[3] {\calg_2 = (G'_{#1}, G'_{#2}, \dots, G'_{#3})}

\newcommand{\re}[2] {\res_{\calg} ( \calg_{#1}, \calg_{#2})}
\newcommand{\ree}[1] {\res_{\calg} ( \calg_{#1})}
\newcommand{\gt}[4] {\graft_{{#1}, {#2}} ( \calg_{#3}, \calg_{#4})}

\newcommand{\la}[2] {\call ( \calg_{#1} \sqcup \calg_{#2}) }

\newcommand\restr[2]{{
  \left.\kern-\nulldelimiterspace 
  #1 
  \vphantom{\big|} 
  \right|_{#2} 
  }}
  
\newcommand{\e}{\delta}

\def\B{\mathcal{B}}

\newcommand{\za}{\alpha}
\newcommand{\zb}{\beta}
\newcommand{\zd}{\delta}
\newcommand{\zD}{\Delta}
\newcommand{\ze}{\epsilon}
\newcommand{\zg}{\gamma}
\newcommand{\zG}{\Gamma}
\newcommand{\zl}{\lambda}
\newcommand{\zs}{\sigma}
\newcommand{\zS}{\Sigma}

\DeclareMathOperator{\Bang}{Bang}

\DeclareMathOperator{\Brac}{Brac}

\newcommand{\pred}{\textup{pred}}
\newcommand{\suc}{\textup{succ}}
\newcommand{\Int}{\textup{Int}}

\newcommand{\C}{\mathcal{C}}
\newcommand{\A}{\mathcal{A}}

\def\Acal{\mathcal{A}}

\DeclareMathOperator{\res}{Res \,} 
\DeclareMathOperator{\graft}{Graft \,  } 
 
\DeclareMathOperator{\match}{Match \, }

\begin{document}\date{\today}
\title[ Snake graph calculus III: Band graphs and snake rings]{Snake graph calculus and cluster algebras from surfaces III: Band graphs and snake rings}
\author{Ilke Canakci}
\address{Department of Mathematics 
University of Leicester 
University Road 
Leicester LE1 7RH 
United Kingdom}
\email{ic74@le.ac.uk}
\author{Ralf Schiffler}\thanks{The authors were supported by NSF grant DMS-10001637; the first author was also supported by EPSRC grant number EP/K026364/1, UK, and by the University of Leicester,  and the second author was also supported by the NSF grants  DMS-1254567, DMS-1101377 and by the University of Connecticut.}
\address{Department of Mathematics, University of Connecticut, 
Storrs, CT 06269-3009, USA}
\email{schiffler@math.uconn.edu}


\begin{abstract}
We introduce several commutative rings, the \emph{snake rings}, that have strong connections to cluster algebras.
The elements of these rings are residue classes of unions of certain labeled graphs that were used to construct canonical bases in the theory of cluster algebras. 
We obtain \emph{several} rings by varying the conditions on the structure as well as the labelling of the graphs. The most restrictive form of this ring is isomorphic to the ring $\mathbb{Z}[x,y]$ of polynomials in two variables over the integers. A  more general form contains \emph{all} cluster algebras of unpunctured surface type. 

The definition of the rings requires the snake graph calculus which we also complete in this paper building on two earlier articles on the subject. Identities in the snake ring correspond to bijections between the posets of perfect matchings of the  graphs. One of the main results of the current paper is the completion of the explicit construction of these bijections.

\end{abstract}

 \maketitle




\section{Introduction} 
This article is the third and final of a sequence of papers introducing the snake graph calculus, which, on the one hand, is an efficient computational tool for cluster algebras of surface type, and on the other hand,  provides a framework for a more systematic study of the combinatorial structure of abstract snake and band graphs. 

The main result of this paper is two-fold: the completion of the snake graph calculus and the introduction of several commutative rings, the \emph{snake rings}.
The definition of these rings requires the snake graph calculus, and one of our motivating goals for developing the calculus was the introduction of the snake rings.

The elements of the snake rings are residue classes of unions of so-called abstract snake and band graphs, which are a {generalisation} of certain labeled graphs that were used to construct canonical bases in the theory of cluster algebras. Our \emph{abstract} snake and band graphs are not limited to a particular choice of a surface but rather inspired from the combinatorial structure of  \emph{all} surface type cluster algebras.
We obtain several different snake  rings by varying the conditions on the structure as well as on the labelling of the graphs. This ring, in its most restrictive form, is isomorphic to the ring $\mathbb{Z}[x,y]$ of polynomials in two variables over the integers. Thus every polynomial has a realisation as a union of snake and band graphs; actually many such realisations, and it is an interesting question which polynomials correspond to a  connected graph.

In a more general form, the snake ring contains all cluster algebras of unpunctured surface type. This is a very large and   well studied class of cluster algebras. At this level, the snake ring provides a very efficient tool for explicit computation in these cluster algebras. Because of the ubiquity of cluster algebras, this tool is likely to be useful for computations in many areas of mathematics and physics. It has already been used by the authors of the current paper and their co-authors in the areas of representation theory of associative algebras, in the theory of cluster algebras themselves, and in knot theory.
The snake ring also has a natural relation to the skein algebras in hyperbolic geometry via the surface model for cluster algebras.
Moreover, the snake rings   have an interesting connection to the theory of distributive lattices, because  the identities in the snake ring correspond to bijections between the posets of perfect matchings of the  graphs.  
The poset of  perfect matchings of a snake graph or a band graph is a finite distributive lattice. By the fundamental theorem of finite distributive lattices, it is  therefore isomorphic to the lattice $J(\mathcal{P})$ of order ideals
in some poset $\mathcal{P}$ determined by the snake or band graph, see \cite{Stanley}.  

The main motivation for snake graph calculus comes from cluster algebras which were introduced in \cite{FZ1}, and further developed in \cite{FZ2,BFZ,FZ4}, motivated by combinatorial aspects of canonical bases in Lie theory \cite{L1,L2}. A cluster algebra is a subalgebra of a field of rational functions in several variables, and it is given by constructing a distinguished set of generators, the \emph{cluster variables}. These cluster variables are constructed recursively and their computation is rather complicated in general. By construction, the cluster variables are rational functions, but Fomin and Zelevinsky showed in \cite{FZ1} that they are Laurent polynomials with integer coefficients. Moreover, these coefficients are known to be non-negative \cite{LS4}.

An important class of cluster algebras is given by cluster algebras of surface type \cite{GSV,FG1,FG2,FST,FT}. From a classification point of view, this class is very important, since it has been shown in \cite{FeShTu} that almost all (skew-symmetric) mutation finite cluster algebras are of surface type. For generalizations to the skew-symmetrizable case see \cite{FeShTu2,FeShTu3}. The closely related surface skein algebras were studied in \cite{M,T}.

Snake graphs were introduced in the context of cluster algebras of surface type in \cite{MS, MSW} to provide a combinatorial formula for the Laurent expansions of the cluster variables building on previous work in \cite{S2,ST,S3}. A special type of snake graphs had also been studied earlier in \cite{Propp}. Furthermore the band graphs were introduced in \cite{MSW2} in order to construct canonical bases for the cluster algebras in the case where the surface has no punctures and has at least 2 marked points. As an application of the computational tools developed in \cite{CS,CS2} and in the present paper, it is proved in \cite{CLS} that the basis construction of \cite{MSW2} also applies to surfaces with non-empty boundary and with exactly one marked point.

In \cite{MSW2}  the snake graphs and band graphs were constructed using the geometric interpretation of the elements of the cluster algebra as collections of curves in the surface. Fixing an initial cluster in the cluster algebra corresponds to fixing a triangulation $T$ of the surface. Then each cluster variable $x_\zg$ of the cluster algebra corresponds to a unique arc $\zg$ in the surface and the Laurent expansion of $x_\zg$ in the initial cluster is determined by the crossing pattern of the arc $\zg$ with the triangulation $T$. The combinatorial formula of \cite {MS,MSW} for this Laurent expansion was given as a sum over all perfect matchings of the snake graph $\calg_\zg$ determined by $\zg$ and $T$. 

An arbitrary element of the cluster algebra is a polynomial in the cluster variables, thus it corresponds to a formal sum of multisets of arcs in the surface.
In \cite{MW} it was shown that the relations in the cluster algebra are skein relations, which means that they have a geometric interpretation as local smoothing of crossings of curves in the surface. Using the skein relations, one can find the expansion of an arbitrary cluster algebra element $z$  in the canonical bases of \cite{MSW2} by successively smoothing all the local crossings in the multisets of curves corresponding to $z$ in the surface. For example if $\zg$ is some multiset of curves which has a crossing then the corresponding cluster algebra element can be written as 
  \begin{equation}\label{equation 1}x_{\zg}= y_{\za}x_{\za} +y_{\zb}x_{\zb},\end{equation}
  where  $\za, \zb$ are the (isotopy classes of) curves obtained by smoothing the crossing in $\zg$, and $y_{\za},y_{\zb}$ are certain coefficients.
  But in order to \emph{explicitly} compute the equation (\ref{equation 1}) in the cluster algebra, one needs to construct the snake and band graphs for all the curves involved in $\zg,\za$ and $ \zb$ and use them to compute the Laurent expansions of each of the cluster algebra elements involved.

  Our original motivation for the snake graph calculus was to compute the expression on the right hand side of equation (\ref{equation 1}) directly in terms of snake and band graphs. So instead of taking the curves $\zg$, smoothing their crossing to obtain the curves $\za, \zb$ and then constructing their snake or band graphs, we take the snake or band graphs $\calg_\zg$ of $\zg$, determine what a crossing means in terms of these graphs, and then directly construct snake or band graphs $\calg_{\za}$ and $\calg_{\zb}$ as a resolution of this crossing and rewrite equation (\ref{equation 1}) as  an equation of graphs as follows 
  \begin{equation}\label{equation 2}\calg_{\zg}= y_{\za}\calg_{\za} +y_{\zb}\calg_{\zb}.\end{equation}
  
  The list of all possible cases for these resolutions of crossing snake and band graphs is surprisingly long and complicated. The complete list is given in the current paper extending those of  \cite{CS,CS2}. The case where $\calg_{\zg}$ is a pair of snake graphs is treated in \cite{CS}, the case where $\calg_{\zg}$ is a single snake graph in \cite{CS2}  and all the cases where $\calg_{\zg}$ involves a band graph
  are treated in the current paper. 
  
  With this list of rules of snake graph calculus at hand, we obtain a very efficient and manageable tool for computations in the cluster algebra. 
   The advantage of using snake graphs instead of curves is due to the following
   fundamental difference between the smoothing of a crossing of curves and the resolution of a crossing of snake graphs. The definition of smoothing is very simple. It is defined as a local transformation replacing a crossing {\Large $\times$} with the pair of segments 
  $\genfrac{}{}{0pt}{5pt}{\displaystyle\smile}{\displaystyle\frown}$ (resp. {$\supset \subset$}).    
But once this local transformation is done, one needs to find representatives inside the isotopy classes of the resulting curves which realise the minimal number of crossings with the fixed triangulation. This means that one needs to deform the obtained curves isotopically, and to `unwind' them if possible, in order to see their actual crossing pattern, which is crucial for the applications to cluster algebras. This can be quite intricate, especially in a higher genus surface.

The situation for the snake and band graphs is exactly opposite. The definition of the resolution is very complicated because one has to consider many different cases. But once all these cases are worked out, {there is} a complete list of rules in hand, which one can apply very efficiently in actual computations.
The reason for this is that the definitions of the resolutions already take into account the isotopy mentioned above.

As in our earlier papers \cite{CS} and \cite{CS2}, we continue to  develop the theory on an abstract level. We define snake graphs, band graphs and their resolutions in a purely combinatorial way without any reference to a triangulated surface. Resolutions are operations on these graphs which induce bijections on the sets of perfect matchings of the graphs involved. In the case where the graphs actually come from a triangulated surface, we show that our bijections give rise to the skein relations in the cluster algebra. In  particular, we obtain a new proof of the skein relations. We also determine which abstract snake and band graphs can actually be realised in an unpunctured surface, and also which ones can be realised in a surface with punctures, see Theorem \ref{lem unpunctured}.

This abstract approach allows us to introduce the snake rings mentioned above. The elements of the snake rings are residue classes of unions of snake graphs and band graphs.  The difference between the various snake rings stems from imposing different properties on  snake and band graphs and their labels.
The snake ring $\calls$ in Theorem \ref{thm Gamma} contains each cluster algebra of unpunctured surface type as a subring, whereas the snake ring $\calx\caly\cals_{geo}$ of Corollary \ref{corgeo} is isomorphic to $\mathbb{Z}[x,y]$. 

The snake rings also provide an algebraic framework for the systematic study of abstract snake and band graphs and their lattices of perfect matchings as combinatorial objects, and we will present results  in this direction in a forthcoming paper. 

\smallskip
The snake graph calculus has been applied in \cite{CLS} to study cluster algebras from unpunctured surfaces with exactly one marked point and to show that the cluster algebra coincides with the upper cluster algebra in this case.
Snake graphs have also been used as a conceptual tool in the study of extension spaces in cluster categories and module categories of Jacobian algebras associated to surfaces without punctures in \cite{CaSc}. Perfect matchings of certain graphs have also been used in \cite{MSc} to give expansion formulas in the cluster algebra structure of the homogenous coordinate ring of a Grassmannian.

The article is organized as follows.  In section~\ref{sect 2}, we recall basic definitions pertaining to abstract snake and band graphs, and we introduce the notion of crossing and self-crossing band graphs. We construct the resolutions of these  crossings and self-crossings in section~\ref{section3}. In section~\ref{section 4}, we prove the existence of  a bijection between the set of perfect matchings of a crossing snake or band graphs and the set of perfect matchings of its resolution, see Theorem \ref{thmbijection}. In section \ref{sect surfaces}, we review the definition of labeled snake graphs arising from cluster algebras of surfaces, as well as the construction of the canonical bases of \cite{MSW2}.
We then show in section~\ref{sect 5} that, in the case where the snake or band graphs are actually coming from 
arcs and loops in a surface, the resolutions of the graphs correspond exactly to the smoothings of the curves.
 We also show that the corresponding skein relation holds in the cluster algebra.
In section \ref{sect 7}, we introduce various rings, that we call snake rings, and study their properties as well as their relation to cluster algebras. We also determine which snake and band graphs arise from unpunctured and from punctured surfaces, see Theorem \ref{lem unpunctured}.

 \section{Abstract snake graphs and abstract band graphs}
\label{sect 2}

 Abstract snake graphs and abstract band graphs have been introduced in \cite{CS,CS2} motivated by the snake graphs and band graphs appearing in the combinatorial formulas for elements in cluster algebras of surface type  in \cite{Propp,MS,MSW,MSW2}.

The construction of abstract snake graphs and band graphs is completely detached from triangulated surfaces. Our goal is to study these objects in a combinatorial way. 
We shall simply say snake graphs and band graphs, since we shall always mean \emph{abstract} snake graphs and \emph{abstract} band graphs. 
 
In this section, we recall the constructions of \cite{CS, CS2} pertaining to the snake graphs and band graphs and introduce the notions of overlaps and crossing overlaps for band graphs. Throughout
we fix the standard orthonormal basis of the plane.

\subsection{Snake graphs} A {\bf tile} $G$ is a square in the plane whose sides are parallel or orthogonal  to the elements in    the  fixed basis. All tiles considered will have the same side length.
\begin{center}
  {\small
\begingroup%
  \makeatletter%
  \providecommand\color[2][]{%
    \errmessage{(Inkscape) Color is used for the text in Inkscape, but the package 'color.sty' is not loaded}%
    \renewcommand\color[2][]{}%
  }%
  \providecommand\transparent[1]{%
    \errmessage{(Inkscape) Transparency is used (non-zero) for the text in Inkscape, but the package 'transparent.sty' is not loaded}%
    \renewcommand\transparent[1]{}%
  }%
  \providecommand\rotatebox[2]{#2}%
  \ifx\svgwidth\undefined%
    \setlength{\unitlength}{68.59006958bp}%
    \ifx\svgscale\undefined%
      \relax%
    \else%
      \setlength{\unitlength}{\unitlength * \real{\svgscale}}%
    \fi%
  \else%
    \setlength{\unitlength}{\svgwidth}%
  \fi%
  \global\let\svgwidth\undefined%
  \global\let\svgscale\undefined%
  \makeatother%
  \begin{picture}(1,0.70549664)%
    \put(0,0){\includegraphics[width=\unitlength]{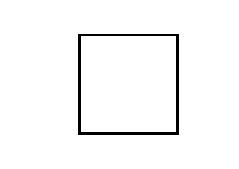}}%
    \put(0.49750979,0.31061177){\color[rgb]{0,0,0}\makebox(0,0)[lb]{\smash{$G$}}}%
    \put(0.00429408,0.31029854){\color[rgb]{0,0,0}\makebox(0,0)[lb]{\smash{West}}}%
    \put(0.79324631,0.31021311){\color[rgb]{0,0,0}\makebox(0,0)[lb]{\smash{East}}}%
    \put(0.36461283,0.61868105){\color[rgb]{0,0,0}\makebox(0,0)[lb]{\smash{North}}}%
    \put(0.36019916,0.00759722){\color[rgb]{0,0,0}\makebox(0,0)[lb]{\smash{South}}}%
  \end{picture}%
\endgroup%
}
\end{center}
We consider a tile $G$ as  a graph with four vertices and four edges in the obvious way. A {\em snake graph} $\calg$ is a connected planar graph consisting of a finite sequence of tiles $ \gi 12d$ with $d \geq 1,$ such that
$G_i$ and $G_{i+1}$ share exactly one edge $e_i$ and this edge is either the north edge of $G_i$ and the south edge of $G_{i+1}$ or the east edge of $G_i$ and the west edge of $G_{i+1}$,  for each $i=1,\dots,d-1$.
An example is given in Figure \ref{signfigure}.
\begin{figure}
\begin{center}
  {\tiny \scalebox{0.9}{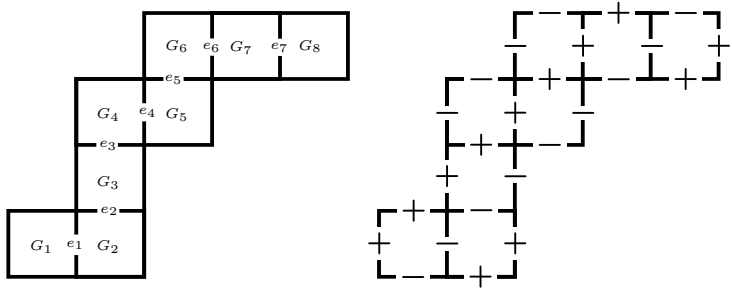}}
 \caption{A snake graph with 8 tiles and 7 interior edges (left);
 a sign function on the same snake graph (right)} 
 \label{signfigure}
\end{center}
\end{figure}
The graph consisting of two vertices and one edge joining them is also considered a snake graph.

 \begin{remark} It follows from the definition that if $\calg$ is a snake graph with tiles $G_1,\ldots,G_d$ then
\begin{itemize}
 \item[(i)] $G_i $ and $G_j$ have no edge in common whenever $|i-j| \geq 2.$
 \item[(ii)] $G_i $ and $G_j$ are disjoint whenever $|i-j| \geq 3.$
\end{itemize}  
\end{remark}

 We sometimes use the notation $\calg =(\gi 12d)$ for the snake graph and $\calg [i, i+t] = (\gi i{i+1}{i+t})$ for the subgraph of $\calg$ consisting of the tiles $\gi i{i+1}{i+t}.$
 {One may think of this subgraph as a closed interval inside $\calg$.}

The $d-1$ edges $e_1,e_2, \dots, e_{d-1}$ which are contained in two tiles are called {\em interior edges} of $\calg$ and the other edges are called {\em boundary edges.}  
Denote by $\Int(\calg)=\{e_1,e_2,\ldots,e_{d-1}\}$  the set of interior edges of $\calg$.
We will always use the natural ordering of the set of interior edges, so that $e_i$ is the edge shared by the tiles $G_i$ and $G_{i+1}$.

We denote by  $\calgSW$ the 2 element set containing the south and the west edge of the first tile of $\calg$ and by $\calgNE$ the 2 element set containing the north and the east edge of the last tile of $\calg$.  If $\calg$ is a single edge, we let $\calgSW=\emptyset$ and $\calgNE=\emptyset$.
 It will occasionally be convenient to extend the ordering on interior edges to a partial ordering on the slightly larger set $\calgSW\cup\Int(\calg)\cup\calgNE$ by specifying that the edges in $\calgSW$ are less than all interior edges and the edges in $\calgNE$ are greater than all interior edges.

{If $i\ge 2$ and $i+t\le d-1$, the notation $\calg(i,i+t)$ means $\calg[i,i+t]\setminus\{e_{i-1},e_{i+t}\}$.  One may think of this subgraph as a open interval inside $\calg$}.

If  $\calg=(\gi 12d)$ is a snake graph and $e_i$ is the interior edge shared by the tiles $G_i$ and $G_{i+1}$, we define the snake graph $\calg\setminus \textup{pred} (e_i)$ to be the graph obtained from $\calg$ by removing the vertices and edges that are predecessors of  $e_i$, more precisely,
\[ \calg\setminus \pred (e_i) =\calg[i+1 , d].\]
It will be convenient to extend this construction to the edges $e\in \calgNE $, thus 
\[ \calg\setminus \pred (e) =\{e\} \quad\textup{if } e\in \calgNE.\]  
Similarly, we define the snake graph $\calg\setminus \textup{succ} (e_i)$ to be the graph obtained from $\calg$ by removing the vertices and edges that are successors of  $e_i$, more precisely,
\[ \calg\setminus \suc (e_i) =\calg[1 , i].\]
It will be convenient to extend this construction to the edges $e\in \calgSW $, thus 
\[ \calg\setminus \suc (e) =\{e\} \quad\textup{if } e\in \calgSW.\] 

A snake graph $\calg$ is called {\em straight} if all its tiles lie in one column or one row, and a snake graph is called {\em zigzag} if no three consecutive tiles are straight.
 We say that two snake graphs are \emph{isomorphic} if they are isomorphic as graphs.
\subsection{Sign function} 

A {\em sign function} $f$ on a snake graph $\calg$ is a map $f$ from the set of edges of $\calg$ to $\{ +,- \}$ such that on every tile in $\calg$ the north and the west edge have the same sign, the south and the east edge have the same sign and the sign on the north edge is opposite to the sign on the south edge. See Figure \ref{signfigure} for an example.
%
%

Note that on every snake graph,  there are exactly two sign functions.

We remark that if we associate a fixed sign function $f$ on a subgraph $\calg'$ of a snake graph $\calg,$ we obtain a sign function on the snake graph $\calg$ by successively applying the rules of the sign function. We call this sign function the {\em induced sign function of $\calg'$\! on $\calg$ by $f.$}

A snake graph is  determined up to symmetry by its sequence of tiles together with a sign function.

 Given a snake graph $\calg=(G_1,G_2,\ldots,G_d)$ with sign function $f$, we define the \emph{opposite} snake graph $\ocalg$ as follows; it has the opposite sequence of tiles $\ocalg=(G_d,G_{d-1},\ldots,G_1)$ and its sign function $\bar f$ assigns to the interior edge shared by the tiles $G_\ell$ and $G_{\ell-1}$ the  sign that the function $f$ assigns to the edge shared by the tiles $G_{\ell-1}$ and $G_\ell$.
 
Although the definition of sign function may not seem natural at first sight, it has a geometric meaning which is explained in Remark \ref{rem sign}.

\subsection{Band graphs}\label{sect band} {Band graphs are obtained from snake graphs by identifying a boundary edge of the first tile with a boundary edge of the last tile, {where both edges have the same sign}. We use the notation $\calg^\circ$ for general band graphs, indicating their circular shape, and we also use the notation $\calg^b$ {to indicate} that the band graph is constructed by glueing a snake graph $\calg$ along an edge $b$.} 

More precisely, to define a band graph $\calg^\circ$, we start with an abstract snake graph $\calg=(\gi 12d)$ with $d\ge 1$, and fix a sign function on $\calg$. Denote by $x$ the southwest vertex of $G_1$, let $b\in\calgSW$ the south edge (respectively the west edge) of $G_1$, and let $y$ denote the other endpoint of $b$, see Figure \ref{figband}. Let $b'$ be the unique edge in $\calgNE$ that has the same sign as  $b$, and let $y'$ be the northeast vertex of $G_d$ and $x'$ the other endpoint of $b'$.
\begin{wrapfigure}{l}{0.15\textwidth}
  \begin{center}
    {\scalebox{0.15}{\includegraphics{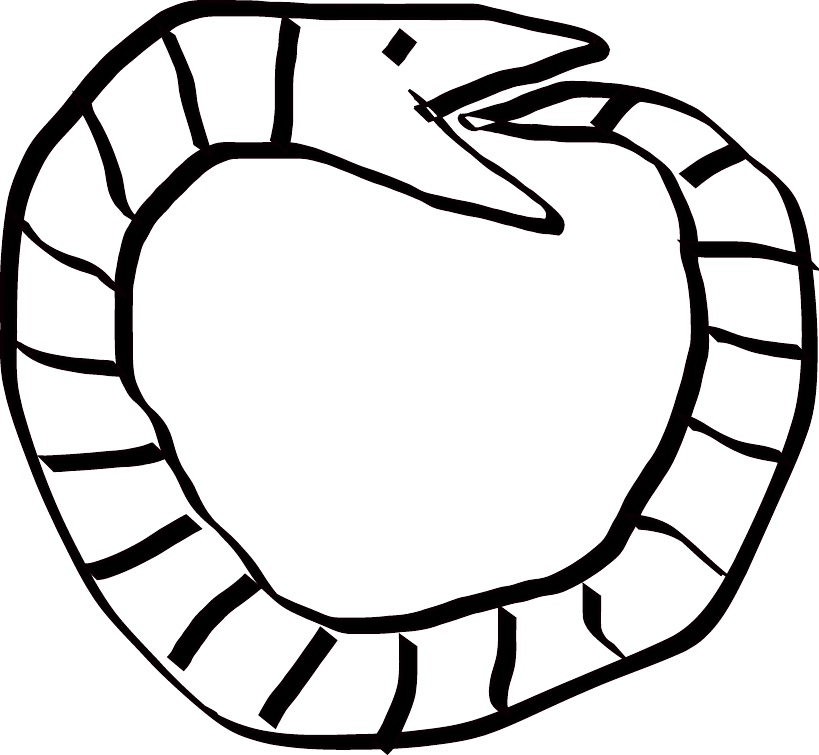}}}
  \end{center}
\end{wrapfigure}   
Let $\calg^b$ denote the graph obtained from $\calg $ by identifying the edge $b$ with the edge $b'$ and the vertex $x$ with $x'$ and $y$ with $y'$. The graph $\calg^b$ is called a \emph{band graph} or \emph{ouroboros}\
\footnote{Ouroboros: a snake devouring its tail.}. Note that non-isomorphic snake graphs can give rise to isomorphic band graphs. See Figure \ref{figband} for an example.

\begin{figure}
\begin{center}
  {\Large \scalebox{.7}{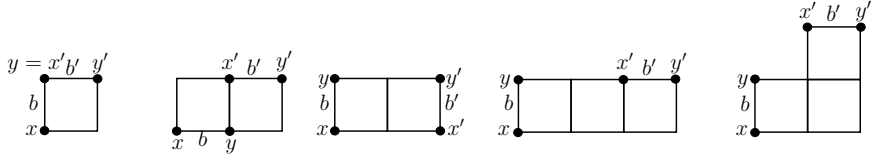}}
 \caption{Examples of small band graphs; the two band graphs with 3 tiles are isomorphic.}
 \label{figband}
\end{center}
\end{figure}

The interior edges of the band graph $\calg^b$ are by definition the interior edges of $\calg$ plus the glueing edge $b=b'$. {As usual}, we denote by $e_i$ the interior edge shared by $G_i$ and $G_{i+1}$ and we denote the glueing edge by $e_0$ or $e_d$. 
Given a  band graph $\calg^\circ$ with an  interior edge $e$, we denote by $\calg^\circ_e $ the snake graph obtained by cutting $\calg^\circ$ along the edge $e$. 
Note that  $(\calg^\circ_e)^e = \calg^\circ$, for all band graphs $\calg^\circ$ and that $({\calg^b})_b=\calg$, for all snake graphs $\calg$.
Moreover, if $\calg^\circ$ has $d$ interior  edges $e_1,e_2,\ldots,e_d$ then the isomorphism classes of the $d$ snake graphs  $\calg^\circ_{e_i}$, $i=1,\dots,d $,
are not necessarily distinct.

\begin{defn}\label{def group}
 Let  $\mathcal{R}$  denote the free abelian group generated by all isomorphism classes of {finite disjoint} unions of snake graphs and band graphs.  {If $\calg$ is a snake graph, we also denote its class in $\mathcal{R}$ by $\calg$, and we say that $\calg\in \mathcal{R}$ is a {\em positive} snake graph  and that its inverse  $-\calg\in \mathcal{R}$ is a {\em negative} snake graph.}
 \end{defn}

\subsection{Labeled snake and band graphs}
 A \emph{labeled} snake graph is a snake graph in which each edge and each tile carries a label or weight. 
For example, for snake graphs from cluster algebras of surface type, these labels are  cluster variables.

Formally, a labeled snake graph is a snake graph $\calg$ together with two functions \[\{\textup {tiles in }\calg\}\to \calf \qquad\textup{ and }\qquad \{\textup {edges in }\calg\}\to \calf,\] 
where $\calf$  is a  set.
Labeled band graphs are defined in the same way. 

 Let  $\mathcal{LR_\calf}$  denote the free abelian group generated by all isomorphism classes of unions of labeled snake graphs and labeled band graphs with labels in $\calf$.

 \subsection{Overlaps and self-overlaps involving band graphs}\label{section2point5}  Overlaps of snake graphs have been introduced in \cite{CS} and self-overlaps of snake graphs in \cite{CS2}. We now extend these notions to band graphs.
Let $\calg_2^\circ$ be a band graph. Recall that there exists  a snake graph  $\calg_2=(\gi 12d)$ with sign function $f$ such that  $\calg_2^\circ=\calg_2^b$ is obtained from  $\calg_2$ by identifying an edge $b\in {}_{SW}\calg_2$ with the unique edge $b'\in\calg_2^{N\!E}$ such that $f(b)=f(b')$.

It will be convenient to introduce the following infinite snake graph $\calg_2^\infty$, which can be thought of as the universal cover of $\calg_2^\circ$. An example of this construction is given in Figure \ref{calginfinity}.
For $i\in \mathbb{Z}$, let $\calg_{2i}$ denote a copy of $\calg_2$ with glueing edges $b_i,b_i'$, where $b_i$ corresponds to $b$ and $b_i'$ corresponds to $b'$. Define the infinite snake graph $\calg_2^\infty$ by identifying the edges $b_{i}$ and $b_{i+1}'$ for all $i$ in such a way, that $\calg_2^\infty$ is a snake graph locally isomorphic to $\calg_2^\circ$.
Note that if the number of tiles in $\calg_2^\circ$ is even then all $\calg_{2i}$ can be glued without changing their orientation, and if the number of tiles is odd, then every other $\calg_{2i}$ must be flipped over in order to glue. 

\begin{figure}
\begin{center}
\scalebox{0.7}{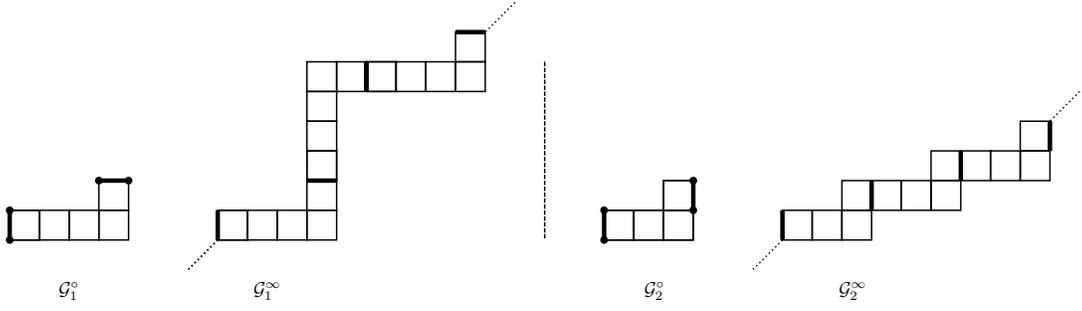}
\caption{Two examples of the infinite snake graph $\calg^\infty$. On the left the number of tiles in $\band_1$ is odd,  so that in $\calg_1^\infty $ every other copy of $\band_1$ is flipped over. On the right, the number of tiles in $\band_2$ is even, so in $\calg_2^\infty$ all copies of $\band_2$ have the same orientation.}
\label{calginfinity}
\end{center}
\end{figure}

Let   $\calg_1$ be a snake graph, and let $\calg_2^\circ$ be a band graph. We say that   $\calg_1$  and $\calg_2^\circ$ have an {\em overlap} $\calg$ if $\calg$ is a snake graph consisting of at least one tile and there exist an interior edge $b\in \Int(\calg_2^\circ)$ and two embeddings of graphs $i_1 : \calg \rightarrow \calg_1$ and $i_2: \calg \rightarrow  (\calg_2^\circ)_{b}$ which are maximal in the following sense.

\begin{itemize}
\item[(i)] If $\calg$ has at least two tiles and if there exists a snake graph $\calg'$ with an interior edge $b'\in \Int(\calg_2^\circ)$ and two embeddings  of graphs $i'_1: \calg' \to \calg_1 ,$ $i'_2: \calg' \to (\band_2)_{b'}$ such that $i_1(\calg) \subseteq i'_1(\calg')$ and $i_2(\calg) \subseteq i'_2(\calg')$ then $i_1(\calg) = i'_1(\calg')$ and $i_2(\calg) = i'_2(\calg').$
\item[(ii)] 
If $\calg$ consists of a single tile then
\begin{itemize}
 \item[(a)] $i_1(\calg)$ is the first or the last tile of $\calg_1$, or
 \item[(b)] 
the two subgraphs of $\calg_1$ and $\calg_2^\infty$ consisting of the overlap and the two adjacent tiles are either both straight or both zigzag subgraphs.
\end{itemize}
\end{itemize}

\begin{remark}
It is possible that $\calg\cong(\calg_2^\circ)_{b}$. So in order for $i_2$ to be an embedding, it is necessary to consider it as a  map 
 into the snake graph $(\calg_2^\circ)_{b}$ and not into the band graph $\calg_2^\circ$.
\end{remark}

\begin{remark}\label{remsis1}
  Without loss of generality, we may choose $b$ to be
 the interior edge of $\calg_2^\circ$ which marks the beginning of the overlap $i_2(\calg)$. In other words, we choose $b$ such that the overlap $i_2(\calg)$ in  $\calg_2$ starts with the first tile. 
\end{remark}
\smallskip

 In order to define the notion a self-overlap in one single band graph, we need to distinguish two cases depending on the direction of the embedded subgraphs as follows. Let $ \gr12d$ be a snake graph, let $\calg$ be another snake graph together with two embeddings of graphs $i_1:\calg\to\calg_1$ and $i_2:\calg\to\calg_1$ such that $i_1(\calg)=\calg_1[s,t]\ne i_2(\calg)=\calg_1[s',t']$. 
We may assume without loss of generality that the embedding $i_1$ maps the southwest vertex of  the first tile of $\calg$ to the southwest vertex of  $G_s$ in $\calg_1[s,t].$ We then say that the embeddings $i_1,i_2$ are  in the \emph{same direction} if  $i_2 $  maps the southwest vertex of  the first tile of $\calg$ to the southwest vertex of  $G_{s'}$ in $\calg_1[s',t']$, and we say that $i_1,i_2$ are  in the \emph{opposite direction} if $i_2$ maps this vertex to the northeast vertex of $G_{t'}$ in $\calg_1[s',t']$.

\begin{remark}
The notion of direction depends on the embeddings $i_1$ and $i_2$ and not only on the subgraphs $\calg_1[s,t]$ and $ \calg_1[s',t']$.
\end{remark}

We say that a band graph $\calg_1^\circ$  has a \emph{self-overlap $\calg$ in the same direction} if $\calg$ is a snake graph and  there exist  interior edges $b_1,b_2\in \Int(\calg_1^\circ)$, $b_1\ne b_2$, and two embeddings of graphs $i_1 : \calg \rightarrow (\calg_1^\circ)_{b_1},$ $i_2: \calg \rightarrow (\calg_1^\circ)_{b_2}$  in the same direction, such that $i_1(\calg)\ne i_2(\calg)$, the two sign functions on $\calg_1$  induced by $f$ are equal to each other,   and the following condition hold.
\begin{itemize}
\item[(i')]  If $\calg$ has at least two tiles and if there exists a snake graph $\calg'$ with an interior edge $b'\in \Int(\calg_1^\circ)$ and two embeddings  of graphs $i'_1: \calg' \to (\band_1)_{b_1'} ,$ $i'_2: \calg' \to (\band_1)_{b_2'}$ such that $i_1(\calg) \subseteq i'_1(\calg')$ and $i_2(\calg) \subseteq i'_2(\calg')$ then $i_1(\calg) = i'_1(\calg')$ and $i_2(\calg) = i'_2(\calg').$
\item[(ii')] 
If $\calg$ consists of a single tile then
the two subgraphs of $\calg^\infty_1$ consisting of the overlaps and the two adjacent tiles are either both straight or both zigzag subgraphs.
\item[(iii)] If the overlap is not the whole band graph then there exist $b_1,b_2,i_1,i_2$ such that the intersection $i_1(\calg)\cap i_2(\calg)$ is connected. 
\end{itemize}

\begin{remark}
 It is   possible that  the overlap is  the whole band graph, that is, $\calg\cong (\calg_1^\circ)_{b_1}\cong (\calg_1^\circ)_{b_2}$. In this case the intersection $i_1(\calg)\cap i_2(\calg)$ is never connected. 
\end{remark}

\begin{example}\label{exband2tiles}
 The band graph on the left of Figure \ref{band2tiles} has an overlap $\calg$ consisting of a single tile, with embeddings $i_1(\calg)=G_1$ and $i_2(\calg)=G_2$. This is indeed an overlap, because $\calg_1^\infty$ is straight. The band graph on the right of Figure \ref{band2tiles} also has an overlap $\calg$ consisting of a single tile, with embeddings $i_1(\calg)=G_1$ and $i_2(\calg)=G_2$. This is  an overlap, because $\calg_1^\infty$ is zigzag. 
\begin{figure}
\begin{center}
\scalebox{0.8}{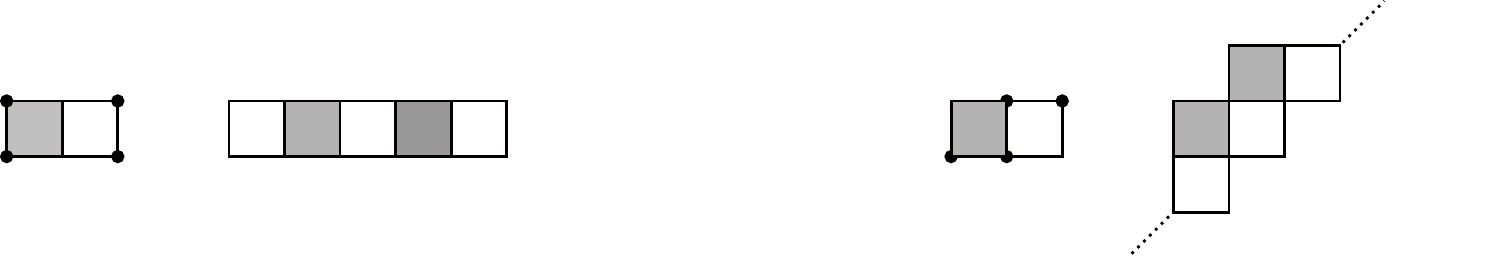}
\caption{Example \ref{exband2tiles}. }
\label{band2tiles}
\end{center}
\end{figure}
\end{example}

 We say that a band graph $\calg_1^\circ$ has a \emph{self-overlap $\calg$ in the opposite direction} if $\calg$ is a snake graph and there exist  interior edges $b_1,b_2\in \Int(\calg_1^\circ)$, $b_1\ne b_2$, and two embeddings of graphs $i_1 : \calg \rightarrow (\calg_1^\circ)_{b_1},$ $i_2: \calg \rightarrow (\calg_1^\circ)_{b_2}$ in the opposite direction such that {$i_1(\calg)$ and $ i_2(\calg)$ are disjoint and there is at least one tile between them}, the two  sign functions on $\calg_1$  induced by $f$ are opposite to each other, and  the conditions (i'), (ii') and (iii) above are satisfied.

 Now we define self-overlaps with intersections.
Let $\calg_1=(\calg_1^\circ)_{b_1}$ and label the tiles such that $\calg_1=(\gi 12{d})$. Let $s=1,t,s',t'$,    be such that $i_1(\calg)=\calg_1[1,t]$ and 
\[i_2(\calg)=\left\{\begin{array}{ll}\calg_1[s',t'] &\textup{if $i_2(\calg) $ is a connected subgraph of $\calg_1$}\\
\calg_1[s',d]\cup\calg_1[1,t'] &\textup{otherwise.}\end{array}\right.\]
The self-overlap is said to have an \emph{intersection} if $i_1(\calg)\cap i_2(\calg)$ contains at least one edge; in other words, $s'\le t+1$ or  $i_2(\calg)$ is not connected. In this case, we say $\calg_1$ has an \emph{intersecting self-overlap}. 

\begin{remark} We can always arrange the setup such that $s'$ is at most half of the total number of tiles in $\calg_1$ simply by interchanging the roles of $i_1$ and $i_2 $ if necessary.  Then because of condition (iii) above, the intersection of the overlaps always is a connected subgraph of  $\calg_1$ unless the overlap is the whole band graph. 
\end{remark}

\begin{remark}
 If  the overlap is the whole band graph then the intersection is also the whole band graph.\end{remark}
{For labeled snake graphs, we define overlaps by adding the  requirement that the embeddings $i_1$ and $i_2$ are label preserving.}

\subsection{Crossing overlaps}  The notion of crossing overlaps for snake graphs was introduced in \cite{CS}. Here we extend this notion to overlaps that involve band graphs. There are two cases to consider, namely a crossing between a snake graph and a band graph, and a crossing between two band graphs.
\subsubsection{Snake graph and band graph} Let $\calg_1$ be a snake graph with at least one tile and $\band_2$ a band  graph with overlap $\calg$ and embeddings $i_1(\calg)\subset \calg_1$ and $i_2(\calg)\subset\calg_2$, where $\calg_2=(\calg_2^\circ)_b$  is  the snake graph obtained from $\calg_2^\circ$ by cutting at $b$, for some interior edge $b\in\Int(\calg_2^\circ)$. 
 By Remark \ref{remsis1}, we may assume that $i_2(\calg)$ in $\calg_2$ starts with the first tile. Thus {$s'=1$}. Let $t'$ be such that $i_2(\calg)=\calg_2[1,t']$.
Label $b\in {}_{SW}\calg_2$ and $b'\in \calg_2^{N\!E}$ the edges that  correspond to the interior edge $b$ of $\calg_2^\circ$.
Let $e_i$ (respectively $e'_i$) denote the interior edges of $\calg_1$ (respectively $\band_2$). 

If the overlap is not the whole band graph, that is if $i_2(\calg)\ne \calg_2$, we let $s\le t$ be such that $i_1(\calg)=\calg_1[s,t]$. On the other hand, if the overlap is equal to the whole band graph, that is if $i_2(\calg)=\calg_2$, we let $\overline{i_1(\calg)}$ be the largest subgraph of $\calg_1$ which contains $i_1(\calg)$ and which is isomorphic to a subgraph of $\calg_2^\infty$, and we 
let $s\le t$ be such that $\overline{i_1(\calg)}=\calg_1[s,t]$.
Let $f$ be a sign function on $\calg.$ Then $f$ induces a sign function $f_1$ on $\calg_1$ and $f_2$
 on $\band_2.$

\begin{defn} \label{crossing} We say that $\calg_1$ and $\band_2$ {\em cross in} $\calg$ if one or both of the following conditions hold. Recall that since $s'=1$, we have $e'_{s'-1}=e'_0=b$.

\begin{itemize}
 \item[(i)] \label{i}
\[ \begin{array}{lrclll}
 &f_1(e_{s-1})& =& -f_1(e_t) & \mbox{ and } &s>1, t<d ;  \\
 \mbox{ or}\\
&  f_2 (e'_{s'-1})&=&-f_2(e'_{t'}) & \mbox{ and } & t'<d'  ;
\end{array}\]
 \item[(ii)]  \[\begin{array}{lrcllcc}
 &f_1(e_{t}) &=& f_2(e'_{s'-1})& \mbox{ and } &s=1, t<d,  t'=d'  \\
  \mbox{ or}\\
&f_1 (e_{s-1})&=&f_2(e'_{t'}) & \mbox{ and } &s>1 , t=d, t'<d' &
\end{array}\] 
\end{itemize}
\end{defn} 

Examples of crossing overlaps are given in {Figure \ref{straight}}.
\begin{remark}
  In the definition above, we distinguish cases (i) and (ii) because they are conceptually different. In case (i), we compare signs of edges in the same graph, and in case (ii) we compare signs of edges from different graphs.
  \end{remark}

\begin{remark}\label{rem2.8}
 If $s=1$ and $t=d$ then the only way we can have a crossing is when $t'<d'$. In particular, if the overlap is equal to both $\calg_1$ and $\calg_2$ then there is no crossing.
\end{remark}

\subsubsection{Two band graphs} Let $\band_1,\band_2$ be two band  graphs, with overlap $\calg$ and embeddings $i_1(\calg)\subset \calg_1$ and $i_2(\calg)\subset\calg_2$, where $\calg_i=(\calg_i^\circ)_{b_i}$  is  the snake obtained from $\calg_i^\circ$ by cutting at some interior edge $b_i\in\Int(\calg_i^\circ)$. Again, we may choose $b_2$ to be such that the overlap in  $\calg_2$ starts with the first tile. 
Let $t'$ be such that  $i_2(\calg)=\calg_2[1,t']$.

If the two band graphs are isomorphic and the overlap is the whole band graph, thus $\calg\cong\calg_1$ and $\band_1\cong\band_2$ then we define it to be non-crossing.

Suppose that at least one of the band graphs, say $\band_1$, is strictly bigger than the overlap. Thus $i_1(\calg)\subsetneq\calg_1$.
Let $e_i$ (respectively $e'_i$)  denote  the interior edges of $\band_1$ (respectively $\band_2$). 

If  $i_2(\calg)\ne \calg_2$, we let $b_1, t$ be such that $i_1(\calg)=\calg_1[1,t]$,  where $\calg_1=(\band_1)_{b_1}$. On the other hand, if the overlap is equal to the whole band graph, that is if $i_2(\calg)=\calg_2$, we let $\overline{i_1(\calg)}$ be the largest subgraph of $\band_1$ which contains $i_1(\calg)$ and which is isomorphic to a subgraph of $\calg_2^\infty$, and we 
let $ b_1, t$ be such that $\overline{i_1(\calg)}=\calg_1[1,t]$, where $\calg_1=(\band_1)_{b_1}$.
Let $f$ be a sign function on $\calg.$ Then $f$ induces a sign function $f_1$ on $\band_1$ and $f_2$
 on $\band_2.$ 
 
\begin{defn} \label{defcrossbands} 
We say that $\calg_1^\circ$ and $\calg_2^\circ$ {\em cross in} $\calg$ if one or both of the following conditions hold. Recall that since $s=s'=1$, we have $e_{s-1}=e_0=b_1$ and $e'_{s'-1}=e'_0=b_2$.
\begin{itemize}
 \item[(i)] 
\[ \begin{array}{lrclll}
&  f_2 (e'_{s'-1})&=&-f_2(e'_{t'}) & \mbox{ and } & t'<d' ; 
\end{array}\]

 \item[(ii)]  \[\begin{array}{lrcllcc}
 &f_1(e_{t}) &=& f_2(e'_{s'-1})& \mbox{ and } & t<d,\, t'=d' . \\
 \end{array}\] 
\end{itemize}
\end{defn} 
If $t=d$ and $t'=d'$ then the overlap is equal to both band graphs and there is no crossing.

\subsection{Crossing self-overlaps}
 For snake graphs, the notion of crossing self-overlaps has been introduced in \cite{CS2}. Here we extend this notion to band graphs.

 \subsubsection{Selfcrossing band graphs}

Let $\calg_1^\circ=\calg_1^b$ be a band graph with self-overlap $i_1(\calg)=\calg_1[1,t]$ and $i_2(\calg)=\calg_1[s',t']$. Let $f$ be a sign function on $\calg_1^\circ$.

\begin{defn}\label{def self-crossing band}
\begin{itemize}
\item[(a)] Suppose that $\calg$ is not the whole band graph $\calg_1^\circ$.
 We say that $\calg_1^\circ$ has a \emph{self-crossing} (or \emph{self-crosses}) in $\calg$ if  the following two conditions hold
\begin{itemize}
 \item[(i)] \label{2.4i}
\[ \begin{array}{lrcllclrclll}

  f (e_{s'-1})=-f(e_{t'}) & \mbox{ if }  t'<d;  \ and
\end{array}\]
 \item[(ii)] 
  \[\begin{array}{lrcllcc}
 &f(e_{t}) &=& f(e_{s'-1}).&
\end{array}\] 
\end{itemize}
\item[(b)] Suppose that $\calg$ is the whole band graph $\calg_1^\circ$.  We say that $\calg_1^\circ$ has a \emph{self-crossing} (or \emph{self-crosses}) in $\calg$ if 
\[f(e_t)=f(e_{s'-1}).\]

\end{itemize}

\end{defn}

 In the examples considered earlier in Figure \ref{band2tiles}, the self-overlap is self-crossing in the example on the left, but it is not crossing in the example on the right.
  
To illustrate part (b) of the above definition, we give an example in Figure \ref{straight} of the same band graph with two different overlaps, both equal to the whole band graph. The edge $e_t$ is the glueing edge in both examples. In the example on the left, the edge $e_{s'-1}$ is shared by the third and the fourth tile and its sign is opposite to the one of $e_t$. Thus this overlap is non-crossing. 
 In the example on the right, the edge $e_{s'-1}$ is shared by the second and the third tile and its sign is equal to the one of $e_t$. Thus this overlap is crossing.

\begin{figure}
\begin{center}
\scalebox{1}{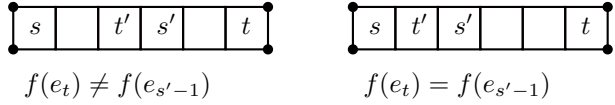}
\caption{Two examples of overlaps  $\band[s,t]\cong\band[s',t']$ isomorphic to the whole band graph. The overlap on the left is not crossing and the overlap on the right is crossing. }
\label{straight}
\end{center}
\end{figure}

\section{Resolutions}\label{section3}
 In this section, we define the resolutions of crossings and self-crosings involving band graphs. For the resolutions involving only snake graphs, we refer to \cite{CS2}. The resolution of the crossings consists of a sum of two elements in the group $\mathcal{R}$. In section \ref{section 4}, we show that there is a bijection between the sets of perfect matchings of the original graphs and those obtained by the resolutions, and in section \ref{sect 7}, we introduce a ring structure on $\mathcal{R}$ and consider the ideal generated by all resolutions. Thus in the resulting quotient ring a crossing pair of snake and/or band graphs as well as a self-crossing snake or band graph will be equal to its resolution. We shall see that this quotient ring is strongly related to cluster algebras from surfaces.

Throughout this section, we will often use the notation $\calg\cup_e\calg'$ to indicate that the two graphs $\calg$ and $\calg'$ are glued along an edge $e$.
\subsection{Resolutions for crossings between a snake graph and a band graph} \label{sect 3.1}
Let $\calg_1=(G_1,G_2,\ldots,G_d)$ be a snake graph and $\band_2$ a band graph, such that $\calg_1$ and $\band_2$ have a crossing overlap $ i_1(\calg)=\calg_1[s,t]$ and $i_2(\calg)\subset\band_2$.  According to Remark \ref{remsis1}, we may choose $b$ such that the overlap $i_2(\calg)$ in  $\calg_2=(\band_2)_b$ starts with the first tile. 
Let $t'$ be such that $i_2(\calg)=\calg_2[1,t']$.
Label $b\in {}_{SW}\calg_2$ and $b'\in \calg_2^{N\!E}$ the edges that  correspond to the interior edge $b$ of $\calg_2^\circ$. 
We define two snake graphs $\calg_{34}$ and $\calg_{56}$ as follows.
Let \[\calg_{34}=\calg_1[1,s-1]\cup_{e_{s-1}}\calg_2 \cup_{b'}\calg_1[s,d],\]
where the first two subgraphs are glued along the edges $e_{s-1}$ of $\calg_1 $ and the unique edge in $\calgSW_2\setminus\{b\}$, and the last two subgraphs    are glued along the edges $b'\in \calgNE_2$ and the edge in ${}_{SW}G_s$ corresponding to $b$ under $i_1$.

The snake graph $\calg_{56}$ will be defined as a subgraph of the following snake graph.
\[\calg_{56}'=\calg_1[1,s-1]\cup \overline\calg_2[d',t'+1]\cup\calg_1[t+1,d],\]
where the first two subgraphs are glued along the unique boundary edge in $G_{s-1}^{N\!E}$ and the unique edge in $\calgNE_2$ that is different from $b'$,
and the last two subgraphs are glued along the unique boundary edge in ${}_{SW}G'_{t'+1}$ and the unique boundary edge in ${}_{SW}G_{t+1}$.

In other words, the sequence of tiles of the snake graph $\calg'_{56}$ is 
\[(G_1,G_2,\ldots,G_{s-1}, G_d',G_{d-1}',\ldots,G'_{t+1}, G_{t+1}, G_{t+2},\ldots,G_{d}).\]
Let $f_{56}$ be the sign function on $\calg_{56}'$ induced by $f$ on $\calg_2^\circ$.
Then we define $\calg_{56}$ as follows
\begin{enumerate}
\item if $s\ne 1 $ and $t\ne d$, then
\[ \calg_{56} = \calg_{56}' , \] 

\item  if $s= 1 $ and $t\ne d$  (see Figure \ref{selfcrossingA6} for an example), then
\[ \calg_{56} = \calg_{56}'\setminus\pred(e),\] \textup{where $e$ is the first edge in $\Int(\overline{\calg}_2[d',t'+1])\cup ( \overline{\calg}_2[d',t'+1])^{N\!E} $ 
 such that $f_{56}(e)=f_{56}(b')$, } 

\item  if $s\ne 1 $ and $t= d$, then  
\[ \calg_{56} =  \calg_{56}'\setminus \suc(e') ,\] 
where $e'$ is the last edge in 
$\Int(\overline{\calg}_2[d',t'+1])\cup {}_{SW}( \overline{\calg}_2[d',t'+1]) $ such that $f_{56}(e')=f_{56}(e'_{t'})$,

\item  if $s=1 $ and $t= d$ (see Figure \ref{selfcrossingA5} for an example), then 
\begin{enumerate}
\item let 
\[\calg_{56}= 
(\calg_{56}'\setminus \pred(e))\setminus \suc(e')
,\]
where $e$ is the first edge in $\Int(\calg_{56}')\cup\calg_{56}'^{N\!E} $ such that $f_{56}(e)=f_{56}(b')$
and $e'$ is the last edge in $\Int(\calg_{56}'\setminus\pred(e))\cup{}_{SW}(\calg_{56}'\setminus\pred(e)) $ such that $f_{56}(e')=f_{56}(e'_{t'})$, if such an $e'$ exists,

\item
if $e'$ does not exist, let 
\[\calg_{56}= 
(\calg_{56}'\setminus \suc(e'))\setminus \pred(e)
,\]
where $e'$  is the last edge in $\Int(\calg_{56}')\cup{}_{SW}\calg_{56}' $ such that $f_{56}(e')=f_{56}(e'_{t'})$
and $e$ is the first edge in $\Int(\calg_{56}'\setminus\suc(e'))\cup(\calg_{56}'\setminus\suc(e'))^{N\!E} $ such that $f_{56}(e)=f_{56}(b')$, if such an $e$ exists,

\item  if $\calg_{56}'$ is a single tile, let
\[\calg_{56}= \textup{unique boundary edge of } {}_{SW}\calg_2.\] 
\end{enumerate}
\end{enumerate}

\begin{figure}
\begin{center}
\scalebox{0.8}{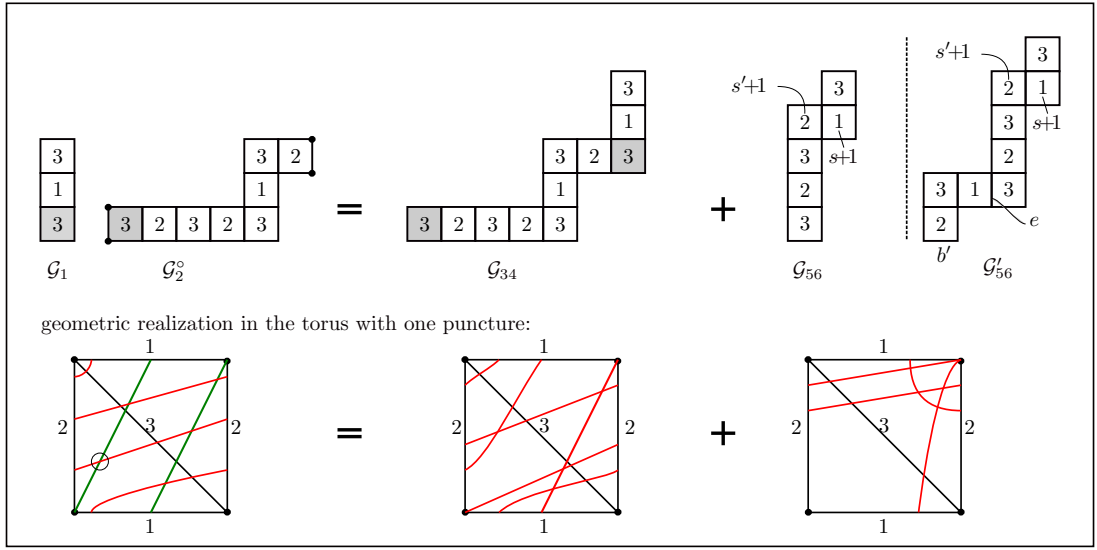}
\caption{ An example of a resolution of a crossing between a snake graph and a band graph with $s=t=1$ and $t<d$. The labels in the tiles correspond to the labels of the arcs of the triangulation in the geometric realisation. The overlap (shaded) consists of the first tile in each of the two graphs $\calg_1$ and $\band_2$. The crossing point in the geometric realisation is circled.}
\label{selfcrossingA6}
\end{center}
\end{figure}

\begin{figure}
\begin{center}
\scalebox{0.87}{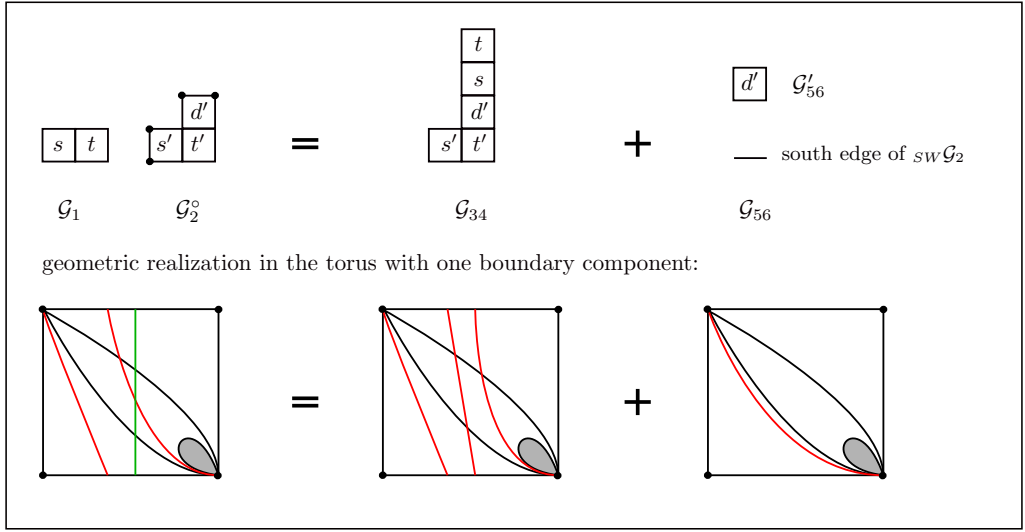}
\caption{ An example of a resolution of a crossing between a snake graph and a band graph with $s=1$ and $t=d$.}
\label{selfcrossingA5}
\end{center}
\end{figure}

\begin{remark}  The cases 4(a),(b),(c) cover all possibilities. Indeed, on the one hand, $\calg_{56}'$ is at least one tile, since otherwise $i_1(\calg)=\calg_1$ and $i_2(\calg)=\calg_2$ which would not be a crossing overlap  by Remark \ref{rem2.8}, and on the other hand,  $\calg_{56}$ cannot have more than one tile, because then the interior edge between the two tiles would have the same sign as the edge $e_{s-1}$ or the edge $e_t$ and therefore one of $e$ or $e'$ existed in the earlier cases. 
\end{remark}
\begin{remark}
  In the case 4(c),
we could also define $\calg_{56}$ to be the unique boundary edge of $(G'_{t'})^{N\!E}$.  This makes sense because if the band graph is a labeled band graph coming from a surface without punctures, then both edges would have the same label, which would also be the label of the tile $G'_{d'}$.
\end{remark}
We have the following special case when the overlap is the whole band graph.
\begin{prop}
 If $i_2(\calg)=\calg_2$ then $\calg_{56}=\calg_1\setminus i_1(\calg)$.
\end{prop}

\begin{proof}
If $i_2(\calg)=\calg_2$ then by definition $\calg_{56}'$ is equal to $\calg_1[1,s-1]\cup\calg_1[t+1,d]=\calg_1\setminus i_1(\calg).$ 
 So it suffices to show that $\calg_{56}=\calg_{56}'$, which amounts to  showing that $s\ne 1$ and $t\ne d$.
Suppose $s=1$. Since $i_2(\calg)=\calg_2$, we also have $t'=d'$. Since the overlap is crossing,  Definition \ref{crossing} implies that $f_1(e_{\overline{t}})=f_2(e'_{s'-1})=f_2(b')$, where $\overline{t}$ is such that $\overline{i_1(\calg)}=\calg_1[1,\overline{t}]$ is the largest subgraph of $\calg_1$ which contains $i_1(\calg)$ and which is isomorphic to a subgraph of $\calg_2^\infty$. The maximality of $\overline{i_1(\calg)}$ implies that $f_1(e_{\overline{t}})\ne f_2(b')$, a contradiction.
Similarly, $t$ cannot be equal to $d$, and thus $\calg_{56}=\calg_{56}'$ and we are done.
\end{proof}

\begin{defn} 
 In the above situation, we say that the element $\calg_{34} + \calg_{56} \in\mathcal{R}$ is the {\em resolution of the crossing} of $\calg_1$ and $\calg_2^\circ$ in $\calg$ and we denote it by 
 $\res_{\calg}(\calg_1,
\band_2).$ 
\end{defn}

\subsection{Resolutions for crossings between two band graphs}\label{sect 3.2}
Let $\band_1$  and $\band_2$ be band graphs which have a crossing overlap $ i_1(\calg)\subset\band_1$ and $i_2(\calg)\subset\band_2$. Denote by $a$ the interior edge of $\calg_1^\circ$ and by $b$ the interior edge of $\calg_2^\circ$ which mark the beginning of the   overlaps, and let  $\calg_1=(\calg_1^\circ)_a$  and $\calg_2=(\calg_2^\circ)_b$ be the snake graphs obtained  by cutting along $a$ and $b$. In other words, we choose $a$ and $b$ such that the overlaps $i_1(\calg)$ in  $\calg_1$ and $i_2(\calg)$ in  $\calg_2$ start with the first tile of $\calg_1$ and $\calg_2$,  respectively. 

As usual, we denote the tiles of these snake graphs as follows. Let $\calg_1=(G_1,G_2,\ldots,G_d)$ and let $\calg_2=(G'_1,G'_2,\ldots,G'_{d'})$.

Let $t, t'$ be such that $i_1(\calg)=\calg_1[1,t]$ and $i_2(\calg)=\calg_2[1,t']$.
Label $a\in {}_{SW}\calg_1$ and $a'\in \calg_1^{N\!E}$ the edges that  correspond to the interior edge $a$ of $\calg_1^\circ$. 
Similarly label $b\in {}_{SW}\calg_2$ and $b'\in \calg_2^{N\!E}$ the edges that  correspond to the interior edge $b$ of $\calg_2^\circ$. 
 We define two band graphs $\band_{34}$ and $\band_{56}$ as follows. An example is given in Figure \ref{torus3}.

\begin{figure}
\begin{center}
\scalebox{0.8}{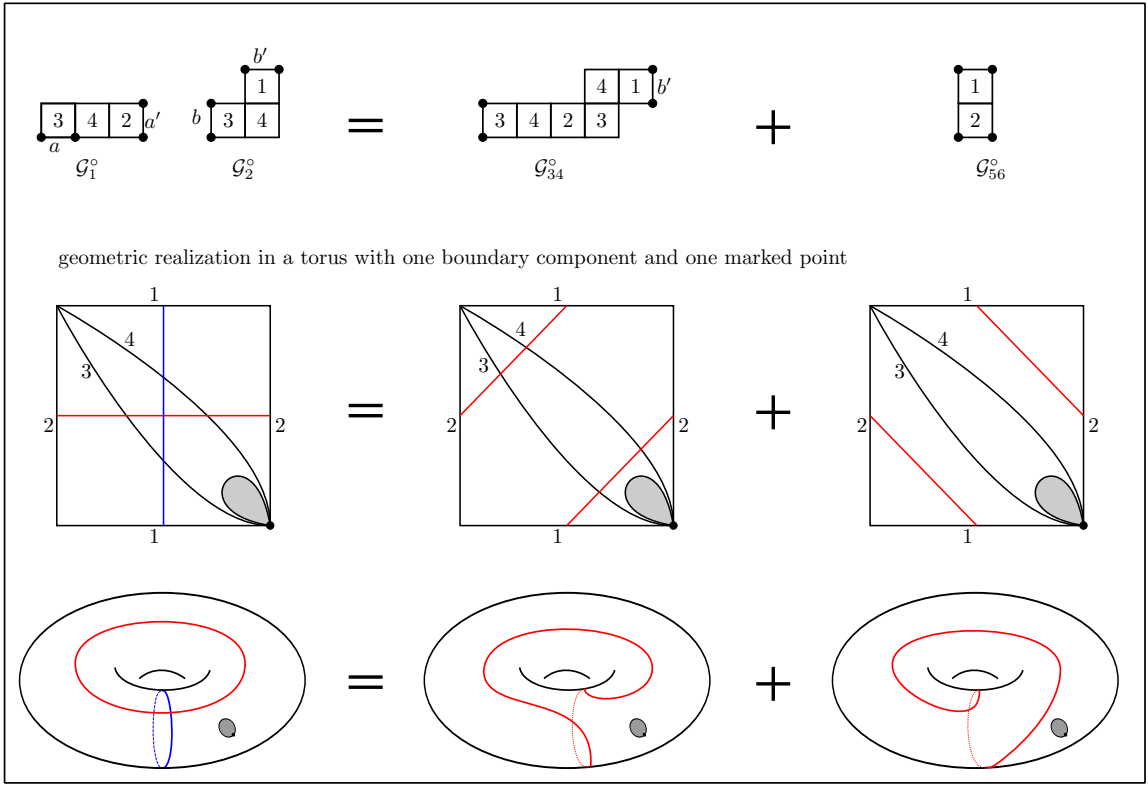}
\caption{ An example of a resolution of a crossing between two band graphs.}
\label{torus3}
\end{center}
\end{figure}

Let \[\band_{34}=(\calg_1\cup_{a'}\calg_2 )^{b'},\]
where the two subgraphs are glued along the edges $a'$  of $\calg_1 $ and the unique edge in $\calgSW_2\setminus\{b\}$, and the resulting snake graph is glued to a band graph along the edges $b'\in \calgNE_2$ and the unique   edge in ${}_{SW}\calg_1\setminus\{a\}$.

Let
\[\band_{56}=(\calg_1[t+1,d]\cup_{e}\ocalg_2[d',t'+1 ])^{e'},\]
where the two subgraphs are glued along the unique edge of $\calgNE_1\setminus\{a'\} $ and the  unique edge of $\calgSW_2\setminus\{b\} $, and the resulting snake graph is glued to a band graph along the unique boundary edge in ${}_{SW}G_{t+1}$ and the unique   boundary edge in $G_{t'+1}^{'N\!E}$. 

\begin{defn} \label{resolution bandgraphs}
 In the above situation, we say that the element $\band_{34}  + \band_{56} \in \mathcal{R}$
   is the {\em resolution of the crossing} of $\band_1$ and $\band_2$  at the overlap $\calg$ and we denote it by $\res_{\calg} (\band_1,\band_2).$ 
\end{defn}

\subsection{Grafting of a snake graph and a band graph} \label{sect 3.3graft}

Let $\calg_1=(\gi 12d)$ and $\calg_2=(\gii 12{d'})$ be two snake graphs which have at least one tile, and let $f_2$ be a sign function on $\calg_2$. Let $\calg_2^\circ=\calg_2^b$ be the band graph obtained from  $\calg_2$ by identifying an edge $b\in {}_{SW}\calg_2$ with the unique edge $b'\in\calg_2^{N\!E}$ such that $f_2(b)=f_2(b')$. Let $\ze$ be the unique edge in ${}_{SW}\calg_2$ that is different from $b$ and let $\ze'$ be the unique edge in $\calg^{N\!E}_2$ that is different from $b'$.
Let $\delta$ denote the north edge of $G_d$ if $\ze$ is the south edge of $G'_1$ and let $\delta$ be the east edge of $G_d$ if $\ze$ is the west edge of $G'_1$. Let $\delta'$ be the unique edge in $\calg_1^{N\!E}$ that is different from $\delta$. See Figure~\ref{figgrafting} for an example.

Define two snake graphs as follows.

\begin{align*}
 \calg_{34}=&\calg_1\cup(\calg_2 \setminus \suc(e')),\  \parbox[t]{.7\textwidth}{glued along the edges $\delta$ and $\ze$, where $e'\in\Int(\calg_2)$  is the last edge such that $f(e')=-f(b')$;}
\\
 \calg_{56}=&\,\calg_{1}\cup\overline{(\calg_2\setminus \pred(e))},\ \parbox[t]{.7\textwidth}{glued along the edges $\delta'$ and $\ze'$, where $e\in \Int(\calg_1)$ is the first edge such that $f(e)=-f(b)$.}&
\end{align*}
\medskip

We also consider the case where the snake graph $\calg_1$ is a single edge.
In this case, we define the two snake graphs as follows, see Figure \ref{figgrafting2} for an example.
\begin{align*}
 \calg_{34}=&\calg_2  \setminus\pred(e),\ \textup{where $e\in \Int(\calg_2) \cup\calg_2^{N\!E}$ is the first edge such that $f(e)=-f(b)\,$;}
\\
\calg_{56}=& \calg_2 \setminus \suc(e')\setminus\pred(e),\ \textup{where}\\&\left\{ \begin{array}{l} 
 \textup{$e'\!\in\Int(\calg_2)\cup\calgSW_2$  is the last edge such that $f(e')=-f(b')\,$;}\\\textup{$e\,\in \Int(\calg_2\setminus 
 \suc(e')) \cup (\Int(\calg_2\setminus 
 \suc(e'))^{N\!E}$ is the first edge such that $f(e)=f(b)$.}\end{array}\right. 
 \end{align*}

\begin{figure}
\begin{center}
\scalebox{1}{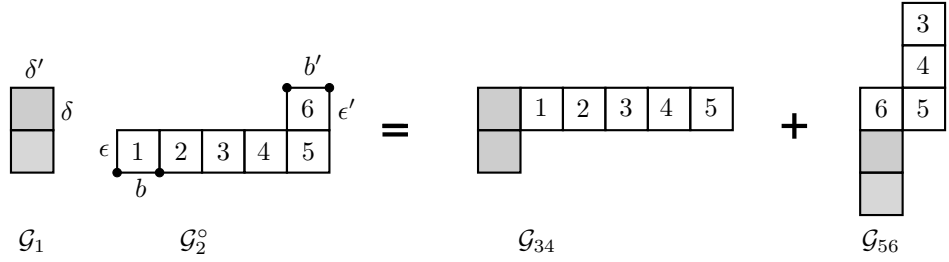}
\caption{ An example of grafting of a snake graph and a band graph.}
\label{figgrafting}
\end{center}
\end{figure}

\begin{figure}
\begin{center}
\scalebox{1}{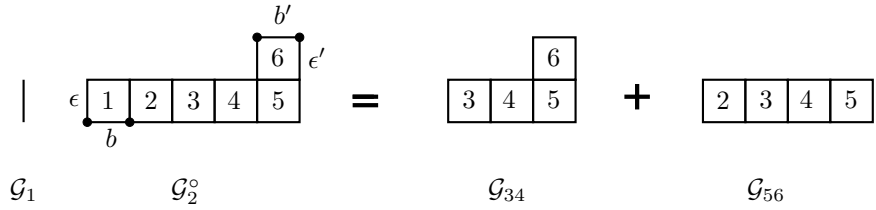}
\caption{ An example of grafting of a band graph and a single edge.}
\label{figgrafting2}
\end{center}
\end{figure}

\begin{defn} \label {grafting}
 In the above situation, we say that the element $\calg_{34} + \calg_{56} \in\mathcal{R}$ is the {\em resolution of the grafting of $\calg_2$ on $\calg_1^\circ$ in $G_d$ }and we denote it by $\graft_{d} (\calg_1^\circ,\calg_2).$ 
\end{defn}

\subsection{Resolutions for self-crossing band graphs}\label{sect 3.3}
Let $\calg_1^\circ$ be a band graph with self-crossing $i_1(\calg)\cong i_2(\calg)$.  According to remark \ref{remsis1}, we may choose $b\in\Int(\band_1)$  such that the overlap $i_1(\calg)$ in  $\calg_1=(\band_1)_b$ starts with the first tile.

Let $\calg_1=(G_1,G_2,\ldots,G_d)$  be the tiles of $\calg_1$, 
let $t$ be such that $i_1(\calg)=\calg_1[1,t]$ and let $s'$ be such that the first tile of $i_2(\calg)$ is $G_{s'}$.
Label $b\in {}_{SW}G_1$ and $b'\in G_d^{N\!E}$ the edges that  correspond to the interior edge $b$ of $\calg_1^\circ$. 

\emph{Case 1. Overlap in the same direction.}

We define the following band graphs.

\[\begin{array}
{rcll}
\calg_3^\circ &=& (\calg_1[s',d])^{b'} ;\\ \\
\calg_4^\circ&= &(\calg_1[1,s'-1])^{c}, \parbox[t]{.70\textwidth}{ where $c=e_{s'-1}$ is the interior edge shared by $G_{s'-1}$ and $G_{s'}$;} \\
\end{array}\]
and $\band_{56}$ depends on several cases and is defined below. 

\begin{enumerate}
\item If $s'\le t$ (see Figure \ref{fig(s'<t)}) 
then
 \[
  \calg_{56}^\circ =\left\{\begin{array}{ll}
  -(\calg_1[2s'-1,d])^{b'}, &\textup{if }2s'-1\le d; \\
  -2, &\textup{if }2s'-2= d. \end{array}\right.\]
  
  We show now that there are no other possibilities for the relation between $s'$ and $d$. 
  If $s' \le t'\le d$ then using the equation $t'-s'=t-1$ we have 
  $d\ge t'=t+s'-1\ge 2s'-1 $. 
  
  Suppose now that $t'\notin[s',d]$. 
  First note that since $s'\le t$ and the intersection of the overlaps is connected 
  then 
 $t'\notin[s',d]$ implies that the overlap is the whole band graph, thus $t=d$ and $t'=s'-1$.
Note that if $2s'-2>d$ then using $b=e_{s'-1}$ as cutting edge to redefine $\calg_1=(\calg_1^\circ)_b$ and making a    change of variables $(s,t) \leftrightarrow (s',t')$, we are in  the case $2s'-1\le d$ above. 

In the case $2s'-2=d$, we have $t=d$ and $\calg_1^\circ$ is the 2-bracelet $\brac_2(\calg_4^\circ)$ of $\calg_4^\circ$.  Moreover $\band_3  \cong \band_4$.

\begin{figure}
\begin{center}
\scalebox{0.8}{ 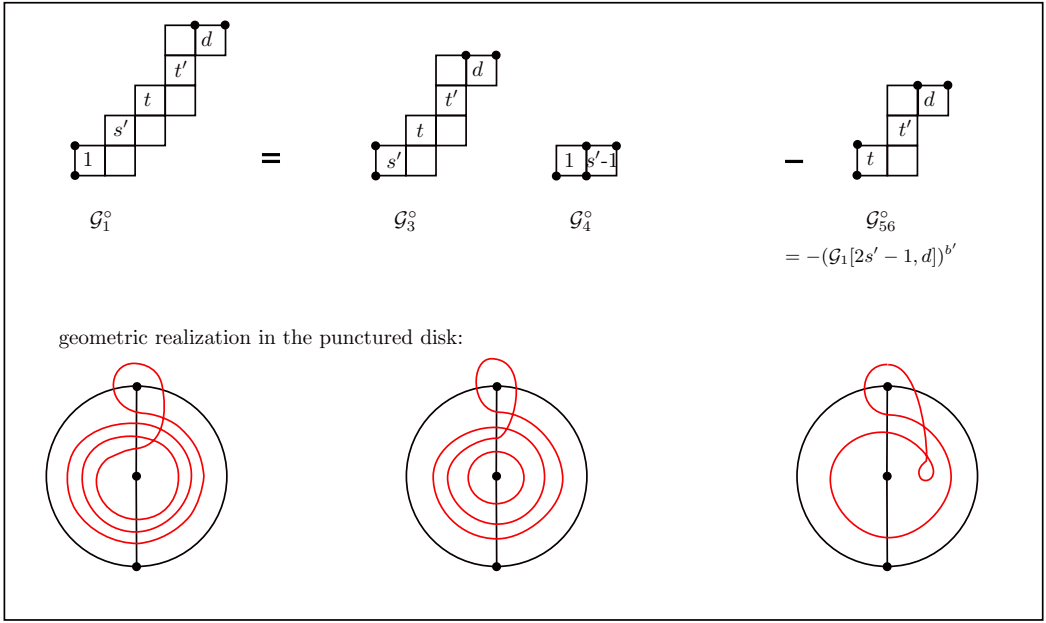} 
\caption{Example of resolution of self-crossing band graph when $s'\le t$ together with geometric realisation on the punctured disk.}
\label{fig(s'<t)}
\end{center}
\end{figure}

\item If $s'> t+1$ 
then, if $t'<d$, let
 \[
  \calg_{56}^\circ =(\overline{\calg}_1[s'-1,t+1]\cup_c \calg_1[t'+1,d])^{a'};
 \]
 where $c\in\GNE_{s'-1}$ and $a'\in\GNE_d$ are the unique edges such that $c\ne e_{s-1}$ and $a'\ne b'$. 

On the other hand, if $t'=d$, then redefine $\calg_1=(\calg_1^\circ)_b$ with $b=e_{s'-1}$ the interior edge of $\calg^\circ_1$ which is the first edge of $i_2(\calg)$ and make a change of variables $(s,t) \leftrightarrow (s',t')$. Then this situation corresponds to the case (3).

\item  If $s'=t+1$ then we have two subcases. 
\begin{enumerate}
\item If $t'=d$ then let 
\[   \calg_{56}^\circ =- \,\emptyset. \]

For an example see Figure \ref{figexample}.
\begin{figure}
\begin{center}
\scalebox{1}{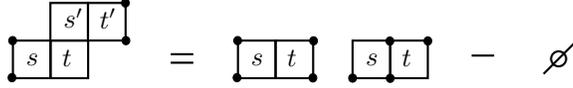}
\caption{ An example of a resolution of a band graph with $s'=t+1$ and $t'=d$.}
\label{figexample}
\end{center}
\end{figure}

\item 
If $t'< d$, we 
let $\ell\ge1$ be the smallest integer such that $f(e_{t'+\ell})=f(e_{d-\ell})$. Note that such an $\ell$ always exist since for $\ell=d-t'$ the above condition becomes $f(e_d)=f(e_{t'})$, which always holds since
$f(e_{t'})=-f(e_{s'-1})=-f(e_t)=f(e_{s-1})=f(e_d)$.
Define
%
\[   \calg_{56}^\circ =\pm \left\{ 
\begin{array}
 {cl}
 (\calg_1\setminus\pred(e_{t'+\ell})\setminus \suc(e_d-\ell))^c &\parbox[t]{.4\textwidth}{if $ d>t'+2\ell $, where $c$ is the unique boundary edge in $\GSW_{t'+\ell+1}$;} \\
 \emptyset &\textup{if $d=t'+2\ell$;}\\
 0 & \textup{if $d<t'+2\ell$;}
\end{array}\right.
\]
where the sign  of $\calg_{56}^\circ$ is $+$ if $f(e_{t'})=f(e_{t'+\ell})$, and it is $-$ otherwise.

Examples of this case are given in  Figure \ref{fig(s'=t+1)}.

\end{enumerate}
\end{enumerate}

\begin{remark}
 It is not necessary to consider the case $t'>d$ in the above definition of $\band_{56}$. Indeed if $t'>d$, then we must have $s'>t$, since the intersection of the overlaps $i_1(\calg)\cap i_2(\calg)$ is connected, and then interchanging the roles of $i_1(\calg) $ and $i_2(\calg)$ will produce the case (1) above. 
\end{remark}

\begin{figure}
\begin{center}
\scalebox{0.8}{ 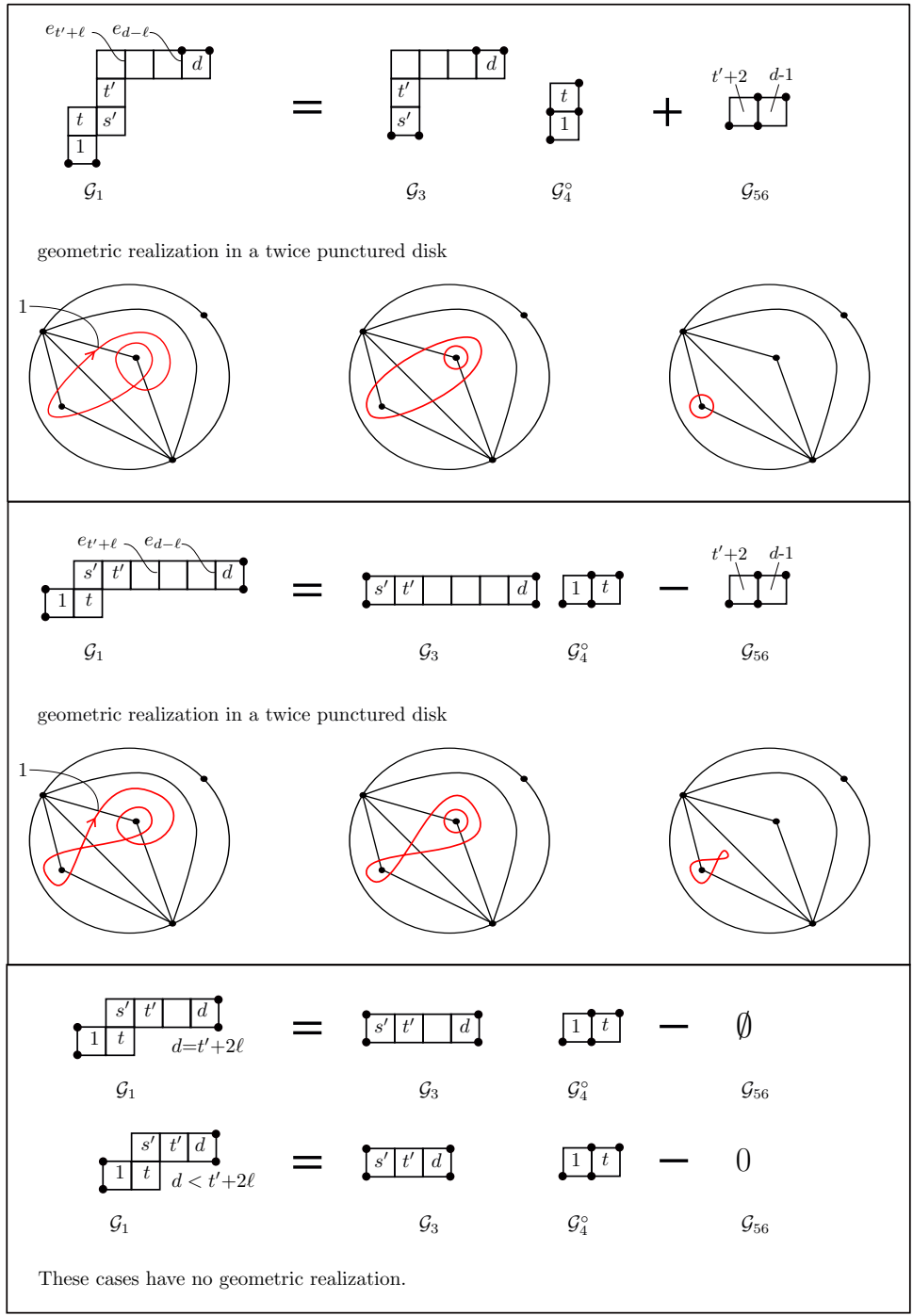}
\caption{Examples of resolutions of self-crossing  band graphs when $s'=t+1$. }
\label{fig(s'=t+1)}
\end{center}
\end{figure}

\begin{defn} \label{resolution self-crossing in same direction}
 In the above situation, we say that the element $\band_{3}\sqcup\band_4  + \band_{56} \in \mathcal{R}$
   is the {\em resolution of the self-crossing} of $\band_1$  at the overlap $\calg$ and we denote it by $\res_{\calg} (\band_1).$ 
\end{defn}

\emph{Case 2. Overlap in the opposite direction.} 
As before, 
let $\band_1=\calg_1^b$ be a band graph with self-crossing overlap
$i_1(\calg)=\calg_1[1,t]\cong \calg_1[s',t']=i_2(\calg).$ We suppose now that the overlap is in the opposite direction, thus the first tile of $\calg$ is mapped to the first tile of $\calg_1[1,t]$ {{under $i_1$}} and to the last tile of $\calg_1[s',t']$ {{under $i_2$}}.

With this notation, we define

\[\begin{array}{rcl} 
\band_{34}&=&\left(\calg_1[1,t]\cup\ocalg_1[{ s'-1,t+1}]\cup\calg_1[s',d]\right)^b,
\parbox[t]{.60 \textwidth}{ where the first two subgraphs are glued along}\\ 
&&\parbox[t]{.88\textwidth}{{ the unique boundary edge of $G_t^{N\!E}$ and the interior edge $e_{s'-1}$, and the last two subgraphs are glued along the interior edge $e_t$ and the unique boundary edge in $_{SW}G_{s'}$; }}
\\ \\
\band_5 &=& \left(\calg_1[1,s-1]\cup \calg_1[t'+1,d]\right)^{b},
\parbox[t]{.55 \textwidth}{where the two subgraphs are glued along the unique boundary edges of $G_{s-1}^{N\!E}$ and $_{SW}G_{t'+1}$.}
  \\ \\
\band_6&=& \left(\calg_1[t+1,s'-1]\right)^c, 
\textup{where $c$ is the unique boundary edge in {$_{SW}G_{t+1}$}.}
\end{array}\]

\begin{defn} \label{resolution self-crossing in opposite direction}
 In the above situation, we say that the element $\band_{34}  + (\band_{5}\sqcup \band_6 )\in \mathcal{R}$
   is the {\em resolution of the self-crossing} of $\band_1$  at the overlap $\calg$ and we denote it by $\res_{\calg} (\band_1).$ 
\end{defn}


\section{Perfect matchings}\label{section 4}

{In this section, we show that there is a bijection between the set of perfect matchings of crossing or self-crossing snake and band graphs and the set of perfect matchings of the resolution. In section~\ref{sect 5}, we will show that for labeled snake and band graphs coming from an unpunctured surface, this bijection is weight preserving and induces an identity in the corresponding cluster algebra.
}

Recall that a {\em perfect matching} $ P$ of a graph $G$ is a subset of the set of edges of $G$ such that each vertex of $G$ is incident to exactly one edge in $ P.$  

For band graphs, we need the notion of \emph{good} perfect matchings, which was introduced in \cite[Definition 3.8]{MSW2}. 
Let $\calg$ be a snake graph and let  $\calg^b$ be the band graph obtained from $\calg $ by glueing along the edge $b$ as defined in section~\ref{sect band}. If $P$ is a perfect matching of $\calg$ containing the edge $b$,  then $P\setminus\{b\}$ is a perfect matching of $\calg^b$.
In the following definition of good perfect matchings for arbitrary band graphs, we start with the band graph and cut it at an interior edge.
\begin{defn}
 Let  $\calg^\circ$ be a  band graph. A perfect matching $P$ of $\calg^\circ$ is called a \emph{good perfect matching} if there exists an interior edge $e$ in $\calg^\circ$ such that $P\sqcup\{e\}$ is a perfect matching of the snake graph  
$(\calg^\circ)_e$ obtained by cutting $\calg^\circ$ along $e$.

\end{defn}

\begin{defn}\ 
\begin{enumerate}
\item If $\calg$ is a snake graph,  let $\match \calg$ denote the set of all perfect matchings of  $\calg$.
\item If $\calg^\circ$ is  a band graph,  let $\match \calg^\circ$ denote the set of all good perfect matchings of  $\calg^\circ$. 
\item If $R=(\calg_1\sqcup\calg_2 + \calg_3\sqcup\calg_4)\in \mathcal{R}$, we let \[\match R=\match \calg_1\times\match\calg_2 \cup \match\calg_3\times\match \calg_4.\]
\end{enumerate}
\end{defn}

The following Lemma will be useful later on.
\begin{lem}\cite[Lemma 4.3]{CS2}
 \label{lemNE}
 Let $\calg$ be a snake graph with sign function $f$ and let $P$ be a matching of $ \calg$ which consists of boundary edges only. Let NE be the set of all north and all east edges of the boundary of $\calg$ and let SW be the set of all south and west edges of the boundary. Then
 \begin{enumerate}
\item $f(a)=f(b)$ if $a$ and $b$ are both in $P \,\cap\,$NE or both in $P\,\cap\,$SW.

\item $f(a)=-f(b)$ if one of $a,b$ is in  $P\,\cap\,$NE and the other  in $P\,\cap\,$SW.
\item If $a\in P\,\cap\,$NE, or if  $a\in $SW but $a\notin P$, then 
\[P= \{b\in \textup{NE}\mid f(b)=f(a)\}\cup \{b\in \textup{SW}\mid f(b)=-f(a)\}.\]

\item If $a\in P\,\cap\,$SW, or if  $a\in $NE but $a\notin P$, then 
\[P= \{b\in \textup{NE}\mid f(b)=-f(a)\}\cup \{b\in \textup{SW}\mid f(b)=f(a)\}.\]
\end{enumerate}
\end{lem}

\subsection{Orientation and switching position}
 Let $\calg$ be a snake graph. Recall that, by definition, $\calg$ comes with a fixed embedding in the plane and that each edge of $\calg$ is either horizontal or vertical.

\subsubsection{Orientation}
We define two orientations of $\calg$ as follows.
\begin{itemize}
\item [(i)] The \emph{left-right orientation} consists in orienting all horizontal edges of $\calg$ from left to right and all vertical edges upwards.
\item [(ii)] The \emph{right-left orientation} consists in orienting all horizontal edges of $\calg$ from right to left and all vertical edges downwards.\end{itemize}

Let $\calg=(\gi12d)$  be the ordered sequence of tiles defining $\calg$. Thus $G_1$ is the southwest tile of $\calg$ and $G_d$ is the northeast tile. We refer to this ordered sequence as the \emph{direction} of $\calg$. The snake graph $\ocalg=(G_d,\ldots,G_2,G_1)$ is the snake graph $\calg$ in the opposite direction. 

\subsubsection{Order} We also define a partial order on the set of edges of $\calg$ as follows. Let $e_1, e_2$ be two edges in $\gr12d$, and let $k_1,k_2$ be the smallest positive integers such that the tile $G_{k_j}$ contains the edge $e_j$, for $j=1,2$. {Then} we define a partial order by
\[
e_1 \le e_2 \quad \Longleftrightarrow \quad \left\{\begin{array}{ll}
  k_1<k_2 \textup{ or }  \\ 
   k_1=k_2 \textup{ and $e_1\in {}_{SW}G_{k_1}$ and $e_2\in G_{k_1}^{N\!E}$ or }\\
   k_1=k_2 \textup{ and $e_1=e_2$}. 
\end{array}\right.
\]
Note that two edges $e_1,e_2$ are non-comparable with respect to this partial order if and only if they are the south and the west edge of the same tile or if they are the north and the east edge of the same tile. In particular, the partial order $\le$ induces a total order on every perfect matching $P$ of $\calg$. 

\subsubsection{Switching position} 

Fix the left-right orientation on $\calg$.

First consider the case where  $\gr12d$ and $\grr12d$ are two snake or band graphs with overlap $i_1(\calg)\subset\calg_1 $ and $i_2(\calg)\subset\calg_2 $. 
Let $P=(P_1,P_2)\in\match(\calg_1,\calg_2)$ be a perfect matching. Consider the set 
\[ \cals=\{(e_1,e_2)\in P_1\times P_2\mid \textup{$e_1 $ and $e_2$ have the same starting point in $\calg_3$.}\}\] 
Thus if $(e_1,e_2)\in \cals$ then the starting point of $e_1$ in $\calg_1 $ lies inside $i_1(\calg)$ and
the starting point of $e_2$ in $\calg_2 $ lies inside $i_2(\calg)$.

Now consider 
the case where  $\gr12d$ is a snake or band graph with self-overlap $i_1(\calg)\subset\calg_1 $ and $i_2(\calg)\subset\calg_1 $. 
Let $P=P_1\in\match\calg_1$ be a perfect matching. Consider the set 
\[ \cals=\{(e_1,e_2)\in P_1\times P_1\mid \textup{ $e_1\ne e_2$, $e_1 $ and $e_2$ have the same starting point in $\calg_3$.}\}\]
Again,  if $(e_1,e_2)\in \cals$ then  the starting point of $e_1$ in $\calg_1 $ lies inside $i_1(\calg)$ and
the starting point of $e_2$ in $\calg_1 $ lies inside $i_2(\calg)$.

 In both cases the set $\cals$ has a total order defined by 
$(e_1,e_2)\le(e_1',e_2')$ if $e_1\le e_1'$ in $P_1$.

\begin{defn}
 If $\cals$ is non-empty, we call the vertex that is the common starting point of the first pair $(e_1,e_2)\in \cals$ the \emph{switching position of the perfect matching $P$ with respect to the overlap $i_1(\calg)\cong i_2(\calg)$.}

If $\cals$ is empty, we say that the perfect matching $P$ \emph{has no switching position}.
 
\end{defn}
See Figure \ref{figswitching2} for an example.

\begin{figure}
\begin{center}
\scalebox{1}{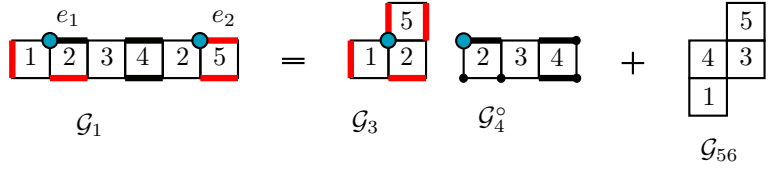}
\caption{An example of the switching operation. The snake graph $\calg_{1}$ has a selfcrossing overlap consisting of a single tile labeled 2. The perfect matching indicated in red and black has a switching position marked in blue. This is the starting point of the edges $e_1$ and $e_2$.  In $\calg_3$ the two edges $e_1,e_2$ start at the same point.}
\label{figswitching2}
\end{center}
\end{figure}

\begin{lem}\label{lemboundary}
If $P$ does not have a switching position, then the restrictions $P|_{i_1(\calg)}$ and  $P|_{i_2(\calg)}$ consist in complementary boundary edges of $\calg$. More precisely, for every edge $a$ in $\calg$, we have
\begin{itemize}
\item[\rm(a)]  if $i_1(a)\in P$ or $i_2(a)\in P$ then $a $ is a boundary edge of $\calg$, 
\item[\rm(b)]  if $i_1(a)\in P$ then $i_2(a)\notin P$,
\item[\rm(c)]  if $i_2(a)\in P$ then $i_1(a)\notin P$,
\item [\rm(d)] if $a$ is a boundary edge in $\calg$ and $a\notin \calgSW\cup\calgNE$ then $a\in P|_{i_1(\calg)}$ or $a\in  P|_{i_2(\calg)}$.
\end{itemize}
\end{lem}

\begin{proof}
 This has been shown in \cite[Section 7]{CS}.
\end{proof}

\subsection{Switching operation}\label{switching}
The following construction was introduced in \cite{CS} and used also in \cite{CS2}. Let $\calg_1,\calg_2$ be two snake  graphs  with a crossing local overlap $\calg$ and embeddings $i_1:\calg\to\calg_1$ and $i_2:\calg\to \calg_2$. Let $\calg_3,\calg_4$ be the pair of snake graphs in the resolution of the crossing which contain the overlap. Thus $\calg_3$ contains the initial part of $\calg_1$, the overlap, and the terminal part of $\calg_2$, whereas $\calg_4$
contains the initial part of $\calg_2$, the overlap, and the terminal part of $\calg_1$.

Given two perfect matchings $P_1\in\match\calg_1$ and $P_2\in \match\calg_2$ which have a switching position, using the edges of $P_1$ up to that position and the edges of $P_2$ after that position yield a perfect matching on $\calg_3$.  Moreover, using the edges of $P_2$ up to the switching position and those of $P_1$ after the switching position {also} yields a matching of $\calg_4$. {The possible local configurations of the perfect matchings at the switching position are explicitly listed in Figures 6-11 in \cite[section 3]{CS}.}

If a switching position exists, we define the \emph{switching operation} to be the map $(P_1,P_2)\mapsto (P_3,P_4)$, 
where $P_3$, respectively $P_4$, is the matching of $\calg_3$, respectively $\calg_4$, obtained by applying the method above at the first switching position. In the same way, we define the switching operation sending matchings of $ \calg_3\sqcup \calg_4$ to matchings of $\calg_1\sqcup\calg_2$. {
If no switching position exists, then the restrictions $(P_1\cup P_2)|_{\calg_5}$  and $(P_1\cup P_2)|_{\calg_6}$ are perfect matchings of $\calg_5$ and $\calg_6$, respectively.}

\smallskip
The switching operation generalises in a straightforward
way when we replace the pair $(\calg_1,\calg_2)$ of crossing snake graphs by
\begin{itemize}
\item[-] a crossing pair consisting of a snake graph and a band graph,
\item [-] a crossing pair consisting of two band graphs, 
\item[-]
 a single snake graph or band graph  with a crossing self-overlap, as long as the self-overlap does not have an intersection. In particular, this applies to the case where the self-overlap is in the opposite direction.
 \end{itemize}

We illustrate this in the case of a crossing pair consisting of a snake graph and a band graph in the case (4c) of section~\ref{sect 3.1}, where the graph $\calg_{56}'$ consists of a single tile and $\calg_{56}$ is the unique boundary edge of $\calgSW_2$, see Figure \ref{figresolutionwithmatching1}.

\begin{figure}
\begin{center}
\scalebox{0.9}{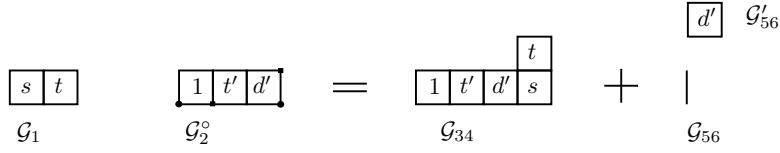}
\caption{ An example of a resolution of a snake and a band graph with $s=1,t=d$}
\label{figresolution}
\end{center}
\end{figure}

\begin{figure}
\begin{center}
\scalebox{0.9}{\Large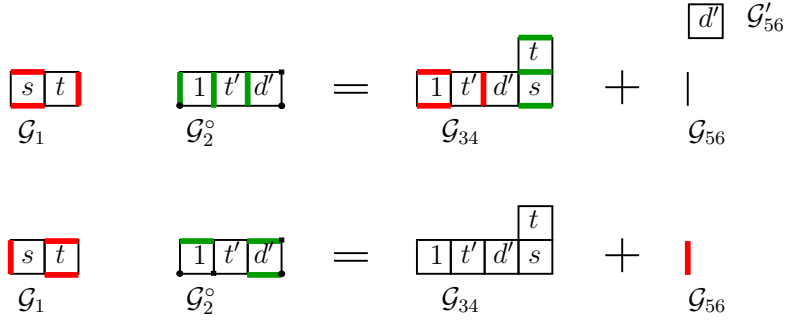}
\caption{An example of the bijection with $s=1,t=d$ with switching position in the top row, without switching position in the bottom row.}
\label{figresolutionwithmatching1}
\end{center}
\end{figure}

\subsection{Some lemmas about perfect matchings} 
Throughout this subsection, we consider a pair of crossing snake or band graphs, or a self-crossing snake or band graph, and we denote the overlap by $\calg$ and the embeddings by $i_1,i_2$. 
\begin{lem}
 \label{lemma1}
 Let $\calg_1$ be a snake or band graph with selfcrossing overlap $i_1(\calg)=\calg_1[s,t] $, $i_2(\calg)=\calg_1[s',t']$ with $s'=t+1$, and let $e_t$ denote the interior edge shared by the tiles $G_t$ and $G_{s'}$. If $P_1\in\match \calg_1$ has no switching position, then  $e_t\in P_1$.
\end{lem}
\begin{proof} 
 Without loss of generality, we may suppose that $G_{t+1}$ is north of $G_t$. Then let $b_2$ be the west  edge of $G_{t+1}$ and $b_2'$ be the east  edge of $G_t$
see Figure \ref{figswitchingtplusone}. Note that $b_2$ and $b_2'$ are boundary edges of $\calg_1$.
 Suppose that $P_1$ does not contain $e_t$. Then we have $b_2\in P_1$ or $b_2'\in P_1$.

 There are exactly 3 local configurations for which the matching $P_1$ does not have a switching position on $\calg_1[s-1,s+1]$ and $\calg_1[s'-1,s'+1]=\calg_1[t,t+2]$, and these 3 cases are listed in \cite[Figure 7]{CS}.
If $b_2\in P_1$, then $P_1$ is not one of these 3 cases, hence it has a switching position.
The case where $b_2'\in P_1$ is symmetric, where the symmetry is given by considering the opposite snake graph $\ocalg_1$. 
\end{proof}

\begin{figure}
\begin{center}
\small\scalebox{1}{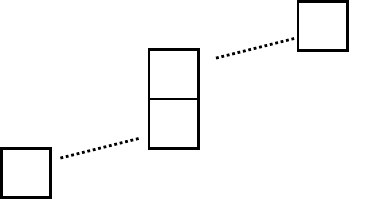}
\caption{Proof of Lemma \ref{lemma1}.}
\label{figswitchingtplusone}
\end{center}
\end{figure}

\begin{lem}\label{lemma3} Let $\calg_{12}$ be a pair of crossing snake graphs or a self-crossing  snake graph, and let $\res_{34}$ be the part of the resolution containing the overlap. Then  the switching operation 
\[\varphi \colon \{P\in\match\calg_{12} \mid P \textup{ has a switching position}\} \to \match\res_{34}\] is injective.
\end{lem}

\begin{proof}
Let $\calg$ denote the overlap. Suppose $ \varphi(P)=\varphi(P') $, where $P$ has switching position $x\in \calg$ and $P'$ has switching position $x'\in \calg$. Let $a,c$ be the edges in $\calg$ with starting point $x$ (or endpoint $x$) such that $i_1(a), i_2(c)\in P $ are the switching edges in $\calg_1$. Then $\varphi(i_1(a))$ and $\varphi(i_1(c))$ have starting points $i_3(x)$ and $i_4(x)$ respectively, where $i_3, i_4$ denote the embeddings of the overlap into $\res_{34}$. Moreover, $x$ is the first such point in $\calg$ because otherwise we would have had an earlier switching position for $P$ in $\calg_1$. Thus $i_3(x), i_4(x)$ is a switching position for $\varphi(P')$, and it follows that $x=x'$. Moreover, the assumption $\varphi(P)=\varphi(P') $ implies that $P=P'$ before and after the switching position, hence $P=P'$.
\end{proof}

\subsection{Bijection}
 The following theorem is the main result of this section.

\begin{thm}\label{thmbijection} Given two crossing snake or band graphs, or a single snake or band graph with a self-crossing, there is a bijection between the set of perfect matchings of the crossing graph and the set of perfect matchings of the resolution of the crossing.
\end{thm}

The theorem has been proved for pairs of snake graphs in \cite{CS}, and for self-crossing snake graphs in \cite{CS2}. Except for the case were the self-crossing snake graph has an intersection of overlaps, these proofs all use the switching operation described in section \ref{switching}. 
 The proof is analogous for the cases of  a crossing between a snake graph and a band graph, a crossing between two band graphs, as well as for a self-crossing of a single band graph as long as the self-overlap does not have an intersection.
 
 It therefore remains to show the theorem for a single band graph  with a crossing self-overlap which has an intersection. In other words, we are considering the cases (1) and (3) of section~\ref{sect 3.3}.
 Recall that $\calg_1=(\band_1)_b=(G_1,G_2,\ldots,G_d)$ and the two overlaps are given by
 \[ i_1(\calg) =\calg_1[1,t] \quad i_2(\calg) =\calg_1[s',t']\]
 Case (1). If $s'\le t$, then there are the two subcases $2s'-1\le d$ and $2s'-2=d$.
In the case  $2s'-1\le d$, the overlap is not the whole band graph and we have $1<s'\le t<t'<d$, and the proof is analogous to the proof in the case of a self-crossing snake graph with $s'\le t$ in \cite[Section 4.4]{CS2}.

Suppose therefore that  $2s'-2=d$.
We illustrate this case in an example in Figure \ref{figresolutionwithmatching}.
 We have already seen in section \ref{sect 3.3} that in this case we have $t=d$ and $\band_1$ is the 2-bracelet $\brac_2(\band_4)$ of $\band_4.$ Moreover, $\band_3  \cong \band_4$.
 
\begin{figure}
\begin{center}
\scalebox{0.7}{\Large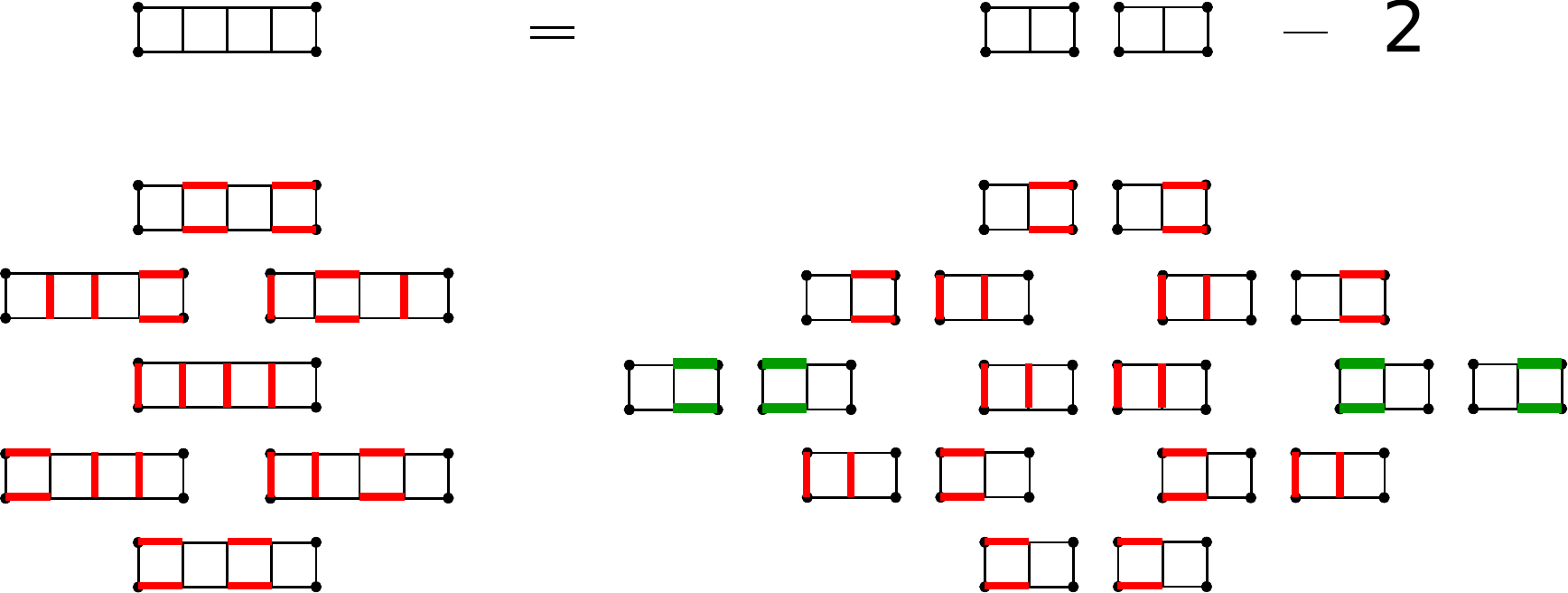}
\caption{The bijection for the resolution of the 2-bracelet of a band graph with two tiles. The 7 perfect matchings of the bracelet on the left are mapped to the 7 perfect matchings of the square of the band graph shown in red on the right. The 2 green perfect matchings are the only matchings of the square of the band graph for which there is no switching position. These two matchings correspond to $\band_{56}=-2$.}
\label{figresolutionwithmatching}
\end{center}
\end{figure}

\begin{lem}
 \label{lembracelet} 
  Let $\band_1$ be a band graph with selfcrossing overlap $i_1(\calg)=\calg_1[1,d] $, $i_2(\calg)=\calg_1[s',d]\cup \calg_1[1,s'-1]$ with $2s'-2=d$. Then every $P\in\match \calg_1$ has a switching position.
\end{lem}

\begin{proof}
 Suppose that $P$ has no switching position. Then Lemma \ref{lemboundary} implies that the restriction of $P$ to the two overlaps $i_1(\calg)$ and $i_2(\calg)$ consist {of} complementary boundary edges of the overlap. Since both $i_1(\calg)$ and $i_2(\calg)$ are equal to the whole band graph $\band_1$, we see that $P$ consists entirely of boundary edges of $\band_1$. It follows from Lemma \ref{lemNE} that all south and all west edges in $P$ have the same sign $\ze$ and all north and all east edges in $P$ have the same sign $-\ze$. Since $t=d$, the glueing edge $b$ is the same as the edge $e_t$. Moreover, since the overlap is crossing, we know that the sign of $e_{s'-1}$ is equal to the sign of $e_t=b$.  Since $b\in{}_{SW}G_1$ and $e_{s'-1}\in{}_{SW}G_{s'}$ are interior edges, it follows that the boundary in ${}_{SW}G_1$ and ${}_{SW}G_{s'}$ have the same sign. Thus the boundary edge in ${}_{SW}G_1$ is in $P$ if and only if 
  the boundary edge in ${}_{SW}G_{s'}$ is in $P$, which contradicts the assumption that the restriction of $P$ to the two overlaps $i_1(\calg)$ and $i_2(\calg)$ consist {of} \emph{complementary} boundary edges.  
\end{proof}

\begin{proof} [Proof of Theorem \ref{thmbijection} in case (1)]   Let $\varphi:\match \band_1\to\match (\band_3,\band_4) $ be the switching operation. By Lemma~\ref{lembracelet}, every $P\in \match \band_{1}$ has a switching position and thus the map $\varphi $ is well-defined. It is injective, by Lemma \ref{lemma3}. 
The only perfect matchings of $\match (\band_3,\band_4) $ that are not in the image of $\varphi$ are those which have no switching position. 
Since $\band_3\cong \band_4$ and  the overlap  consists in the whole band graph, the only two matchings without switching positions are the pairs $(P_-,P_+)$ and $(P_+,P_-)$ consisting in complementary boundary edges. 
This proves the theorem. 
\end{proof}
%

\noindent Case (3). If $s'=t+1$ then we have two subcases.

(a) First assume that $t'=d$.
We shall need two Lemmas before presenting the proof of the theorem.

\begin{lem}
 \label{lemma2}
 If $\band_1$ is a band graph with crossing self-overlap $i_1(\calg)=\calg_1[1,t], i_2(\calg)=\calg_1[s',t']$ with $s'=t+1$ and $t'=d$, then every matching $P_1\in \match\band_1$ has a switching position.
\end{lem}
\begin{proof}
 Suppose that $P_1$ does not have a switching position. By Lemma \ref{lemma1} we know that $P_1$ contains the edge $e_t$. 
 Interchanging the roles of $i_1(\calg)$ and $i_2(\calg)$, the same argument shows that  $P_1$ contains the glueing edge $b'$ of the band graph $\band_1$. But the endpoint of this glueing edge and the endpoint of $e_t$ form a switching position for $P_1$, a contradiction.
\end{proof}

\begin{lem}
 \label{lemma4}   If $\band_1$ is a band graph with crossing self-overlap $i_1(\calg)=\calg_1[1,t], i_2(\calg)=\calg_1[s',t']$ with $s'=t+1$ and $t'=d$, then there exists a unique matching $\bar P$ such that
 every matching $P'$ of $\match(\band_3,\band_4)\setminus\{\bar P\}$  has a switching position $i_3(x'), i_4(x')$ where $x'\in\calg$ is the starting point of two edges $a,c$ in $\calg$ such that $i_3(a), i_4(c)\in P'$.
 
Moreover, the  exceptional matching $\bar P$  consists of complementary boundary edges of $\band_3,\band_4$ such that the boundary edges in $\calgNE_3$ and $\calgNE_4$ are in $\bar P$.
\end{lem}
\begin{proof}
 First note that the switching position can always be realised as the starting point of two edges in the overlap unless it occurs in the last tile of the overlap.
  
 Suppose $P'$ does not have a switching position. Then Lemma~\ref{lemboundary} implies that $P'$ is one of the two matchings that consist of complementary boundary edges of $\band_3,\band_4$. Since, by definition of the resolution, the glueing edges are different in $\band_3$ and $\band_4$, we see that either $P'$ is the exceptional matching $\bar P$ of the statement, or $P'$ contains the unique boundary edges in $\calgSW_3$ and $\calgSW_4$. In the latter case, the starting point of these two edges form a switching position as in the Lemma.
 
 Clearly the matching $\bar P$ does not have a switching position as in the statement.
\end{proof}

\begin{proof} [Proof of Theorem \ref{thmbijection} in case (3)(a)]  Let $\varphi:\match \band_1\to\match (\band_3,\band_4) $ be the switching operation. According to Lemma~\ref{lemma2}, every $P\in \match \band_{1}$ has a switching position, and thus the map $\varphi $ is well-defined. It is injective, by Lemma \ref{lemma3}. 

We claim that the image of $\varphi$ is $\match (\band_3,\band_4)\setminus\{\bar P\}$ where $\bar P$ is the matching of Lemma \ref{lemma4}. This will complete the proof, since $\bar P$ corresponds to the unique matching of $\band_{56}=\emptyset$.

Let $P'\in \match (\band_3,\band_4)\setminus\{\bar P\}$. Lemma \ref{lemma4} implies that $P'$ has a switching position which is the starting point of two edges $i_3(a), i_4(c)$. Then switching produces a matching $P$ of $\band_{1}$ whose switching position is given by the same point which now is the starting point of $i_1(a)$ and $i_2(c)$. Switching back shows that $P'=\varphi(P)$ and we are done.\end{proof}
(b) Now assume that $t'<d$. 
We start by giving a new proof of the theorem in the case where $\calg_1$ is a snake graph. In \cite{CS2}, we have proved this case  by specifying an explicit bijection between the set of perfect matchings. But  when $s'=t+1$, the switching operation alone was not enough to define this bijection. 
Here we give a new proof, which does not give the bijection explicitly but rather uses  identities from  snake graph calculus avoiding the case $s'=t+1$. The computation is carried out in a simple example in Figure \ref{proofsprimeistplusone}. For the notation used in the proof, we refer to the more complicated example in Figure \ref{figbigpicture}.
The proof in the case where $\band_1$ is a band graph will use a very similar argument.
\begin{figure}
\begin{center}
\scalebox{0.7}{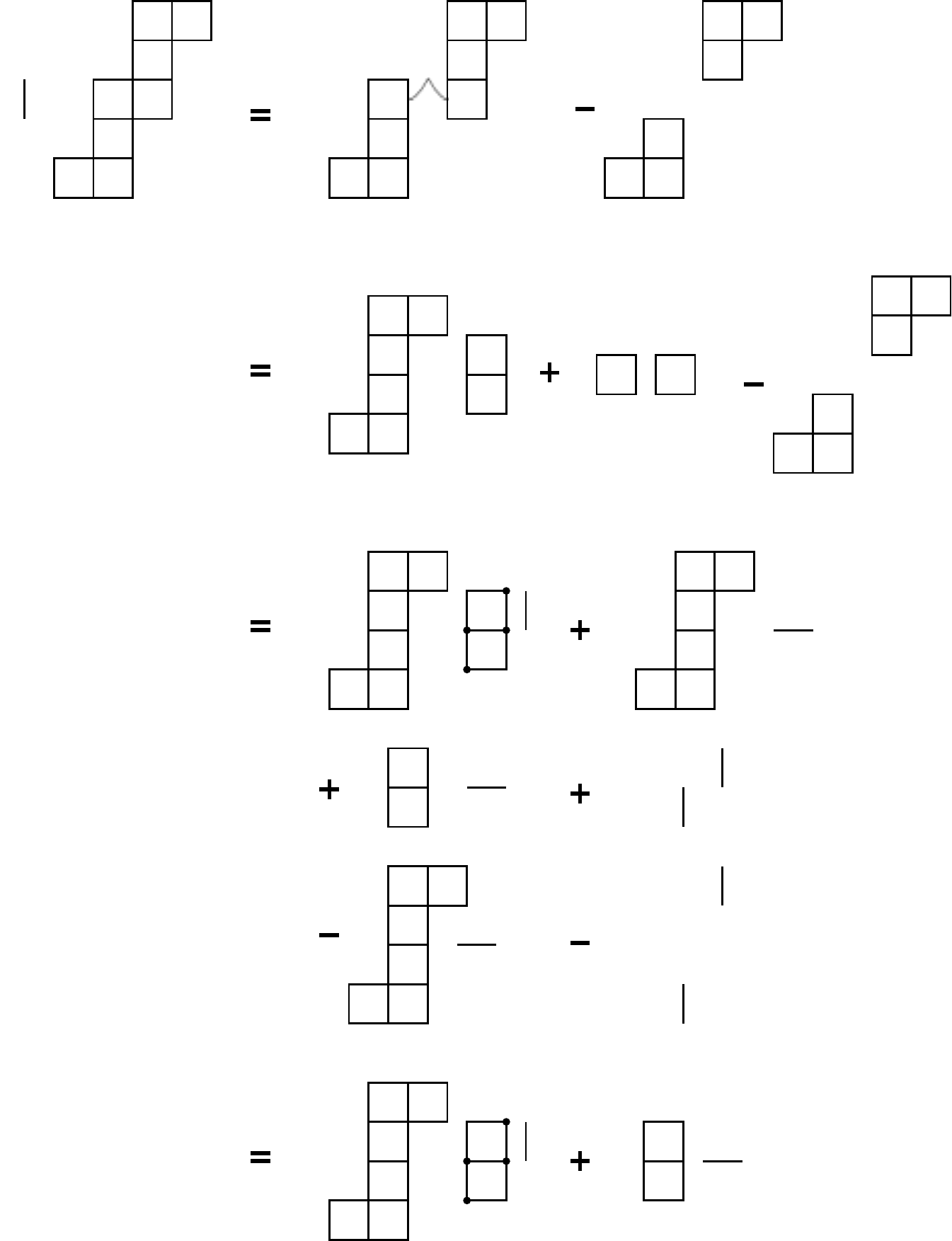}
\caption{An example of the computation in section \ref{sect 441}. Here $u=2,s=3, t=4$, $\ell_1=r_1= 3, \ell_2=1$, $r_2'=2'$.}
\label{proofsprimeistplusone}
\end{center}
\end{figure}

\begin{figure}
\begin{center}
\scalebox{1}{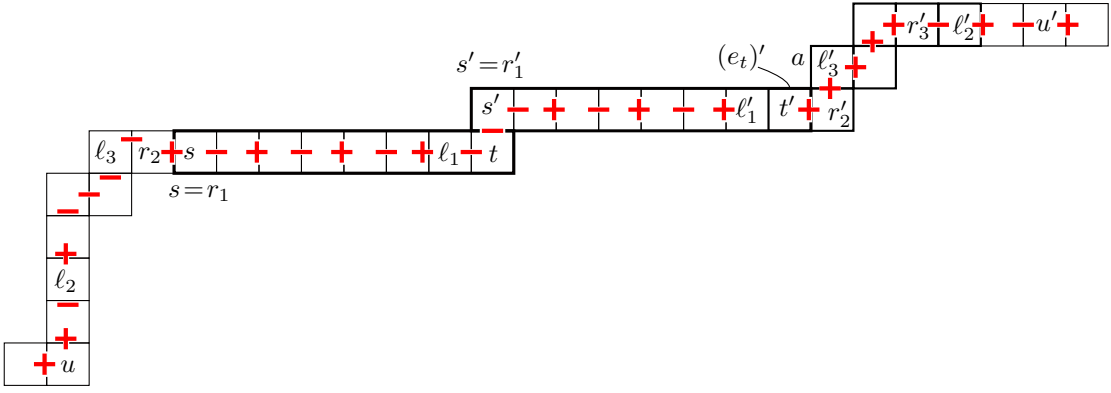}
\caption{Notation in section \ref{sect 441}.}
\label{figbigpicture}
\end{center}
\end{figure}

\subsubsection{Proof of Theorem \ref{thmbijection} for a self-crossing snake graph with $s'=t+1$} \label{sect 441}
To fix notation, let $\calg_1=\calg_1[1,d]$, let $i_1(\calg)=\calg_1[s,t] \cong \calg_1[s',t']= i(\calg_2)$ and let $j_1(\calh)=\calg_1[u,s-1] \cong \ocalg_1[u',t'+1]=j_2(\calh)$ be the secondary overlap in opposite direction. In other words, $u<s-1$ is the least positive integer such that the snake graphs $\calg_1[u,s-1] $ and $\ocalg_1[u',t'+1]$ are isomorphic. Thus $f(e_{s-2})\ne f(e_{t'+1}), f(e_{s-3})\ne f(e_{t'+2}), \ldots, f(e_u)\ne f(e_{u'-1})$ and, if $u\ne 1$ and $u'\ne d$ then $f(e_{u-1})= f(e_{u'})$. 
In the example in Figure~\ref{proofsprimeistplusone}, we have $u=2$ and $\calh$ is a single tile.

Throughout this proof we will use the following notational convention. If $a$ is an edge or a tile in $i_1(\calg) $ (respectively $j_1(\calh)$) so that $a=i_1(e)$ (respectively $a=j_1(e)$) for some edge or tile $e$ in $\calg$ (respectively $e$ in $\calh$) then we denote by $a'$ the corresponding edge or tile in $i_2(\calg)$ (respectively $j_2(\calh)$); thus $a'=i_2(e)$ (respectively $a'=j_2(e)$). 
Note that this notation extends our notation $s,s'$, $t,t'$ and $u,u'$.
Also note that $(u+1)'=u'-1, (u+2)' = u'-2 , \ldots, (s-1)'= t' +1$, since the overlap $\calh$ is in the opposite direction. Moreover, $(e_u)'=e_{u'-1}, (e_{u+1})'=e_{u'-2} , \ldots, (e_{s-1})'= e_{t'} $ for the interior edges.
Fix a sign function $f$ on $\calg$ and suppose without loss of generality that $f(e_t)=-$. It then follows from the crossing condition that $f(e_{s-1})=f(e_{t'})=+$. 

We will first consider the case where the $\calh$ overlap is not crossing. That is 
\begin{equation}
\label{notcross}
f(e_{u-1})=+  \quad \textup{ and } \quad f(e_{u'})=+. 
\end{equation}
The other case will follow easily from this one.

First we use the grafting of the snake graphs $\calg_1[1,t]$ and $\calg_1[t+1,d]$ to get the following expression for $\calg_1$
\begin{equation}
 \label{eqbij1}
 e_t \,\calg_1 =  \calg_1[1,t]\, \calg_1[t+1,d]  - \calg_1[1,\ell_1]\, \calg_1[r_1'+1,d]  ,
\end{equation}
where 
\[\begin{array}{ccc}
 \ell_1<t &\textup{ is the largest integer such that }& f(e_{\ell_1})=f(e_t)=-,
\\
r_1'>t &\textup{ is the least integer such that } &f(e_{r'_1})=f(e_t)=-.
\end{array}
\]
In other words, the edges $e_{\ell_1}$ is the last interior edge in $i_1(\calg)$ that has sign $-$ and $e_{r_1} $ is the first interior edge in $i_1(\calg)$ that has sign $-$. In particular, we have $r_1\le \ell_1$.
The first two snake graphs on the right hand side of equation (\ref{eqbij1}) have a crossing overlap $i_1(\calg)$, $i_2(\calg)$. Resolving this crossing, we get

\begin{equation}
 \label{eqbij2}
  \calg_1[1,t]\, \calg_1[t+1,d]   = \calg_3 \,\calg_4 + \calg_1[1,\ell_2]\,\calg_1[r'_2+1,d]
\end{equation}
where $\calg_3$ is the snake graph of the Theorem, $\calg_4$ is the snake graph obtained from the band graph $\band_4$ of the Theorem by cutting along the edge $e_t$ and
\[\begin{array}{ccc}
 \ell_2< s-1 &\textup{ is the largest integer such that }& f(e_{\ell_2})=f(e_{s-1})=+,
\\
r_2'>t' &\textup{ is the least integer such that } &f(e_{r_2'})=f(e_{t'})=+.
\end{array}
\]
Note that it follows from our assumption (\ref{notcross}) that $\ell_2\ge u-1$ and $r_2'\le u'$.
In other words the edge
$e_{\ell_2}$ is the last interior edge in $j_1(\calh)\cup G_{u-1} $ that has sign + and 
$e_{r'_2}$ is the first interior edge in $j_2(\calh) \cup G_{u'+1} $ that has sign +.

Furthermore, the self-grafting formula of \cite[section 3.4]{CS2} gives 
\begin{equation}
 \label{eqbij3}
  \calg_4 = \band_4 \,e_t + \calg_1 [r_1 +1, \ell_1], \end{equation} 
  where we agree that if $r_1=\ell_1$ then $ \calg_1 [r_1 +1, \ell_1]= e_{\ell_1}$ is the single edge shared by $G_{r_1+1}$ and $G_{\ell_1}$.  Recall that $r_1\le \ell_1$.
  
On the other hand, the snake graphs $ \calg_1[1,\ell_2]$ and $\calg_1[r'_2+1,d]
$ on the right hand side of equation (\ref{eqbij2}) have an overlap $\calg_1[u, \ell_2]\cong \ocalg_1[\ell_2',u']$, unless $\ell_2=u-1$. Recall that $u$ is the first tile in $j_1(\calh)$ and $u'$ is the last tile in $j_2(\calh)$. Let us suppose without loss of generality that $\ell_2'\ge r_2'+1$.
 This overlap is crossing  because, on the one hand, condition (\ref{notcross})  gives $f(e_{u'}) = + $, and on the other hand, $f(e_{\ell_2'-1}) = f((e_{\ell_2})')=-f(e_{\ell_2})=-$.
 The resolution of this crossing overlap yields
 \begin{equation}
 \label{eqbij4}
  \calg_1[1,\ell_2]\,\calg_1[r'_2+1,d] = ( \calg_1[1,\ell_2]\cup\ocalg[\ell_2'-1,r_2'+1]) \, \calg_1[\ell_2',d]
  +\calg_{56}\, e_t
\end{equation}
where $\calg_{56}=\calg_1[1,u-1]\cup \calg_1[u'+1,d]$ is the snake graph of the theorem and $e_t$ is a single edge $a\in \GSW_{r_2'+1}$ because all the interior edges of $\calg_1[r_2'+1,\ell_2'-1]$ have sign +. We denote this single edge by $e_t$ because if the snake graph comes from a triangulated surface then this edge has the same label as the edge $e_t$ and $r_2'$, which can be seen from the following local configuration.
\[ 
\begingroup%
  \makeatletter%
  \providecommand\color[2][]{%
    \errmessage{(Inkscape) Color is used for the text in Inkscape, but the package 'color.sty' is not loaded}%
    \renewcommand\color[2][]{}%
  }%
  \providecommand\transparent[1]{%
    \errmessage{(Inkscape) Transparency is used (non-zero) for the text in Inkscape, but the package 'transparent.sty' is not loaded}%
    \renewcommand\transparent[1]{}%
  }%
  \providecommand\rotatebox[2]{#2}%
  \ifx\svgwidth\undefined%
    \setlength{\unitlength}{146.01238251bp}%
    \ifx\svgscale\undefined%
      \relax%
    \else%
      \setlength{\unitlength}{\unitlength * \real{\svgscale}}%
    \fi%
  \else%
    \setlength{\unitlength}{\svgwidth}%
  \fi%
  \global\let\svgwidth\undefined%
  \global\let\svgscale\undefined%
  \makeatother%
  \begin{picture}(1,0.49858785)%
    \put(0,0){\includegraphics[width=\unitlength]{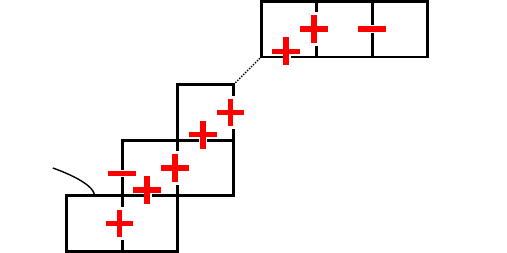}}%
    \put(0.15894367,0.04116263){\color[rgb]{0,0,0}\makebox(0,0)[lb]{\smash{$t'$}}}%
    \put(0.77208162,0.42891063){\color[rgb]{0,0,0}\makebox(0,0)[lb]{\smash{$\ell_2'$}}}%
    \put(0.27276843,0.04412426){\color[rgb]{0,0,0}\makebox(0,0)[lb]{\smash{$r_2'$}}}%
    \put(0.18972819,0.17242508){\color[rgb]{0,0,0}\makebox(0,0)[lb]{\smash{$a$}}}%
    \put(-0.00180849,0.17458412){\color[rgb]{0,0,0}\makebox(0,0)[lb]{\smash{$(e_t)'$}}}%
  \end{picture}%
\endgroup%

\]
In the case where $\ell_2=u-1$, we still obtain equation (\ref{eqbij4}) by grafting the two snake graphs on the left hand side of the equation.

Now let us consider the second pair of snake graphs $ \calg_1[1,\ell_1]$, $ \calg_1[r_1'+1,d] $ on the right hand side of equation (\ref{eqbij1}). Suppose first that $r_1\ne\ell_1$. Then this  pair has an overlap $\calg_1[r_1+1,\ell_1] \cong \calg_1 [r_1'+1,\ell_1']$. This is a crossing overlap because $f(e_{r_1})= -$ and $f(e_{\ell_1'})=-$. The resolution of this crossing yields

\begin{equation}
 \label{eqbij5}
  -\calg_1[1,\ell_1]\,\calg_1[r_1'+1,d]  = -\calg_3 \,\calg_1[r_1+1,\ell_1] 
  - \calg_1[1,\ell_3]\,\calg_1[r_3'+1,d],
\end{equation}
where $\calg_3= (\calg_1[1,\ell_1] \cup\calg_1[\ell_1'+1,d])$ is the snake graph from the theorem and
\[\begin{array}{ccc}
 \ell_3< r_1 &\textup{ is the largest integer such that }& f(e_{\ell_3})=f(e_{r_1})=-,
\\
r_3'>\ell_1' &\textup{ is the least integer such that } &f(e_{r_3'})=f(e_{\ell_1'})=-.
\end{array}
\]
In the case where $r_1=\ell_1$, we use grafting of the two snake graphs $ \calg_1[1,\ell_1]$, $ \calg_1[r_1'+1,d] $ and still obtain equation (\ref{eqbij5}) where $\calg_1[r_1+1,\ell_1]$ becomes a single edge.

 It will be useful to compare $\ell_3,r_3' $ with $\ell_2,r_2'$ defined earlier. 
Since $e_{r_1}$ is the first interior edge in $i_1(\calg)$ whose sign is $-$ and since $f(e_{s-1})=+$, it follows that 
$e_{\ell_3}$ is the last interior edge in $j_1(\calh) $ that has sign $-$.
On the other hand, we have already seen earlier that 
$e_{r'_2}$ is the first interior edge in $j_2(\calh)\cup G_{u'+1} $ that has sign +, and this implies that $e_{r_2-1}$ is the last interior edge in $j_1(\calh)$ that has sign $-$. {Therefore} we have  $\ell_3=r_2-1$ and thus
\begin{equation}
 \label{eqbij6}\ell_3'=r_2'+1.
\end{equation}
A similar argument shows that 
\begin{equation}
 \label{eqbij7}r_3'+1 = \ell_2'.
\end{equation}

We are now ready to combine all these identities. 
Using equations (\ref{eqbij2})-(\ref{eqbij4}) on the first term of the right hand side of equation (\ref{eqbij1}) and equation (\ref{eqbij5}) on the second term, we get
\[
\begin{array}{rcl}
\displaystyle  e_t\calg_1 &=& \calg_3  \band_4 \,e_t +  \calg_3 \calg_1 [r_1 +1, \ell_1]
  \\ \\
  &+&    ( \calg_1[1,\ell_2]\cup\ocalg[\ell_2'-1,r_2'+1]) \, \calg_1[\ell_2',d]
  +\calg_{56}\, e_t 
  \\ \\
  &-& 
 \calg_3\,\calg_1[r_1+1,\ell_1] 
  - \calg_1[1,\ell_3]\,\calg_1[r_3'+1,d].
\end{array}\]

Finally using equations (\ref{eqbij6}) and (\ref{eqbij7}) and simplifying by $e_t$ we get

\begin{equation} 
 \label{eqbij8} \calg_1 = \calg_3  \band_4  
  +\calg_{56}.
\end{equation}
This shows the formula in the case where the $\calh$-overlaps are non-crossing.

Now suppose that the $\calh$-overlaps are crossing. Resolving this crossing gives
\begin{equation}
 \label{eqbij9}
 \calg_1 = \calg_1[1,u-1]\cup\ocalg_1[u',u]\cup\calg_1[u'+1,d] + \calg_{56} \, \brac_2(\band_4).
\end{equation}
Let $\widetilde{\calg_1} = \calg_1[1,u-1]\cup\ocalg_1[u',u]\cup\calg_1[u'+1,d] $. Then $\widetilde{\calg_1}$ is a selfcrossing snake graph with $s'=t+1$ and where the $\calh$ overlaps are not crossing and we can apply our formula from the first case to get $\widetilde{\calg_1}=\widetilde{\calg_3}\widetilde{\band_4}+\widetilde{\calg_{56}}$. Moreover, by definition we have $\widetilde{\calg_{56}}=\calg_{56}$ and $\widetilde{\band_4}=\band_4$.
On the other hand, the bracelet formula yields $ \brac_2(\band_4)=\band_4\band_4-2$. Thus equation (\ref{eqbij9}) becomes 

\begin{equation}
 \label{eqbij10}
 \calg_1 = \widetilde{\calg_3}{\band_4}+{\calg_{56}} +\calg_{56} (\band_4\band_4-2).
\end{equation}

Furthermore, resolving the $\calh$-overlap in $\calg_3$ we see that 
$\calg_3=\widetilde{\calg_3} +\calg_{56}\band_4$ and therefore equation (\ref{eqbij10}) becomes
\[
\begin{array}{rcl}
 \label{eqbij11}
 \calg_1 &=& ({\calg_3}- \calg_{56}\band_4) {\band_4}+{\calg_{56}} +\calg_{56} (\band_4\band_4-2) \\
 &=& \calg_3\band_4-\calg_{56}.
\end{array}
\]
This completes the proof for a self-crossing snake graph with $s'=t+1$.

\begin{proof}
 [Proof of Theorem \ref{thmbijection} in case (3)(b), thus for $s'=t+1$ and $t'<d$]
 The proof for the case of a self-crossing band graph follows the same computation as the one used for self-crossing snake graph in section \ref{sect 441}. The details are left to the reader. In Figure \ref{proofsprimeistplusoneband}, we illustrate the computation for the band graph obtained from the snake graph in the example in Figure \ref{proofsprimeistplusone}.

 \begin{figure}
\begin{center}
\scalebox{0.7}{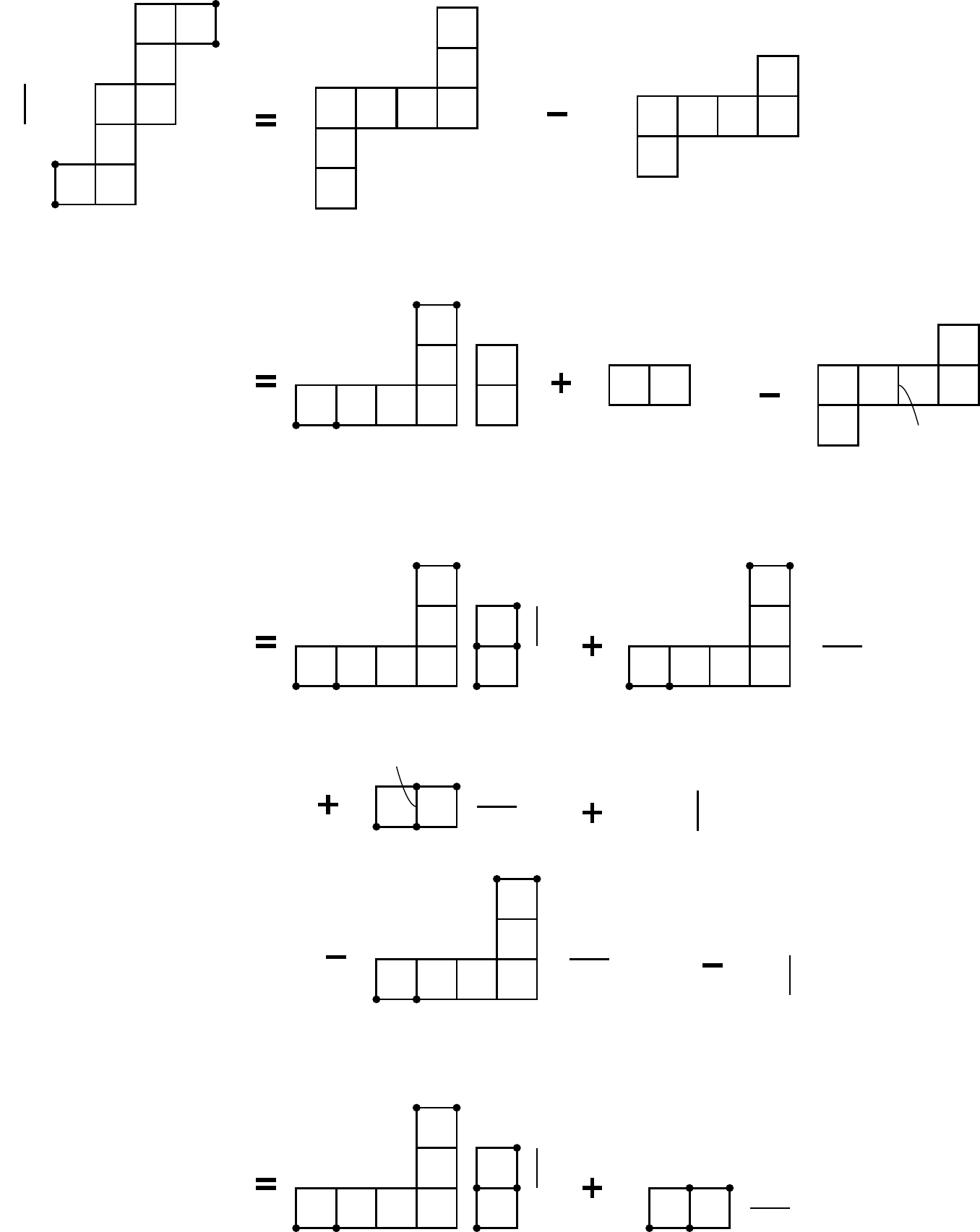}
\caption{The band graph version of the example of Figure \ref{proofsprimeistplusone}.}
\label{proofsprimeistplusoneband}
\end{center}
\end{figure}

Finally, let us take a look at the two  extreme cases $\calg_{56}=\emptyset $ if $d=t'+2\ell$ and $\calg_{56}=0 $ if $d<t'+2\ell$.
 The case $d=t'+2\ell$ can be visualised from the example in Figure \ref{proofsprimeistplusoneband} by removing the tiles with labels $1$ and 5. In this case the band graph $\band_{56} $ is the empty set, see Figure \ref{proofsprimeistplusonebandemptyset}.  In the case $d<t'+2\ell$, we have an even smaller band graph obtained by removing also the tile with label 2. This example is illustrated in Figure  \ref{proofsprimeistplusonebandzero}.
 \begin{figure}
\begin{center}
\scalebox{0.7}{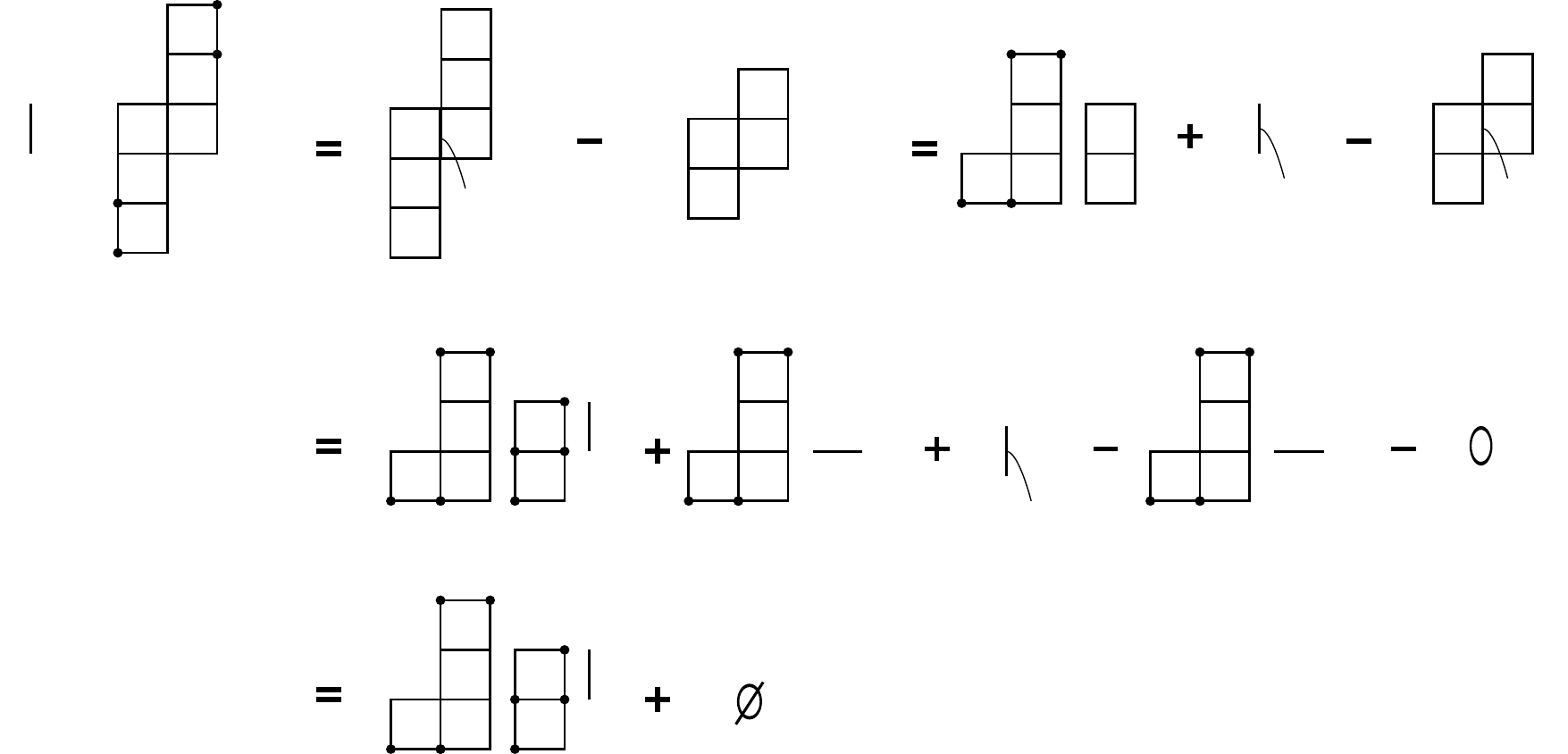}
\caption{The case $<t'+2\ell$. The empty set is obtained as the band graph of a single edge. Note that the empty set is the multiplicative identity of the snake ring, since the product of two graphs is the disjoint union. Thus $\emptyset=1$.}
\label{proofsprimeistplusonebandemptyset}
\end{center}
\end{figure}
 \begin{figure}
\begin{center}
\scalebox{0.7}{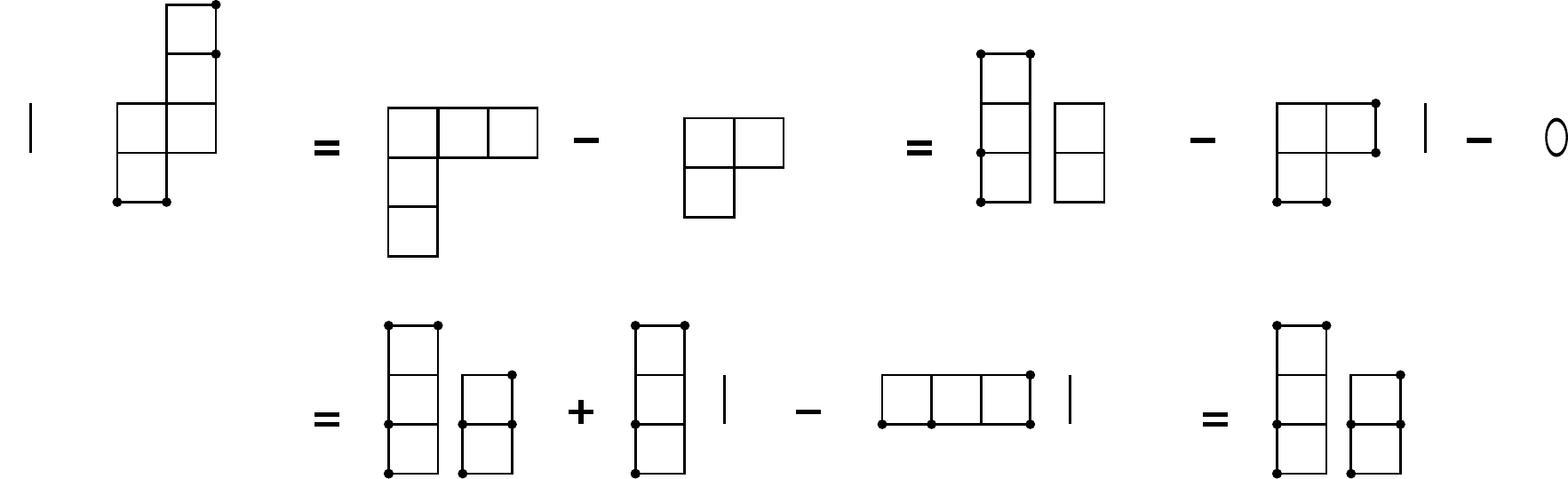}
\caption{The case $d=t'+2\ell$. Note that 0 is the additive identity of the snake ring.}
\label{proofsprimeistplusonebandzero}
\end{center}
\end{figure}
\end{proof}


\section{Labeled snake and band graphs arising from  cluster algebras of unpunctured surfaces}\label{sect 4}

In this section we recall how snake graphs {and band graphs} arise naturally in the theory of cluster algebras. We follow the exposition in \cite{MSW2}.

\subsection{Cluster algebras  from unpunctured
    surfaces}\label{sect surfaces} 
%

Let $S$ be a connected oriented 2-dimen\-sional Riemann surface with
nonempty
boundary, and let $M$ be a nonempty finite subset of the boundary of $S$, such that each boundary component of $S$ contains at least one point of $M$. The elements of $M$ are called {\it marked points}. The
pair $(S,M)$ is called a \emph{bordered surface with marked points}.
 
For technical reasons, we require that $(S,M)$ is not
a disk with 1, 2 or 3 marked points.

\begin{defn}   \label{gen-arc}
A \emph{generalised arc}  in $(S,M)$ is a curve $\gamma$ in $S$, considered up
to isotopy,  such that:
\begin{itemize}
\item[(a)] the endpoints of $\gamma$ are in $M$; 
\item[(b)] except for the endpoints,
$\gamma$ is disjoint  from the boundary of $S$;
\item[(c)]
$\gamma$ does not cut out a monogon or a bigon; and
\item[(d)] $\gamma$ has only finitely many self-crossings.
\end{itemize}

A  generalised arc $\zg$ is called an \emph{arc} if in  addition $\zg$ does not cross itself, except that its endpoints may coincide.

\end{defn}

Curves that connect two
marked points and lie entirely on the boundary of $S$ without passing
through a third marked point are \emph{boundary segments}.
Note that boundary segments are not   arcs.

For any two arcs $\zg,\zg'$ in $S$, let $e(\zg,\zg')$ be the minimal
number of crossings of 
arcs $\za$ and $\za'$, where $\za$ 
and $\za'$ range over all arcs isotopic to 
$\zg$ and $\zg'$, respectively.
We say that arcs $\zg$ and $\zg'$ are  \emph{compatible} if $e(\zg,\zg')=0$. 

A \emph{triangulation} is a maximal collection of
pairwise compatible arcs (together with all boundary segments). 

Triangulations are connected to each other by sequences of 
{\it flips}.  Each flip replaces a single arc $\gamma$ 
in a triangulation $T$ by a (unique) arc $\gamma' \neq \gamma$
that, together with the remaining arcs in $T$, forms a new 
triangulation.

\begin{defn}
Choose any   triangulation
$T$ of $(S,M)$, and let $\tau_1,\tau_2,\ldots,\tau_n$ be the $n$ arcs of
$T$.
For any triangle $\Delta$ in $T$, we define a matrix 
$B^\Delta=(b^\Delta_{ij})_{1\le i\le n, 1\le j\le n}$  as follows.
\begin{itemize}
\item $b_{ij}^\Delta=1$ and $b_{ji}^{\Delta}=-1$ if $\tau_i$ and $\tau_j$ are sides of 
  $\Delta$ with  $\tau_j$ following $\tau_i$  in the 
  clockwise order,
\item $b_{ij}^\Delta=0$ otherwise.
\end{itemize}
 
Then define the matrix 
$ B_{T}=(b_{ij})_{1\le i\le n, 1\le j\le n}$  by
$b_{ij}=\sum_\Delta b_{ij}^\Delta$, where the sum is taken over all
triangles in $T$.
\end{defn}

Note that  $B_{T}$ is skew-symmetric and each entry  $b_{ij}$ is either
$0,\pm 1$, or $\pm 2$, since every arc $\tau$ is in at most two triangles.

\begin{thm} \cite[Theorem 7.11]{FST} and \cite[Theorem 5.1]{FT}
\label{clust-surface}
Fix a bordered surface $(S,M)$ and let $\Acal$ be the cluster algebra associated to
the signed adjacency matrix of a   triangulation. Then the (unlabeled) seeds $\Sigma_{T}$ of $\Acal$ are in bijection
with  the triangulations $T$ of $(S,M)$, and
the cluster variables are  in bijection
with the arcs of $(S,M)$ (so we can denote each by
$x_{\gamma}$, where $\gamma$ is an arc). Moreover, each seed in $\Acal$ is uniquely determined by its cluster.  Furthermore,
if a   triangulation $T'$ is obtained from another
  triangulation $T$ by flipping an arc $\gamma\in T$
{to the arc} $\gamma'$,
then $\Sigma_{T'}$ is obtained from $\Sigma_{T}$ by the seed mutation
replacing $x_{\gamma}$ by $x_{\gamma'}$.
\end{thm}

{From now on suppose that $\cala$ has principal coefficients in the initial seed $\zS_T=(\mathbf{x}_T,\mathbf{y}_T,B_T)$.}


\begin{defn}
A \emph{closed loop} in $(S,M)$ is a closed curve
$\gamma$ in $S$ which is disjoint from the 
boundary of $S$.  We allow a closed loop to have a finite
number of self-crossings.
As in Definition~\ref{gen-arc}, we consider closed
loops up to isotopy.
%
A closed loop in $(S,M)$ is called \emph{essential} if
  it is not contractible
and it does not have self-crossings.
\end{defn}

\begin{defn} 
A \emph{multicurve} is  a finite multiset of generalised  
arcs and closed loops such that there are only a finite number of pairwise crossings among the collection.
We say that a multicurve is \emph{simple}
if there are no pairwise crossings among the collection and 
no self-crossings.
\end{defn}

If a multicurve is not simple, 
then there are two ways to \emph{resolve} a crossing to obtain a multicurve that no longer contains this crossing and has no additional crossings.  This process is known as \emph{smoothing}.

\begin{defn}\label{def:smoothing} (Smoothing) Let $\gamma, \gamma_1$ and $\gamma_2$ be generalised  
arcs or closed loops such that we have
one of the following two cases:

\begin{enumerate}
 \item $\gamma_1$ crosses $\gamma_2$ at a point $x$,
  \item $\gamma$ has a self-crossing at a point $x$.
\end{enumerate}

\noindent Then we let $C$ be the multicurve $\{\gamma_1,\gamma_2\}$ or $\{\gamma\}$ depending on which of the two cases we are in.  We define the \emph{smoothing of $C$ at the point $x$} to be the pair of multicurves $C_+ = \{\alpha_1,\alpha_2\}$ (resp. $\{\alpha\}$) and $C_- = \{\beta_1,\beta_2\}$ (resp. $\{\beta\}$).

Here, the multicurve $C_+$ (resp. $C_-$) is the same as $C$ except for the local change that replaces the crossing {\Large $\times$} with the pair of segments 
  $\genfrac{}{}{0pt}{5pt}{\displaystyle\smile}{\displaystyle\frown}$ (resp. {$\supset \subset$}).    
\end{defn}   

Since a multicurve  may contain only a finite number of crossings, by repeatedly applying smoothings, we can associate to any multicurve  a collection of simple multicurves.  
  We call this resulting multiset of multicurves the \emph{smooth resolution} of the multicurve $C$.

{\begin{remark}\label{rem smoothing}
 This smoothing operation can be rather complicated, since the multicurves are considered up to isotopy. Thus after performing the local operation of smoothing described above, one needs to find representatives of the isotopy classes of $C_+$ and $C_-$ which have a minimal number of crossings with the triangulation, 
 in order to obtain a good representation of the corresponding cluster algebra element that allows one to see compatibility with elements of the initial cluster and to compute its Laurent expansion, for example by constructing its snake graph.  In practice, this can be quite difficult especially if one needs to smooth several crossings. This difficulty was one of the original motivations to develop the snake graph calculus. The isotopy is already contained in the definition of the resolutions of the (self-)crossing snake {and band} graphs. 
\end{remark}}

\begin{thm}\label{th:skein1}(Skein relations) \cite[Propositions 6.4,6.5,6.6]{MW}
Let $C,$ $ C_{+}$, and $C_{-}$ be as in Definition \ref{def:smoothing}. 
Then we have the following identity in $\cala$,
\begin{equation*}
x_C = \pm Y_1 x_{C_+} \pm Y_2 x_{C_-},
\end{equation*}
where $Y_1$ and $Y_2$ are monomials in the variables $y_{\tau_i}$.    
The monomials $Y_1$ and $Y_2$ can be expressed using the intersection numbers of the elementary laminations (associated to triangulation $T$) with  
the curves in $C,C_+$ and $C_-$.
\end{thm}

\subsection{Labeled snake graphs from surfaces}\label{sect tiles}\label{sect graph}

 Let 
$\zg$ be an   arc in $(S,M)$ which is not in $T$. 
Choose an orientation on $\zg$, let $s\in M$ be its starting point, and let $t\in M$ be its endpoint.
We denote by
$s=p_0, p_1, p_2, \ldots, p_{d+1}=t$
the points of intersection of $\zg$ and $T$ in order.  
Let $\tau_{i_j}$ be the arc of $T$ containing $p_j$, and let 
$\zD_{j-1}$ and 
$\zD_{j}$ be the two   triangles in $T$ 
on either side of 
$\tau_{i_j}$. Note that each of these triangles has three distinct sides, but not necessarily three distinct vertices, see Figure \ref{figr1}.
\begin{figure}
\includegraphics{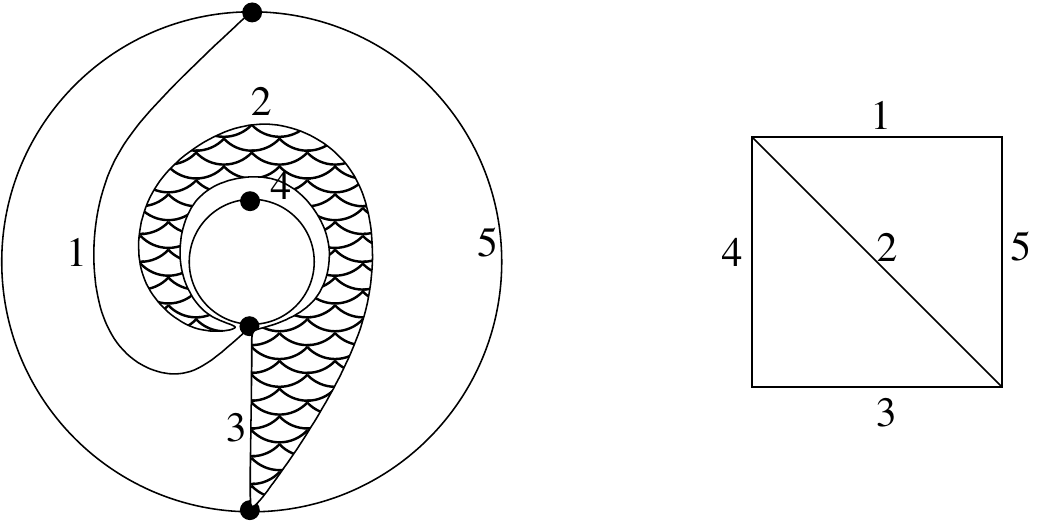}
\caption{On the left, a triangle with two vertices; on the right the tile $G_{j}$ where $i_j=2$. }\label{figr1}
\end{figure}
Let $G_j$ be the graph with 4 vertices and 5 edges, having the shape of a square with a diagonal, such that there is a bijection between the edges of $G_j$ and the 5 arcs in the two triangles $\zD_{j-1}$ and $\zD_j$, which preserves the signed adjacency of the arcs up to sign and such that the diagonal in $G_j$ corresponds to the arc $\tau_{i_j}$ containing the crossing point $p_j$. Thus $G_j$ is given by the quadrilateral in the triangulation $T$ whose diagonal is $\tau_{i_j}$. 

Given a planar embedding $\tilde G_j$ 
of $G_j$, we define the \emph{relative orientation} 
$\mathrm{rel}(\tilde G_j, T)$ 
of $\tilde G_j$ with respect to $T$ 
to be $\pm 1$, based on whether its triangles agree or disagree in orientation with those of $T$.  
For example, in Figure \ref{figr1},  $\tilde G_j$ has relative orientation $+1$.

Using the notation above, 
the arcs $\tau_{i_j}$ and $\tau_{i_{j+1}}$ form two edges of a triangle $\zD_j$ in $T$.  Define $\tau_{e_j}$ to be the third arc in this triangle.

We now recursively glue together the tiles $G_1,\dots,G_d$
in order from $1$ to $d$, so that for two adjacent   tiles, 
 we glue  $G_{j+1}$ to $\tilde G_j$ along the edge 
labeled $\tau_{e_j}$, choosing a planar embedding $\tilde G_{j+1}$ for $G_{j+1}$
so that $\mathrm{rel}(\tilde G_{j+1},T) \not= \mathrm{rel}(\tilde G_j,T).$  See Figure \ref{figglue}.

\begin{figure}
\begin{center}
 \scalebox{1}{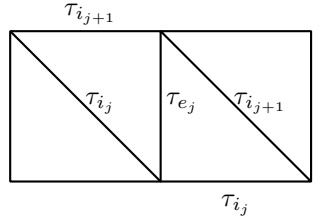}
\caption{Glueing tiles $\tilde G_j$ and $\tilde G_{j+1}$ along the edge labeled  $\tau_{e_j}$}
\label{figglue}
\end{center}
\end{figure}

After gluing together the $d$ tiles, we obtain a graph (embedded in 
the plane),
which we denote  by
${\calg}^\triangle_{\zg}$. 
\begin{defn}
The \emph{(labeled) snake graph} $\calg_{\gamma}$ associated to $\gamma$ is obtained 
from ${\calg}^\triangle_{\zg}$ by removing the diagonal in each tile.
\end{defn}

In Figure \ref{figsnake}, we give an example of an 
 arc $\gamma$ and the corresponding snake graph 
${\calg}_{\zg}$. Since $\gamma$ intersects $T$
five times,  
${\calg}_{\zg}$ has five tiles. 

\begin{remark}
 \label{rem sign}{
 Let $f$ be a sign function on $\calg_\zg$ as in section \ref{sect 2}. The interior edges $e_1,\ldots,e_{d-1}$  are corresponding to the sides of the triangles $\zD_1,\ldots,\zD_{d-1}$ that are not crossed by $\zg$. Two interior edges $e_j,e_k$ have the same sign $f(e_j)=f(e_k)$ if and only if the sides $\tau_{e_j},\tau_{e_k}$ lie on the same side of the segments of $\zg$ in $\zD_j$ and $\zD_k$, respectively.
 }
\end{remark}

 \begin{defn}If $\tau \in T$ then we define its (labeled) snake graph $\calg_{\tau}$ to be the graph consisting of one single edge with weight $x_{\tau}$ and two distinct endpoints (regardless whether the endpoints of $\tau$ are distinct).
 \end{defn}

\begin{figure}
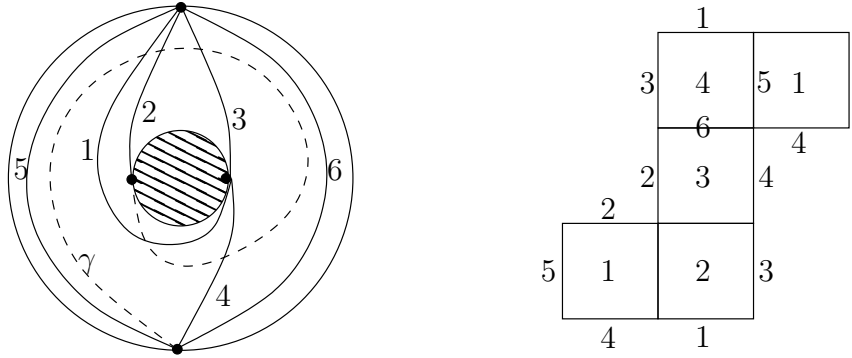
\caption{An arc $\gamma$ in a triangulated annulus on the left and the corresponding {labeled} snake graph $\calg_{\gamma}$ on the right. The tiles labeled 1, 3, 1 have positive relative orientation and the tiles 2, 4 have negative relative orientation.}\label{figsnake}
\end{figure}

Now we associate a similar graph to closed loops.
Let $\zeta$ be a closed loop in $(S,M)$, which may or may not have self-intersections, 
but which is not contractible and has no contractible kinks.  Choose
an orientation for $\zeta$, and a 
triangle $\Delta$ which is crossed by $\gamma$.  Let $p$ be a point in the interior of $\Delta$ which lies on $\gamma$, and let $b$ and $c$ be the two sides of the triangle crossed by $\gamma$
immediately before and following its travel through point $p$.
Let $a$ be the third side of $\Delta$.   
We let $\tilde{\gamma}$ denote the arc from $p$ back to itself that exactly
follows closed loop $\gamma$. 

\begin{figure}
\scalebox{1.3}
{\small
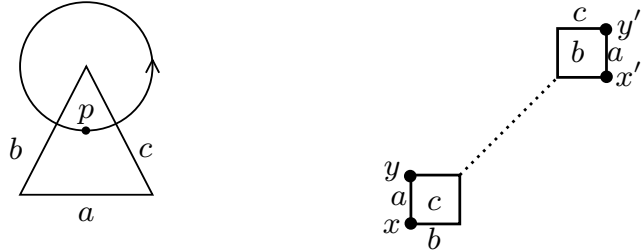
}
\caption{A triangle containing $p$ along a closed loop $\zeta$ (on the left) and the  corresponding band graph with $x\sim x'$, $y\sim y'$ (on the right). }
\label{fig band}
\end{figure}

We start by building the snake graph $\calg_{\tilde{\gamma}}$ 
as defined above.  In the first tile of $\calg_{\tilde{\gamma}}$,  
let $x$ denote the vertex
at the corner of the edge labeled $a$ and the edge labeled $b$, and let $y$ denote the vertex at the other end of the edge labeled $a$.  Similarly, in the last tile of
$G_{\tilde{\gamma}}$,  
let $x'$ denote the vertex at the corner of the edge labeled
$a$ and the edge labeled $b$, and let $y'$ denote the vertex at the other
end of the edge labeled $a$.  See the right of Figure \ref{fig band}.
Our convention for $x'$ and $y'$ are exactly opposite to those in \cite{MSW2}.

\begin{defn} \label{def band} The (labeled) \emph{band graph}  $\calg^\circ_{\zeta}$ associated to the loop $\zeta$
is the graph obtained from $\calg_{\tilde{\zeta}}$ by identifying the edges labeled $a$ in the first and last tiles so that the vertices $x$ and $x'$ and the vertices $y$ and $y'$ are glued together.  \end{defn}

\subsection{Snake graph formula for cluster variables}\label{secdefloop}
Recall that if $\tau$ is a boundary segment then $x_{\tau} = 1$.

If $\calg$ is a (labeled) snake graph and the edges of  a perfect matching $P$ of  
$\calg $ are labeled $\tau_{j_1},\dots,\tau_{j_r}$, then 
the {\it weight} $x(P)$ of $P$ is 
$x_{\tau_{j_1}} \dots x_{\tau_{j_r}}$.

Let $\zg$ be a generalised  arc and  $\tau_{i_1}, \tau_{i_2},\dots, \tau_{i_d}$
be the sequence of arcs in $T$ which $\zg$ crosses. The \emph{crossing monomial} 
of $\gamma$ with respect to $T$ is defined as
$$\mathrm{cross}(T, \gamma) = \prod_{j=1}^d x_{\tau_{i_j}}.$$
%
  

By induction on the number of tiles it is easy to see that the snake graph
$\calg_{\zg}$  
has  precisely two perfect matchings which we call
the {\it minimal matching} $P_-=P_-(\calg_{\zg})$ and 
the {\it maximal matching} $P_+
=P_+(\calg_{\zg})$, 
which contain only boundary edges.
To distinguish them, 
if  $\mathrm{rel}(\tilde G_1,T)=1$ (respectively, $-1$),
we define 
$e_1$ and $e_2$ to be the two edges of 
${\calg}^\triangle_{\zg}$ which lie in the counterclockwise 
(respectively, clockwise) direction from 
the diagonal of $\tilde G_1$.  Then  $P_-$ is defined as
the unique matching which contains only boundary 
edges and does not contain edges $e_1$ or $e_2$.  $P_+$
is the other matching with only boundary edges.
In the example of Figure \ref{figsnake}, the minimal matching $P_-$ contains the bottom edge of the first tile labeled 4.

\begin{lem}\cite[Theorem 5.1]{MS}
\label{thm y}
The symmetric difference $P_-\ominus P$ is the set of boundary edges of a 
(possibly disconnected) subgraph $\calg_P$ of $\calg_\zg$,
which is a union of cycles.  These cycles enclose a set of tiles 
$\cup_{j\in J} G_{j}$,  where $J$ is a finite index set.
\end{lem}

\begin{defn} \label{height} With the notation of Lemma \ref{thm y},
we  define the \emph{height monomial} $y(P)$ of a perfect matching $P$ of a snake graph $\calg_\gamma$ by
\begin{equation*}
y(P) = \prod_{j\in J} y_{\tau_{i_j}}.
\end{equation*}\end{defn}

Following \cite{MSW2}, for each generalised arc $\gamma$, we now define a Laurent polynomial $x_\zg$, as well as a polynomial 
$F_\zg^T$ obtained from $x_{\gamma}$ by specialization.
\begin{defn} \label{def:matching}
Let $\zg$ 
be a generalised arc and let $\calg_\zg$
be its snake graph.  
\begin{enumerate}
\item If $\zg$ 
has a contractible kink, let $\overline{\zg}$ denote the 
corresponding generalised arc with this kink removed, and define 
$x_{\zg} = (-1) x_{\overline{\zg}}$.  
\item Otherwise, define
\[ x_{\gamma}= \frac{1}{\mathrm{cross}(T,\zg)} \sum_P 
x(P) y(P),\]
 where the sum is over all perfect matchings $P$ of $G_{\zg}$.
\end{enumerate}
Define $F_{\gamma}^T$ to be the polynomial obtained from 
$x_{\gamma}$ by specializing all the $x_{\tau_i}$ to $1$.

If $\gamma$ is a curve that 
cuts out a contractible monogon, then we define $\gamma =0$.     
\end{defn}

\begin{thm}\cite[Thm 4.9]{MSW}
\label{thm MSW}
If $\gamma$ is an arc, then 
$x_{\gamma}$ 
is a the cluster variable in $ \A$,
written as a Laurent expansion with respect to the seed $\Sigma_T$,
and $F_{\gamma}^T$ is its \emph{F-polynomial}.
\end{thm}

 Again following \cite{MSW2}, we define for every closed loop $\zeta$, a Laurent polynomial $x_\zeta$, as well 
as a polynomial 
$F_\zeta^T$ obtained from $x_{\zeta}$ by specialization.
\begin{defn} \label{def closed loop} 
Let $\zeta$  
be a closed loop.  
\begin{enumerate}
\item If $\zeta$ is a contractible loop,
 then let $x_\zeta = -2$.
\item If $\zeta$ 
has a contractible kink, let $\overline{\zeta}$ denote the 
corresponding closed loop with this kink removed, and define 
$x_{\zeta} = (-1) x_{\overline{\zeta}}$.
\item Otherwise, let 
$$x_{\zeta} = \frac{1}{\mathrm{cross}(T,\zg)} \sum_{P} 
x(P) y(P),$$ where the sum is over all good matchings $P$ of the  band graph $\calg^\circ_{\zeta}$. 
\end{enumerate}
Define $F_\zeta^T$ to be the Laurent polynomial obtained from 
$x_\zeta$ by specializing all the $x_{\tau_i}$ to $1$.
\end{defn}

\subsection{Bases of the cluster algebra}\label{sec:bangbrac}
{We recall the construction of the two bases given in \cite{MSW2} in terms of bangles and bracelets.}

\begin{defn}
Let $\zeta$ be an essential loop in $(S,M)$.  
The 
\emph{bangle} $\Bang_k \zeta$ is the union of $k$ loops isotopic to $\zeta$.
(Note that $\Bang_k \zeta$ has no self-crossings.)  And the 
\emph{bracelet} $\Brac_k \zeta$ is the closed loop obtained by 
concatenating $\zeta$ exactly $k$ times, see Figure \ref{figbangbrac}.  (Note that it will have $k-1$ self-crossings.)
\end{defn}

\begin{figure}
\scalebox{0.8}{\includegraphics{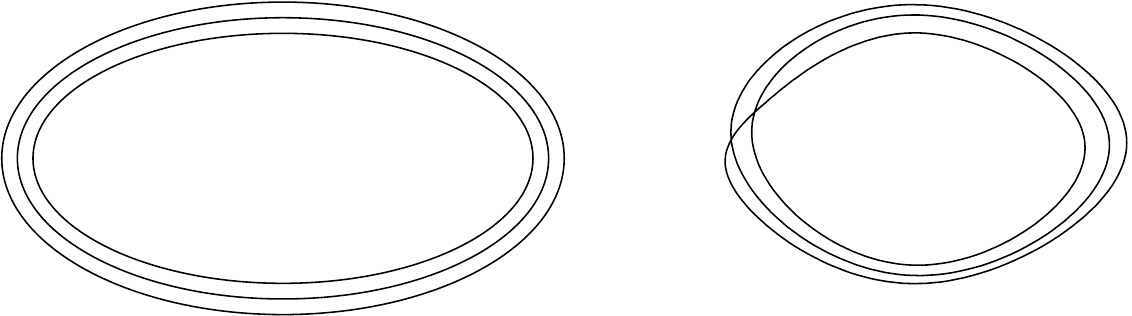}}
\caption{A bangle $\Bang_3 \zeta$, on the left, and a bracelet $\Brac_3 \zeta$, on the right.}\label{figbangbrac}
\end{figure}
Note that $\Bang_1 \zeta = \Brac_1 \zeta = \zeta$.

\begin{defn}
\label{def C0-compatible}
A collection $C$ of arcs and essential loops is called 
\emph{$\C^{\circ}$-compatible} if no two elements of $C$ cross each other.
Let $\C^{\circ}(S,M)$ be the set of all 
$\C^{\circ}$-compatible collections in $(S,M)$.
\end{defn}

\begin{defn}
A collection $C$ of arcs and bracelets is called 
\emph{$\C$-compatible} if:
\begin{itemize}
\item no two elements of $C$ cross each other except for the self-crossings of a bracelet; and
\item given an essential loop $\zeta$ in $(S,M)$, 
there is at most one $k\ge 1$ such
that the $k$-th bracelet $\Brac_k\zeta$ lies in $C$, and, moreover, there is at
most one copy of this bracelet $\Brac_k\zeta$ in $C$.
\end{itemize}
Let $\C(S,M)$  be the set of all $\C$-compatible
collections in $(S,M)$.
\end{defn}

Note that a $\C^{\circ}$-compatible collection may contain 
bangles $\Bang_k \zeta$ for $k \geq 1$, but it will not contain
bracelets $\Brac_k \zeta$ except when $k=1$.
And  a $\C$-compatible collection may contain bracelets, but will never
contain a bangle $\Bang_k \zeta$ except when $k=1$.

\begin{defn}
Given an arc or closed loop $c$, let 
$x_c$ denote the corresponding Laurent polynomial
defined in Section \ref{secdefloop}.
Let $\B^\circ$ 
be the set of all cluster algebra 
elements corresponding to the set $C^{\circ}(S,M)$, 
\[\B^{\circ} = \left\{\prod_{c\in C} x_c \ \vert \ C \in \C^{\circ}(S,M) \right\}.\]
Similarly, let
\[\B = \left\{\prod_{c\in C} x_c \ \vert \ C \in \C(S,M) \right\}.\]
\end{defn}

\begin{remark}
Both $\B^{\circ} $ and $\B $ contain
the cluster monomials of $\A $.
\end{remark}

We are now ready to state the main result of \cite{MSW2}.

\begin{thm}\cite[Theorem 4.1]{MSW2} If the surface has no punctures and at least two marked points then
the sets $\B^\circ$ and $\B$ are bases of the cluster algebra $\A$.                             
\end{thm}

\begin{remark}
 This result has been extended to surfaces with only one marked point in \cite{CLS}.
\end{remark}

\section{Relation to cluster algebras}\label{sect 5} 
In this   section, we show how our results on abstract snake {and band} graphs are related to computations in cluster algebras from unpunctured surfaces. For these cluster algebras,  each cluster variable can be computed using its labeled snake graph. 
Yet another way of representing a cluster variable is by an arc in the surface. We show that two arcs cross if and only if {their} corresponding labeled snake graphs cross, and that the smoothing of the crossing arcs corresponds to the resolution of the crossing labeled snake graphs. As a consequence, two cluster variables are compatible if and only if their corresponding labeled snake graphs do not cross.

\subsection{Crossing curves and crossing graphs}
In this subsection, we show that the notion of crossings for arcs and loops and the notion of crossings for snake graphs and band graphs coincide.
\begin{thm}\label{thm cross} Let $\gamma,\gamma_1, \gamma_2$ be each a generalised arc or a closed loop and let $\calg_\gamma,\calg_1, \calg_2$ the corresponding {labeled} snake graphs or band graphs. 
\begin{itemize}
\item[a)] $\gamma_1, \gamma_2$ cross with a nonempty local overlap $(\tau_{i_s}, \cdots, \tau_{i_t})=(\tau_{i'_{s'}}, \cdots, \tau_{i'_{t'}})$ if and only if $\calg_1, \calg_2$ cross in the corresponding overlap.
\item[b)] $\gamma$ has a self-crossing with a nonempty local overlap $(\tau_{i_s}, \cdots, \tau_{i_t})=(\tau_{i_{s'}}, \cdots, \tau_{i_{t'}})$ if and only if $\calg_{\gamma}$ has a self-crossing in the corresponding overlap.
\end{itemize}
\end{thm}

\begin{proof} 
a) Since crossing is a local condition, this statement follows directly from Theorem 5.3 of \cite{CS}, except for the case where we have two labelled band graphs which are isomorphic and the overlap is isomorphic to both of them. In this case, our Definition \ref{defcrossbands} implies that the band graphs do not cross in this overlap. For the  corresponding closed loops $\zg_1,\zg_2$ we can choose starting points and orientations such that the sequence of crossed arcs of the triangulation is the same for both, and this sequence corresponds to the overlap under consideration. Therefore along this overlap the two loops are isotopic to each other and do not cross. See the right hand side of Figure~\ref{fig28} for an example.

b) For self-crossing arcs, this has been shown in \cite[Theorem 6.1]{CS2}.  Therefore let us assume that $\gamma$ is a loop.

First suppose that $\gamma $ is not a bracelet. 
Choose a parametrization $\zg=\zg(t)$, and say the self-crossing occurs at the times $t_1$ and $t_2.$ Take now two copies $\gamma_1, \gamma_2$ of $\gamma$ and consider their crossing at $\gamma_1(t_1)=\gamma_2(t_2).$ The left picture in Figure~\ref{fig28} shows an example of a figure 8 loop on the torus. The local overlap of $\gamma_1$ and $\gamma_2$ at this crossing is the same as the local overlap of the self-crossing, and the local overlap in the corresponding snake graphs $\calg_1, \calg_2$ of $\gamma_1, \gamma_2$ is the same as the local self-overlap of the snake graph $\calg_{\gamma}.$ Now the result follows from part a).

Finally, suppose that $\gamma $ is a bracelet, {and} $\zg=\Brac_k(\zeta)$ {for some loop $\zeta$}. 
If $\gamma $ has a self-crossing such that the corresponding overlap is not the whole band graph then the same argument as above applies. 
Therefore assume that $\gamma $ has a self-crossing at a point $p$ such that the corresponding overlap  is the whole band graph. In this case,
we cannot use the same argument, because two copies of the same bracelet do not always cross. For example if $\zeta$ is a simple loop, then two copies of $\Brac_k(\zeta)$ do not cross each other because using isotopy we can separate one from the other. See  the right hand side of Figure \ref{fig28} for an example on the torus.
\begin{figure}
\scalebox{.6}{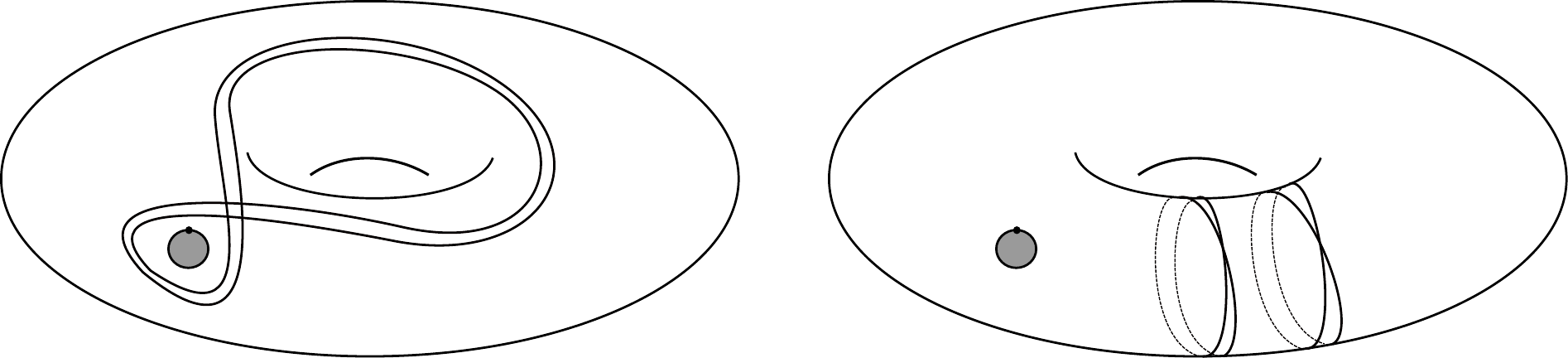}
\caption{Two copies of a figure 8 on the torus crossing each other twice (left); two copies of a bracelet on the torus not crossing each other (right).}\label{fig28}
\end{figure}
Consider the loop $\zg$ as a curve starting and ending at $p$.  
Let $i_1,i_2,\ldots,i_t$ be the sequence of crossing points of the loop $\zg$ with the triangulation in the order determined by  $\zg$, such that the point $i_j$ lies on the arc $\tau_{i_j}$ of the triangulation. Then there exist a unique $s'>1$ such that the overlap is given by two sequences $\tau_{i_1}=\tau_{i_{s'}}$, $\tau_{i_2}=\tau_{i_{s'+1}}$, \ldots, $\tau_{i_{t}}=\tau_{i_{s'-1}}$. Now consider the corresponding labeled band graph $\band$ cut to a snake graph $\calg=\band_b$ in such a way that the first tile of $\calg$ corresponds to the crossing point $i_1$.  Let $f$ be a sign function on $\calg$.  Recall from Remark~\ref{rem sign} that for an interior edge $e_j$ of $\band$,
 the sign $f(e_j)$ is determined by the way that the segment of $\gamma$ between the points $i_j$ and $i_{j+1}$ runs through the triangle formed by $\tau_{i_j}, \tau_{i_{j+1}}$ and $e_j$. In particular, $f(e_t)=f(e_{s'-1})$, since $s'-1=t'$, $\tau_{i_t}=\tau_{i_{t'}}$ and the two segments of $\zg$ from $\tau_{i_t}$ to $\tau_{i_s}$ and $\tau_{i_{t'}}$ to $\tau_{i_{s'}}$ run through the triangle formed by $\tau_{i_t}, \tau_{i_{s}}$ and $e_t$ in the same direction and crossing the same sides. 
Therefore Definition \ref{def self-crossing band} implies that the $\band$ has a self-crossing overlap.

For the reverse implication, suppose we have a  labeled band  $\band=\calg^b$ coming from a surface without punctures such that $\band$ has a self-crossing overlap $\band[s,t]\cong\band[s',t']$ which is the whole band graph. Then the corresponding loop $\zg$ in the surface is crossing the arcs $\tau_{i_{s}},\ldots,\tau_{i_t}$ and this sequence is equal to the sequence $\tau_{i_{s'}},\ldots,\tau_{i_{t'}}$. Using the notation $\zg[j,k]$ for the segment of $\zg$ from the point $i_j$ to $i_k$, we see that the segments $\zg[s,s'-1], \zg[s',2s'-s-1], \zg[2s'-s,3s'-2s-1],\ldots$ run parallel in the surface. So either $\zg$ is a bracelet of the loop corresponding to the segment $\zg[s,s'-1]$ or $\zg$ is a bracelet of a loop corresponding to a segment $\zg[s,s''-1]$ with $s''<s'$. In both cases $\zg$ has a self-crossing with overlap $\tau_{i_{s}},\ldots,\tau_{i_t}$ and $\tau_{i_{s'}},\ldots,\tau_{i_{t'}}$.
\end{proof}

\subsection{Smoothing crossings and resolving snake graphs}
{In this subsection, we show that the smoothing operation for arcs corresponds to the resolution of crossings for snake graphs and band graphs.}

\begin{thm} \label{smoothing1}  Let $\gamma_1, \gamma_2$ be each a generalised arc or a closed loop which cross with a non-empty local overlap, and let $\calg_1, \calg_2$ the corresponding {labeled} snake graphs or band graphs with overlap $\calg$. Then the {labeled} snake graphs and band graphs of the arcs and loops obtained by smoothing the crossing of $\gamma_1$ and $\gamma_2$ in the overlap are given by the resolution $\re 12$ of the crossing of the   $\calg_1$ and $\calg_2$ at the overlap $\calg.$
\end{thm}
\begin{proof}
 In the case where $\zg_1$ and $\zg_2$ are arcs, this is 
\cite[Theorem 5.4]{CS}. If one of the curves $\zg_1,\zg_2$ is a loop then we can prove the result by introducing a puncture on the curve and then using the result for arcs and then removing the puncture again. We do not include the details here because  we are using this technique in the proof of the following Theorem for selfcrossing loops, which is the more interesting case.
\end{proof}

\begin{thm}\label{smoothing2} Let $\gamma_1$ be a self-crossing arc or loop with nonempty local overlap and let $\calg_1$ be the corresponding {labeled} snake graph or band graph with crossing overlap $i_1(\calg)=\calg_1[s,t]$ and $i_2(\calg)=\calg_1[s',t'].$ Then the {labeled} snake graphs and band graphs of the arcs and loops obtained by smoothing the crossing of $\gamma_1$ on the overlap are given by the resolution $\res_{\calg}(\calg_1)$ of the self-crossing of the   $\calg_1$ at the overlap $\calg.$
\end{thm}

\begin{proof} If $\zg_1$ is an arc, this  is \cite[Theorem 6.3]{CS2}. So let $\zg_1$ be a loop and let $\band_1$ be its band graph.
As usual we may choose a snake graph $\calg_1=(\band_1)_b=(G_1,G_2,\ldots,G_d)$ such that  the two overlaps are given by
 \[ i_1(\calg) =\calg_1[1,t] \quad i_2(\calg) =\calg_1[s',t'].\]

We start with the case where the overlap is in the same direction and $s'\ge t+1$. 
Let $\Delta$ be the triangle 
in the surface which contains the segment of $\gamma_1$ between the $t$-th and the $(t+1)$-st crossing point. We introduce a puncture $p$ on this segment of $\zg_1$ and in the interior of $\zD$, see Figure~\ref{fig puncture}, {where the triangulation arcs are black and the arc $\zg$ is red.} We complete $T$  to a triangulation $\dot T$ by adding three arcs $a,b,c$ from the puncture $p$ to the vertices of $\zD$. 
Let $\gamma_{11}$ be the segment of $\gamma_1$ up to the puncture $p$, and let $ \gamma_{12}$ the segment of $\gamma_1$ after   $p$. 
\begin{figure}\begin{center}
\scalebox{0.8}{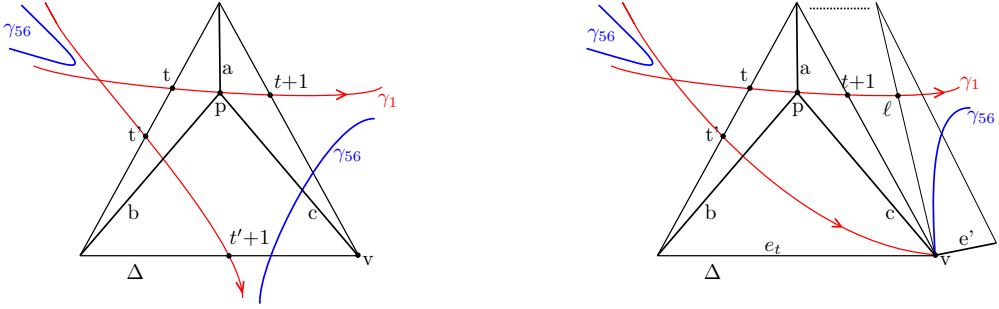}
\caption{Introducing a puncture}\label{fig puncture}\end{center}
\end{figure}

On the other hand, consider the snake graphs $\calg_{11}=\calg_1[1,t]$ and $\calg_{12}=\calg_1[t+1,d].$
Observe that these snake graphs do not necessarily correspond to snake graphs of $\gamma_{11}$ and $\gamma_{12}$, since the arc $\gamma$ might run through the triangle $\Delta$ several times, and introducing a puncture in the surface might create crossings with the new arcs in $\dot T \backslash T$. Then the snake graphs $\widetilde{\calg}_{11}, \widetilde{\calg}_{12}$ corresponding to $\gamma_{11}$ and $\gamma_{12}$ will be obtained from $\calg_{11}$ and $\calg_{12},$ respectively, by inserting single tiles which correspond to these new crossings.

The  triangle $\zD$ has sides  $\tau_{i_t}=\tau_{i_{t'}},\tau_{i_{t+1}},\tau_{i_{t'+1}}$ and $i_{t+1}\ne i_{t'+1}$, since we have an overlap.
Smoothing the self-crossing of $\zg_1$ is the same as smoothing the corresponding crossing of the two curves $\zg_{11}$ and $\zg_{12}$ and then removing the puncture again. This will produce a pair of loops $(\zg_3,\zg_4)$,  as well as a loop $\zg_{56}$ which crosses $\tau_{i_{t+1}}$ and $\tau_{i_{t'+1}}$. In Figure~\ref{fig puncture}, the loop $\zg_{56}$ is the blue one. In terms of band graphs, $\res_{\calg}(\band_1)$ is obtained from $\res_{\calg}(\widetilde{\calg}_{11}, \widetilde{\calg}_{12} )$ by removing the tiles corresponding to the crossings with the arcs at the puncture and glueing. Each of the graphs $\band_3$ and $\band_4$ is glued along the edge labeled $\tau_{i_{t'+1}}$, and they correspond to $\zg_3$ and $\zg_4$. The graph $\band_{56}$ is glued along the edge labeled $\tau_{i_t}$ and corresponds  to $\zg_{56}$.

When the overlap is in the opposite direction, the proof is similar.

Now consider  the case where the overlap is in the same direction and $s'\le t$. Thus the self-overlap has an intersection $\calg_1[s',t]$.

Let $\tau_{i_1},\tau_{i_{2}},\ldots,\tau_{i_s}$ be the sequence of arcs of the triangulation crossed by $\gamma_1$ in order. 
By the definition of overlap, we have that the sequences $\tau_{i_s},\tau_{i_{s+1}},\ldots,\tau_{i_t}$ and
$\tau_{i_{s'}},\tau_{i_{s'+1}},\ldots,\tau_{i_{t'}}$
are  equal or opposite to each other. Since $s'<t$ it follows that they have to be equal, {because} otherwise the segment of $\gamma$ crossing $\tau_{i_{s'+1}},\tau_{i_{s'+2}},\ldots,\tau_{i_{t}}$ would be isotopic to a curve not crossing these arcs at all, see Figure \ref{homotopy2}.

\begin{figure}\begin{center}
\scalebox{1}{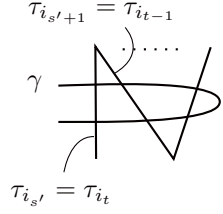}
\caption{Proof of Theorem \ref{smoothing2}}\label{homotopy2}\end{center}
\end{figure}

Let $\gamma[s,t]$ and $\gamma[s',t']$ be the segments of $\gamma$ corresponding to the overlaps and let $\gamma[s',t]$ be their common subsegment corresponding to the intersection of the overlaps.

Since the 
sequences $\tau_{i_s},\tau_{i_{s+1}},\ldots,\tau_{i_t}$ and
$\tau_{i_{s'}},\tau_{i_{s'+1}},\ldots,\tau_{i_{t'}}$
are  equal, it follows that the curves $\gamma[s,t]$ and $\gamma[s',t']$ run parallel before and after their crossing at $p$. Moreover, since $s'<t$, the following sequences  are equal as well:
\[\begin{array}
 {ll}
 \tau_{i_s},\tau_{i_{s+1}},\ldots,\tau_{i_{s'-1}}\\
 \tau_{i_{s'}},\tau_{i_{s'+1}},\ldots,\tau_{i_{2s'-s-1}}\\
 \tau_{i_{2s'-s}},\tau_{i_{2s'-s+1}},\ldots,\tau_{i_{3s'-2s-1}}\\
 \ldots \tau_{i_{t}}\\
\end{array}
\]
Thus the curve $\gamma[s,s']$ after identifying its endpoints is a closed and non-contractible curve. This implies that the curve $\gamma[s,t]$ is of the form as in Figure \ref{fig shape}, where points with equal labels $a,b,c,d,e$ are identified. The crossing point $p$ can be any of the points labeled 1,2,3,4,5. For example, if $p$ is the point labeled 5,4,3,2,1 respectively, then the crossing point $s'$ at the beginning of the second overlap must be the point on $\tau_{i_s}$ crossed by $\gamma$ after the point $e,d,c,b,a$ respectively, and the crossing point $t$  at the end of the first overlap must be  the point on $\tau_{i_{t'}}$ first crossed by $\gamma$ after passing through the point $a,b,c,d,e$ respectively.

\begin{figure}\begin{center}
\scalebox{1}{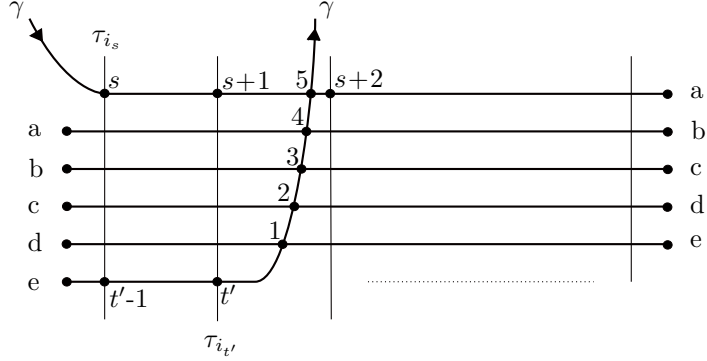}
\caption{Proof of Theorem \ref{smoothing2}.
This situation can arise in any non-simply connected surface. The curve $\gamma$ between the points 5 and 5 is the concatenation of an essential loop $\zeta$ with itself 5 times, thus $\textup{Brac}_5\zeta$.}\label{fig shape}\end{center}
\end{figure}

Therefore the condition $s'\le t$ implies that in the example in Figure \ref{fig shape} the point $p$ must be  the point 2 or 1.
We now study the smoothing of these self-crossings.

If $p=1$  then the smoothing at $p$ will produce the two multicurves $$\{\zeta, \textup{loop with 4 self-crossings at 2,3,4,5}\},$$
 and
  $$\{ \textup{loop with  self-crossings at 3,4,5 and a kink at 2} \}.$$

If $p=2$ then the smoothing at $p$ will produce the two multicurves 
$$\{\textup{Brac}_1\zeta, \textup{loop with 3 self-crossings at 3,4,5}\},$$
 and
  $$\{ \textup{loop with  self-crossings at 4,5 and a kink at 3} \}.$$
  
Again we conclude that the band graphs of the loops obtained by smoothing the self-crossing of $\gamma$ are given by $\res_{\calg}(\calg_1).$

Finally, consider the case where the overlap is the whole loop $\zg_1$, thus $s=1$ and $t=d$. 
As shown in section \ref{sect 3.3}, we can assume that $2s'-2\le d$, because we can exchange the roles of $s$ and $s'$ if necessary. Also assume without loss of generality that our crossing point $p$ lies in the triangle $\zD$ with sides $\tau_{i_1}=\tau_{i_{s'}}$, $\tau_{i_d}=\tau_{i_{t'}}$ and a third side $\tau_a$. 
Then $\zg_1$ crosses the arc $\tau_{i_1}$ at least twice and the arc $\tau_{i_d}$ also at least twice. If $i_{\tilde s}$ is a crossing of $\zg_1$ and $\tau_{i_1}$ and $i_{\tilde t}$ is a crossing of $\zg$ and $\tau_{i_d}$, we denote by $\zg_1[\tilde s, \tilde t]$ the loop obtained by starting at the point $p$ and leaving the triangle $\zD$ through the point $i_{\tilde s}$ on $\tau_{i_1}$, then following $\zg_1$ until the point $i_{\tilde t}$ on $\tau_{i_d}$ and then entering the triangle $\zD$ at this point and ending at $p$.

Thus $\zg_1[s,t]=\zg_1$ is the loop at   $p$ following the first overlap, and $\zg_1[s',t']$ is the loop at   $p$  following the second overlap. Then both curves  $\zg_1[s,t]$ and $\zg_1[s',t']$ are isotopic. 
Moreover there exists a loop $\zeta= \zg[s,t'']$, for some $t''$ with $i_{t''}$ a crossing point of $\zg_1$ and $\tau_{i_d}$, such that $\zg_1$ is a $k$-bracelet of $\zeta$, for some $k,$ {namely}
\[ \zg_1=\Brac_k\zeta.\]
Note that $\zeta $ may have self-crossings.
In this case smoothing the crossing at $p$ will produce on the one hand the pair of loops $(\zg_3,\zg_4)$  given by 
\[\zg_3=\zg_1[s',t], \quad \zg_4=\zg_1[s,t'],\]
and on the other hand the loop  $\zg_{56}=$  starting at $p$ following $\zg_1$ and crossing $\tau_{i_1},\tau_{i_2},\ldots,\tau_{i_{t'}}$ coming to the point $p$ again, then following $\zg_1^{-1}$ and crossing $\tau_{i_t},\tau_{i_{t-1}},\ldots,\tau_{i_{s'}}$ and then ending at $p$.   
Since $p$ is not a marked point, there is an isotopy that removes the pairs of crossing points 
 $({i_{s'}}, i_1)$, $({i_{s'+1}}, i_{2}),\cdots ( i_{s'+t'-1},{i_{t'}}) $. Using this isotopy and the fact that $t'=s'-1$ we get that the loop $\zg_{56} $ has a kink and 
is crossing the arcs 
$\tau_{2s'-1},\tau_{2s'},\tau_{2s'+1},\ldots \tau_{i_t}$. In particular, in the extreme case where $2s'-2=d=t$, we see that $\zg_{56}$ does not cross any arc of the triangulation.

This shows that the corresponding band graphs $\band_3,\band_4$ and $\band_{56}$ are precisely those given in the definition of the resolution in section \ref{sect 3.3}
\[\calg_3^\circ = (\calg_1[s',d])^{b'},
\calg_4^\circ= (\calg_1[1,s'-1])^{c}
\]
\[
  \calg_{56}^\circ =\left\{\begin{array}{ll}
  -(\calg_1[2s'-1,d])^{b'}, &\textup{if }2s'-1\le d; \\
  -2, &\textup{if }2s'-2= d. \end{array}\right.\]
This completes the proof.
\end{proof}

So far, we have considered crossings with a non-empty local overlap. Now we study  crossings with an empty local overlap. Such a crossing point must lie in some triangle in the triangulation, and the condition that the overlap is empty means that the crossing curves do not cross the same side of the triangle. In particular at least one of the two curves must end in this triangle. Therefore we cannot get this situation for the crossing of two loops or the self-crossing of a single loop.

For the other cases we have the following results.
\begin{thm} \label{smoothing2} Let $\gamma_1$ be a generalised arc and $\gamma_2$ be an arc or a loop, such that $\zg_1$ and $\zg_2$ cross in a triangle $\Delta$ with an empty local overlap, and let $\calg_1$ and $\calg_2$ be the corresponding snake or band graphs. Assume $\Delta=\Delta'_0$ is the first triangle $\gamma_1$ meets. Then the snake graphs of the  arcs and loops obtained by smoothing the crossing of $\gamma_1$ and $\gamma_2$ in $\Delta$ are given by the resolution $\gt s{\delta_3}12$ of the grafting of $\calg_2$ on $\calg_1$ in $G_s,$ where $0 \leq s \leq d$ is such that $\Delta=\Delta_s$ and if $s=0$ or $s=d$ then $\delta_3$ is the unique side of $\Delta$ that is not crossed by neither $\gamma_1$ nor $\gamma_2.$
\end{thm}

\begin{proof} If $\zg_2$ is an arc, this is \cite[Theorem 5.7]{CS}. On the other hand, if $\zg$ is a loop then 
as in the proof of Theorem \ref{smoothing2}, we can introduce a puncture on the segment  of $\gamma_2$
between the two crossing points and complete to a triangulation. Then the two segments of $\gamma_1$ before and after the puncture still have the same crossing. We can use Theorem \ref{smoothing2} to resolve that crossing and then remove the puncture to get the desired resolution.
\end{proof}

For self-crossing arcs with empty local overlap, we have the following result from \cite{CS2}.

\begin{thm}\label{smoothingselfgrafting} \cite[Theorem 6.5]{CS2} Let $\gamma_1$  be a generalised arc which has a self-crossing in a triangle $\Delta$ with an empty local overlap, and let $\calg_1$ be the corresponding snake graph. Thus $\Delta=\Delta_0$ is the first triangle $\gamma_1$ meets and  $\Delta=\Delta_s$ is met again after $s$ crossings. Then the snake graphs of the {two arcs and the band graph of the loop} obtained by smoothing the self-crossing of $\gamma_1$ in $\Delta$ are given by the resolution $\graft_{s,\e_3} (\calg_1)$ of the self-grafting of $\calg_1$  in $G_s,$ and if  $s=d$ then $\delta_3$ is the unique side of $\Delta$ that is not crossed by  {$\gamma_1$.}
\end{thm}


\subsection{Snake graph calculus for cluster algebras}\label{sect 6.4}  In this section, we show that  snake graph calculus can be used to make explicit computations in  cluster algebras from unpunctured surfaces. Recall that cluster algebras are generated by cluster variables. Now
 since cluster variables correspond to snake graphs, it follows that arbitrary elements of the cluster algebra correspond to linear combinations of monomials of snake graphs, where the product of snake graphs is given by disjoint union. We show that  the resolution of each crossing or self-crossing of the snake graphs in such a monomial  corresponds to an identity in the cluster algebra. 
 In particular, we give a new proof of the skein relations. Resolving all crossings in the snake graph monomials corresponds to expressing the element of the cluster algebra as a linear combination of the bangles basis of \cite{MSW2}. We distinguish two kinds of crossings depending on whether the overlap is empty or not.

\subsubsection{Non-empty overlaps}

If $\calg$ is a snake graph associated to an arc $\gamma$ in a triangulated surface $(S,M,T)$ then each tile of $\calg$ corresponds to a quadrilateral in the triangulation $T,$ and we denote by $\tau_{i(G)}\in T$ the diagonal of that quadrilateral. With this notation we define
\begin{align*}
 x(\calg) = \prod_{ G \mbox{ tile in } \calg} x_{i(G)}\\
 y(\calg) = \prod_{ G \mbox{ tile in } \calg} y_{i(G)}
\end{align*}
If $\calg=\{\tau\}$ consists of a single edge, we let  $x(\calg)=1$ and $y(\calg)=1$.

Let $\gamma_1$ and $\gamma_2$ be arcs or loops which cross with a non-empty  overlap. Let $x_{\gamma_1}$ and $x_{\gamma_2}$ be the corresponding Laurent polynomials and $\calg_1$ and $\calg_2$ be the snake or band graphs with corresponding overlap $\calg.$ 
Denote by $\calg_{34},\calg_{56}$ the elements of $\mathcal{R}$ given by the resolution $\calg_1\calg_2=\calg_{34}+\calg_{56}$, such that 
the number of tiles in $\calg_{34}$ is equal to the number of tiles in $\calg_1\sqcup\calg_2$, whereas the number of tiles in $\calg_{56}$ is strictly smaller, since $\calg_{56}$ does not contain the overlaps.
Define 
$\tcalg_{56}$  to be the union of all tiles in $\calg_1 \sqcup \calg_2$ which are not in $\calg_{56}.$

Similarly, if $\gamma_1$ is a self-crossing arc or loop with non-empty local overlap, let $x_{\gamma_1}$ be the corresponding Laurent polynomial and $\calg_1$ be  the snake or band graph with corresponding self-overlap $\calg.$ Again, denote by $\calg_{34}$ and $\calg_{56}$ the two elements of $\mathcal{R}$, given by the resolution $\calg_1=\calg_{34}+\calg_{56}$, such that 
the number of tiles in $\calg_{34}$ is equal to the number of tiles in $\calg_1$, whereas the number of tiles in $\calg_{56}$ is strictly smaller, since $\calg_{56}$ does not contain the overlaps.
If $\calg_{56} $ is a positive element of $\mathcal{R}$, define
$\tcalg_{56}$  to be the union of all tiles in $\calg_1 $ which are not in $\calg_{56}$, and if $\calg_{56} $ is a negative element of $\mathcal{R}$, define
$\tcalg_{56}$  to be the union of all tiles in $\calg_{34} $ which are not in $\calg_{56}$.

In all cases, under the bijections of section \ref{section 4} the matchings $P_{56} $ of $\calg_{56}$  are completed to matchings of $\calg_1\sqcup\calg_2$ (respectively $\calg_1$ or $\calg_{34} $) in   a unique way which does not depend on $P_{56}$.  
Moreover, the $y$-monomial of the completion is maximal on a connected subgraph of $\tcalg_{56}$ and trivial on its complement.
We  denote by $\tcalg_{max}$ the component on which the $y$-monomial is maximal.

 Let $\re 12$ be the resolution of the crossing of $\calg_1$ and $\calg_2$ at $\calg$ and  $\ree 1$ the resolution of the self-crossing of $\calg_1$  at $\calg$. 
Define the {\em Laurent polynomial of the resolutions} by

\begin{equation}\label{laurent12}
 \call (\re 12) = \call (\calg_{34}) + y( {\tcalg_{max}}) \call (\calg_{56}),
\end{equation}
and

\begin{equation} \label{laurentself}
 \call (\ree 1) = \call (\calg_{34}) + y( {\tcalg_{max}}) \call (\calg_{56}),
\end{equation}
where 

\begin{equation*}
 \call (\calg ) = \frac{1}{x(\calg)} \sum_{P \in \match(\calg)} x(P) y(P),  \quad\textup{if $\calg$ is positive in $\mathcal{R}$;}
\end{equation*}
and 
\[ \call(\calg)=-\call(-\calg), \quad\textup{if $\calg$ is negative in $\mathcal{R}$.}
\]
\begin{thm} \label{laurent}
\begin{enumerate}
\item 
 Let $\gamma_1$ and $\gamma_2$ be arcs or loops which cross with a non-empty local overlap and let $\calg_1$ and $\calg_2$ be the corresponding snake or band graphs with local overlap $\calg.$ Then
 
\begin{equation*}
 \call(\calg_1 \sqcup \calg_2) = \call (\re 12).
\end{equation*}

\item
 Let $\gamma_1$ be a self-crossing generalised  arc or self-crossing loop with  a non-empty local overlap and let $\calg_1$ be the corresponding snake or band graph with local overlap $\calg.$ Then
 
\begin{equation*}
 \call(\calg_1 ) = \call (\ree 1).
\end{equation*}

\end{enumerate}\end{thm}

\begin{proof}
 If none of the curves $\zg_1,\zg_2$ is a loop, this is Theorem 7.1 of \cite{CS2}.  The essential step of the proof is to show that the switching operation of section \ref{switching} is weight preserving. {That is, if $\calg$ is a (union of) labeled snake and band graphs coming from an unpunctured surface,  $P\in\match \calg$, and $P'$ is obtained from $P$ by a switching operation, then $x(P)=x(P')$ and $y(P)=y(P')$.} 
Then, since the bijection on perfect matchings of section~\ref{section 4} is defined using switching and restriction, it is also weight preserving. To finish the proof one needs to take care of the missing tiles in $\calg_{56}$ and show that the $y(\tcalg_{max})$ is absorbing this discrepancy. 
The proof for the case where $\zg_1$ or $\zg_2$ or both  are loops is analogous.
\end{proof}

\subsubsection{Empty overlaps} For completeness we recall the following results from \cite{CS, CS2}. Let $\gamma_1$ be an arc or a loop and let $\zg_2$ be an arc which cross in a triangle $\Delta$ with an empty overlap. We may assume without loss of generality that $\Delta$ is the first triangle $\gamma_2$ meets. Let $x_{\gamma_1}$ and $x_{\gamma_2}$ be the corresponding Laurent polynomials and $\calg_1$ and $\calg_2$ be their associated snake or band graphs, respectively. 

We know from \cite{CS2} that the snake graphs of the arcs obtained by smoothing the crossing of $\gamma_1$ and $\gamma_2$ are given by the resolution $\gt s{\e_3}12$ of the grafting of $\calg_2$ on $\calg_1$ in $G_s,$ where $s$ is such that $\Delta=\Delta_s$ is the triangle $\gamma_1 $ meets after its $s$-th crossing point, and, if $s=0,$ then $\e_3$ is the unique side of $\Delta$ which is not crossed neither by $\gamma_1$ nor $\gamma_2.$

The edge of $G_s$ which is the glueing edge for the grafting is called the {\em grafting edge}. We say that the grafting edge is {\em minimal in $\calg_1$} if it belongs to the minimal matching on $\calg_1.$

Recall that $\gt s {\e_3}12$ is a pair $(\calg_3 \sqcup \calg_4), (\calg_5 \sqcup \calg_6).$ Let
$\tcalg_{34}$ to be the union of all tiles in $\calg_1\sqcup\calg_2$ that are not in $\calg_{34}$ and
$\tcalg_{56}$ be the union of all tiles in $\calg_1\sqcup\calg_2$ that are not in $\calg_{56}$.
Define

\begin{align*}
 \call (\gt s{\e_3}12) = y_{34} \call (\calg_3 \sqcup \calg_4) + y_{56} \call (\calg_5 \sqcup \calg_6),
\end{align*}
where

\begin{align} \label{skeincoeff}
\left\{\begin{array}{llll}
 y_{34}=1& \textup{and} &y_{56} = y(\tcalg_{56})
 & \mbox{ if the grafting edge is minimal in $\calg_1$; } \\
  y_{34} = y(\tcalg_{34})
  &\textup{and} &y_{56}=1 
  & \mbox{ otherwise. } 
\end{array}\right.
\end{align}
\begin{thm}\cite[Theorem 6.3]{CS} \label{laurent2}
 With the notation above, we have
\begin{align*}
 \la 12 = \call (\gt s{\e_3}12).
\end{align*}
\end{thm}

Similarly, if $\gamma_1$ is a generalised arc which self-crosses in a triangle $\Delta$ with an empty overlap, let $\calg_1$ be the associated snake graph and $x_{\gamma_1}$ be the corresponding Laurent polynomial. We know from Theorem \ref{smoothingselfgrafting} that the snake graphs of the arcs obtained by smoothing the self-crossing of $\gamma_1$ are given by the resolution $\graft _{s,{\e_3}}(\calg_1)$ of the self-grafting of $\calg_1$  in $G_s,$ where $s$ is such that $\Delta=\Delta_s$ and,  if $s=d,$ then $\e_3$ is the unique side of $\Delta$ which is not crossed by $\gamma_1$.

Let $\tcalg_{34}$ be the union of tiles in $\calg_1$ that are not in $\calg_3
\sqcup\calg^{\circ}_4$ and $\tcalg_{56}$ be the union of tiles in $\calg_1$ that are not in $\calg_{56}.$

If $s<d$,  then  $\e_3 $ is  the north or the east edge in $G_s$, and   we let

\begin{align} \label{eq722a}
\left\{\begin{array}{llll}
 y_{34}=1& \textup{and} &y_{56} = y(\tcalg_{56})
 & \mbox{ if $\zd_3$ is minimal in $\calg_1$; } \\
  y_{34} = y(\tcalg_{34})
  &\textup{and} &y_{56}=1 
  & \mbox{ otherwise. } 
\end{array}\right.
\end{align}

If $s=d$, then $\calg_1$ and $\calg_{3}\sqcup\calg^{\circ}_4$ have the same tiles, so 
$y_{34}=1.$
On the other hand,  $\calg_1 \setminus\calg_{56}$ has two components $\calg_1[1,k'-1]$ containing the glueing edge $\delta_3'$ and  $\calg_1[k+1,d]$ containing the glueing edge  $\delta_3$. We let
\begin{equation}\label{eq722b}y_{34}=1,\qquad y_{56}=y_{56}' y_{56}'' ,\end{equation}
where
\[y_{56}'=\left\{\begin{array}{ll}  y(\calg_1[1\,,k'\!-\!1]) &\textup{if  $\delta_3'$ is minimal in $\calg_1$};\\
           1&\textup{otherwise;}\end{array}\right.\]
\[y_{56}''= \left\{\begin{array}{ll}y(\calg_1[k+1\,,d]) &\textup{if  $\delta_3$ is minimal in $\calg_1$};\\
          1 & \textup{otherwise. }\end{array}\right.\]
%
With this notation define

\begin{align*}
 \call (\graft _{s,{\e_3}}(\calg_1)) = y_{34} \call (\calg_3 \sqcup \calg_4^\circ) + y_{56} \call (\calg_{56}).
\end{align*}

\begin{thm} \cite[Theorem 7.4]{CS2}\label{laurent2self}
 With the notation above, we have
\begin{align*}
 \call (\calg_1) = 
 \call (\graft _{s,{\e_3}}(\calg_1)) 
\end{align*}
\end{thm}


\subsection{Skein relations}\label{sect skein}
As a corollary we obtain a new proof of the skein relations.
\begin{cor} \label{skein}
Let $C,$ $ C_{+}$, and $C_{-}$ be as in Definition \ref{def:smoothing}.
Then we have the following identity in the cluster algebra $\cala$,
\begin{equation*}
x_C = \pm Y_1 x_{C_+} \pm Y_2 x_{C_-}.
\end{equation*}
Moreover the coefficients $Y_1$ and $Y_2$ are given by
\textup{(\ref{laurent12}), (\ref{laurentself}), (\ref{skeincoeff}), (\ref{eq722a}) and (\ref{eq722b}).}\qed
\end{cor}

\section{The snake ring}\label{sect 7}
 In this section, we introduce several rings, called snake rings, that are related to cluster algebras. To define these rings we first introduce a ring structure on the group $\mathcal{R}$, where the multiplication is given by the disjoint union, and then take the quotient of this ring by the ideal generated by all resolutions.
 For several choices of labels on snake graphs we can do  a similar construction replacing  $\mathcal{R}$ by its labeled version $\mathcal{LR}$.

\subsection{The ideal of resolutions}\label{sect coeff}
If $\calg=(G_1,\ldots,G_d)$ is a labeled snake graph with sign function $f$, we define the {\em minimal matching} $P_-$ of $\calg$ 
to be the unique matching of $\calg$ that consists of boundary edges of $\calg$ only and such that the unique edge in $\calgSW\cap P_-$ has sign $-$.
The  {\em maximal matching} $P_+$ of $\calg$ is defined to be 
 the matching complementary
to $P_-$  that consists of boundary edges of $\calg$ only. Thus $P_-\cup P_+$ is the set of \emph{all} boundary edges of $\calg$.

In \cite{CS,CS2} and the current paper, we have established a list of identities of snake and band graphs, the resolutions of crossing overlaps and grafting, which were motivated by bijections on the corresponding sets of perfect matchings. These identities are all of the form 
 \begin{equation}\label{eq 7.1}\calg_{12} =y_{34} \calg_{34}+y_{56}\calg_{56}, \end{equation}
where $\calg_{12},\calg_{34}, \calg_{56} $ are snake or band graphs or pairs of snake or band graphs, and $y_{34}$, $y_{56}$ are monomials in the face labels of $\calg_{12}$ according to section~\ref{sect 6.4}.
Recall that the possible cases for $\calg_{12}$ are the following:
\begin{itemize}
\item[-] a pair of snake graphs with crossing overlap, see \cite[section 2.4]{CS},
\item[-] a snake graph with crossing self-overlap, see \cite [section 3.2]{CS2},
\item[-] a snake graph and a band graph with crossing overlap, see section \ref{sect 3.1},
\item[-] a pair of band graphs with crossing overlap, see section  \ref{sect 3.2},
\item[-] a band graph with crossing self-overlap, see section  \ref{sect 3.3},
\item[-] a pair of snake graphs with grafting, see  \cite[section 2.5]{CS},
\item[-] a snake graph with self-grafting, see \cite[section 3.3]{CS2},
\item[-] a snake graph and a band graph with grafting, see section  \ref{sect 3.3graft}.
\end{itemize}

We shall work in several rings of snake and band graphs, and for each of these rings we consider 
the  ideal generated by all 
the relations (\ref{eq 7.1}). We call this ideal the \emph{ideal of resolutions}.

\subsection{Geometric type and unpunctured type}
A labeled snake or band graph $\calg$ is called \emph{geometric} if there exists a triangulated surface $(S,M,T)$ and a curve $\zg$ in $S$ whose labeled snake or band graph $\calg_{S,M,T,\zg}$ is equal to $\calg$. A geometric labeled snake or band graph is called \emph{unpunctured} if the surface $(S,M)$ has no punctures. An unlabeled snake graph is called geometric or unpunctured, if there exists a labelling on $\calg$ which is geometric or unpunctured, respectively.
So in particular, for a geometric snake graph, the edge labels must be cluster variables $x_{S,M,T,i}$ and the face labels must be principal coefficient variables $y_{S,M,T,i}$ of the cluster algebra $\cala(S,M,T)$.
In Figure \ref{unpunctured}, the snake graph on the left is unpunctured since it is the snake graph of an arc in a triangulated hexagon, whereas the snake graph on the right is not geometric since the face labels do not match the corresponding edge labels as described in the following Lemma.

\begin{figure}
\scalebox{1}{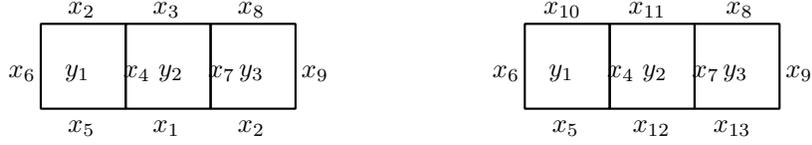}
\caption{An example of an unpunctured snake graph (left) and a non-geometric snake graph (right).}\label{unpunctured}
\end{figure}

\begin{figure}
\scalebox{.95}{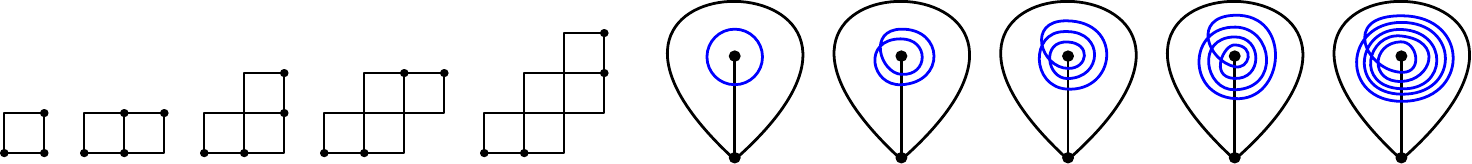}
\caption{ Band graphs all of whose interior edges have the same sign (left) and their geometric realisation in a punctured monogon (right).}\label{fig selffoldloop}
\end{figure}

\begin{lem} \begin{itemize}
\item [\textup{(a)}]
Let $\calg=(G_1,G_2,G_3)$ be a  connected subsnake graph of a labeled geometric snake graph and let $y_{S,M,T,i}$ be the label of the tile $G_2$. 
Then the labels of the boundary edge in $G_1^{N\!E}$ and the boundary edge in ${}_{SW}G_3$ are both equal to $x_{S,M,T,i}$.
\item [\textup{(b)}] In a labeled geometric snake graph, the labels $y_{S,M,T,i},y_{S,M,T,j}$ of any two adjacent tiles together with the label $x_{S,M,T,k}$  of the interior edge shared by the two tiles correspond to three sides $i,j,k$ of the same triangle in the triangulated surface $(S,M,T)$. In particular, if the surface is unpunctured then $i,j,k$ are distinct.
\end{itemize}
\end{lem}
\begin{proof}
 This follows from the construction of snake graphs from surfaces in section \ref{sect graph}.
\end{proof}

\begin{lem}\label{lem annulus} Let $\band$ be an unlabeled band graph.
\begin{itemize}
\item[\textup{(a)}] If not all interior edges in $\band $ have the same sign, then there exists a triangulated annulus $(S,M,T)$ such that $\band$  is equal to the unlabeled band graph associated to the simple loop in $(S,M)$.
\item[\textup{(b)}]  If  all interior edges in $\band $ have the same sign, then $\band$ is equal to the unlabeled band graph associated to a bracelet around the puncture inside a self-folded triangle in a punctured polygon, see Figure \ref{fig selffoldloop}.
\end{itemize}
\end{lem}

\begin{proof}
 (a) Fix a sign function on $\band$. Let $p$ be the number of interior edges in $\band$ that have sign $+$ and $q$ be the number of interior edges having sign $-$. By our hypothesis, we have $p\ge 1$ and $q\ge 1$. Choose an interior edge $b$ of $\band$ that has sign $+$ and let $\calg=\band_b$ be the snake graph obtained from $\band$ by cutting along $b$. Denote by $b'\in \calgSW$ and $b''\in \calgNE$ the unique edges corresponding to $b$. Let $\calg=(G_1,G_2,\ldots,G_d)$ denote the sequence of tiles of $\calg$. Thus $d=p+q$.
 
 We can realise $\calg$ in a polygon with $d+3$ marked points. More precisely, there is a triangulation of the polygon whose diagonals $\tau_1,\tau_2,\ldots,\tau_d$ correspond to the tiles $G_1,G_2,\ldots,G_d$ of $\calg$ and there is an arc $\zg'$ starting at a marked point  $s'$ on the boundary, then crossing $\tau_1,\tau_2,\ldots,\tau_d$ in order and ending at a marked point $t'$, such that the unlabeled snake graph associated to $\zg'$ is equal to $\calg$. See the left picture in Figure \ref{fig polygon}.
 
 Moreover the edge $b'$ of $\calg$ corresponds to a boundary segment of the polygon that is incident to the starting point $s'$ of $\zg'$ and the edge $b''$ of $\calg$ to a boundary segment incident to the ending point $t'$ of $\zg'$. Furthermore, since $b'$ and $b''$ have the same sign in $\calg$, it follows that $b'$ and $b''$ lie on the same side of $\zg'$ in the polygon. Moreover, there are precisely $p+1$ boundary edges in the polygon that lie on the same side of $\zg'$ as $b'$ and $q+2$ boundary edges on the other side. 
 Let $a'\ne b'$ be the unique other boundary edge of the polygon that is incident to $s'$ and let $a''\ne b''$ be the unique other boundary edge of the polygon that is incident to 
$t'$.

We are now ready to construct the annulus in two steps. These steps are also illustrated in Figure   \ref{fig polygon}.

\begin{figure}
\scalebox{.8}{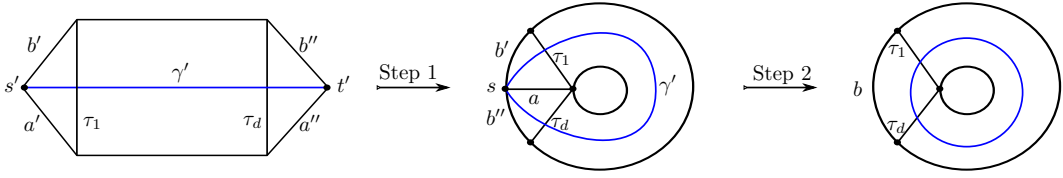}
\caption{Construction of the annulus in Lemma \ref{lem annulus}(a).}\label{fig polygon}
\end{figure}

In the first step, we identify the edges $a'$ and $a''$ and the endpoints $s'$ and $t'$, and we call the resulting arc $a$ and the marked point $s$. We thus obtain an annulus with triangulation $\{\tau_1,\tau_2\ldots,\tau_d,a\}$ and with $p+1$ marked points on one boundary component and $q$ marked points on the other. The boundary component with $p+1$ marked points contains the segments $b'$ and $b''$, and the arc $a$ is an interior arc contained in 2 triangles with sides $a,b',\tau_1$ and $a,b'',\tau_d$, respectively. The curve $\zg'$ becomes an arc starting and ending at the point $s$.

 In the second step, we remove the marked point $s$ as well as the arc $a$, and the two boundary segments $b',b''$ become a single boundary segment $b$. This annulus has $p$ marked points on one boundary component and $q$ marked points on the other. It is triangulated by the arcs $\tau_1,\tau_2\ldots,\tau_d$. The arc $\zg'$ gives rise to the simple loop $\zg$ such that $\calg_{\zg'}^b=\calg^\circ_{\zg}= \band$. This shows (a).

(b) This is immediate as illustrated in Figure \ref{fig selffoldloop}.
\end{proof}

\begin{thm}\label{lem unpunctured}
 \begin{itemize}
\item[\textup{(a)}] Every unlabeled snake graph is geometric and unpunctured.
\item[\textup{(b)}] Every unlabeled band graph is geometric.
\item[\textup{(c)}] An unlabeled band graph is unpunctured if and only if not all of its interior edges have the same sign.
\end{itemize}
\end{thm}

\begin{proof}
  Every unlabeled snake graph can be realised as the snake graph of an arc in a polygon. This implies (a).
  Statement (b) follows from Lemma \ref{lem annulus}.
  
To show (c), let $\band$ be a band graph with tiles $G_1,G_2,\ldots,G_d$ and suppose we have a geometric realisation of $\band$ by a loop $\zg$ in some triangulated surface $(S,M,T)$. Each tile of $\band$ corresponds to a crossing point of the loop $\zg$ with the triangulation $T$. Let $\tau_{i_1},\tau_{i_2},\ldots,\tau_{i_d}$ denote the sequence of arcs of $T$ given by this sequence of crossing points. 
 If the interior edges of $\band$ all have the same sign, then the arcs $\tau_{i_1},\tau_{i_2},\ldots,\tau_{i_d}$ have a common vertex $p$ and the arcs $\tau_{i_j}$ form a fan at $p$. Since $\zg$ is a loop, it then follows that $p$ is a puncture. Thus in this case, $\band $ is not unpunctured.  On the other hand,  if the interior edges of $\band$ do not all have the same sign then $\band$ is unpunctured by Lemma \ref{lem annulus}.
 \end{proof}

\subsection{The geometric unpunctured snake ring}

Let $\calf=\mathbb{Z}[x_{S,M,T,i}^{\pm 1},y_{S,M,T,i}]$ be the ring of Laurent polynomials in $x_{S,M,T,i}$ with coefficients in $\mathbb{Z}[{y_{S,M,T,i}}] $, where $(S,M)$ runs over all   surfaces with marked points and $T$ runs over all triangulations of $(S,M)$.

Let  $\mathcal{LR}$  denote the  $\mathbb{Z}[y_{S,M,T,i}]$-module generated by all isomorphism classes of unions of unpunctured labeled snake graphs and unpunctured labeled band graphs with labels in the ring $\calf$. 
Thus the face labels of the elements of $\callr$ are of the form ${y_{S,M,T,i}}$ and the edge labels are of the form ${x_{S,M,T,i}}$. 

Define a multiplication on $\callr$ by disjoint union of graphs 
\[\calg_1\calg_2=\calg_1\sqcup\calg_2.\]
This defines a ring structure on $\callr$.

\begin{prop}
 $\callr$ is an integral domain with unity element $1=\emptyset$.
\end{prop}
\begin{proof} $\callr $ is clearly commutative and has no zero divisors. 
\end{proof}
 Let $I_{\callr}$ be the ideal of resolutions in $\callr$ as defined in section \ref{sect coeff}. Thus $I_{\callr}$ is generated by all resolutions of labeled crossing overlaps and all labeled grafting in $\callr$.
 Recall that the embeddings of the overlaps must be label preserving. The grafting must be label preserving in the following sense. The grafting edges must have the same label in both graphs, and if the tile containing the grafting edge in one of the graphs has label $y_{S,M,T,i}$ then the label  of the non-grafting edge in the tile containing the grafting edge in  the other graph is $x_{S,M,T,i}$ and thus the face labels of the tiles containing the grafting edges are different from each other, see Figure~\ref{graftinglabels}.

\begin{figure}
\scalebox{1}{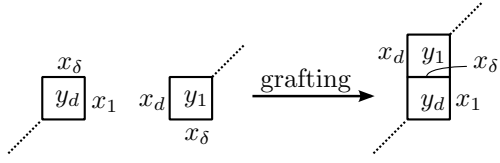}
\caption{Labels for grafting.}\label{graftinglabels}
\end{figure}

\begin{defn}
 The \emph{(unpunctured) snake ring} is the quotient ring 
 \[\calls=\callr/I_{\callr}.\] 
\end{defn}

\begin{remark}
More generally, one can define the \emph{geometric snake ring} $\calls_{geo}=\callr_{geo}/I_{\callr_{geo}}$ by removing the condition that the snake and band graphs are unpunctured.  Similarly one can define the \emph{full snake ring} $\calls_{full}=\callr_{full}/I_{\callr_{full}}$ over an arbitrary integral domain by removing the condition that the snake graphs are geometric. 
\end{remark}
For the rest of this paper, we will be mostly concerned with the unpunctured snake ring, and we will simply refer to it as the snake ring.

\subsection{The homomorphism $\zG$} 
Let $(S,M,T)$ be a triangulated (unpunctured) surface and $\cala$ be the corresponding cluster algebra with principal coefficients in the seed corresponding to the triangulation $T$. Let $\B$ be the bracelet basis defined in section \ref{sect 4}.
Let $\zG:\cala(S,M)\to \calls$ be the map defined on the bracelet basis $\B$ by mapping a basis element to (the class of) its  union of snake and band graphs, by $\zG(1)=\emptyset$, and extended linearly to the whole cluster algebra.

\begin{thm} \label{thm Gamma}The map 
 $\zG$ is an injective ring homomorphism
 \[\zG\colon\cala(S,M)\longrightarrow\calls.\]
\end{thm}

\begin{proof}
 $\zG$ is well-defined, since there are no relations among the basis elements, and $\zG$ is $\mathbb{Z}$-linear by definition. 
 If $b,b'$ are two elements of the bracelet basis $\B$  then let 
 \[bb'=\sum_{b''\in\B} \zl_{b,b'}^{b''} b''\] 
 be the expansion of their product in the basis. 
 By \cite{MSW2} this expansion is obtained by consecutive smoothing of all crossings between the curves associated to $b$ and $b'$ and replacing powers of loops by bracelets (inverse smoothing).  Each smoothing of a crossing in the surface corresponds to a resolution or a grafting relation in $I_{\callr}$
 and thus we have
 \[\zG(b)\zG(b')-\zG\left(\sum_{b''\in\B} \zl_{b,b'}^{b''} b''\right) \in I_{\callr},\] 
and thus $\zG(bb')=\zG(b)\zG(b')$.
Hence $\zG$ is a ring homomorphism.

In order to show injectivity, suppose $\zG(\sum_{b\in\B} \zl_b b)=0$, with $\zl_b\in\mathbb{Z}$. Thus $\sum_{b\in\B} \zl_b \zG(b)$ lies in the ideal $I_{\callr}$.
Now this ideal is generated by all resolutions of crossing overlaps and all grafting relations coming from all triangulated surfaces $(S',M',T')$. Since the labels of $\zG(b)$ come from the fixed surface $(S,M)$ with triangulation $T$, the expression $\zG(b)$ must lie in the ideal generated by the resolution and grafting relations that come from this particular triangulated surface $(S,M,T)$. From the results in section \ref{sect 5} it then follows that  $\sum_{b\in\B} \zl_b b=0$ in the cluster algebra $\cala(S,M)$.
\end{proof}

\begin{cor}
 The ring $\calls$ contains each cluster algebra of unpunctured surface type as a subring.
\end{cor}

\subsection{The homomorphism $\varphi$}
 Recall that if $P$ is a perfect matching of a snake graph $\calg$, then $x(P)$ is the weight of $P$, $y$  is the height of $P$ and $x(\calg)=\mathrm{cross}(\calg)$ is the crossing monomial, see section \ref{secdefloop}.
Let $\bar\varphi\colon \callr\to\calf$ be the map given on unions of snake and band graphs by the formula 
\[\bar \varphi(\calg)=\frac{1}{x(\calg)}\sum_{P\in\match\calg} x(P) y(P),\]
and extended to $\calls$ additively. 
Thus  for any two unions of snake and band graphs $\calg_1$ and $\calg_2$, we have $\bar\varphi(\calg_1+\calg_2)=\bar\varphi(\calg_1)+\bar\varphi(\calg_2).$
Moreover
\[
\begin{array}{rcl}
\bar\varphi(\calg_1\calg_2) &= &\bar\varphi(\calg_1\sqcup\calg_2) \\
 &=&\displaystyle \frac{1}{x(\calg_1\sqcup\calg_2)} \sum_{P\in\match(\calg_1\sqcup\calg_2)} x(P) y(P)\\
 &=&\displaystyle\frac{1}{x(\calg_1)}\frac{1}{x(\calg_2)} \sum_{P_1\in\match\calg_1\atop P_2\in\match\calg_2} x(P_1)x(P_2) y(P_1)y(P_2)\\
 &=&\displaystyle\frac{1}{x(\calg_1)}\sum_{P_1\in\match\calg_1}   x(P_1)y(P_1)
 \
 \frac{1}{x(\calg_2)}  \sum_{ P_2\in\match\calg_2}
 x(P_2) y(P_2)
 \\
&=&\bar\varphi(P_1)\bar\varphi(P_2).\end{array}
\]
We have shown the following lemma.

\begin{lem}
 $\bar\varphi\colon \callr\to \calf$ is a ring homomorphism.
\end{lem}

 $\bar\varphi $ is not injective because it maps a pair of crossing snake graphs to the same element as the resolution of the crossing. More precisely, it follows from section \ref{sect 5} that the kernel of $\bar\varphi$ consists exactly in the ideal $I_{\callr}$. Thus $\bar\varphi $ induces a homomorphism $\varphi$ on the quotient $\calls$. We have the following result.

\begin{thm}\label{thm varphi} The map
 $\varphi$ is an injective ring homomorphism  \[\varphi\colon \calls \longrightarrow \calf .\] \qed
\end{thm}
Since $\calf$ is an integral domain, we obtain the following corollary.
\begin{cor} The snake ring
 $\calls$ is an integral domain.
\end{cor}
Combining Theorems \ref{thm Gamma} and \ref{thm varphi} we get an injective ring homomorphism
\[\xymatrix{\cala(S,M)\ar[r]^-\zG&\calls\ar[r]^\varphi&\calf.}\]
\begin{thm}\label{thm 7.12}
%
 If $b\in\B$ is an element of the bracelet basis of $\cala(S,M)$ then $ \varphi(\zG(b)) $ is the expansion of $x$ in the cluster corresponding to the triangulation $T$ that is used in the definition of $\zG$.
\end{thm}
\begin{proof}
 If $b$ is a  cluster variable this has been shown in \cite{MS} and for the other elements of the basis, this is the definition of $\B$ in \cite{MSW2}.
\end{proof}

\subsection{Specializations}
The snake ring is very large. We are interested in smaller rings that we can obtain by restricting the labels considerably. On the level of the Laurent polynomials, our restrictions will correspond to specializations of the variables $x_{S,M,T,i}$ and $y_{S,M,T,i}$. We consider the specializations
sending all variables $x_{S,M,T,i}$ to a single variable $x$ or to the constant 1, as well as the specializations sending all variables $y_{S,M,T,i}$ to a single variable $y$ or to the constant 1. 
 Many other specializations are interesting and can be defined in a similar way; for example, identifying all the edge labels in the orbit of a cluster automorphims acting on the cluster algebra as defined in \cite{ASS}.

\subsubsection{The ring $F\cals$}
Let $\zs_F$ be the specialization sending each variable $x_{S,M,T,i}$ to the constant 1. Thus 
\[\zs_F\colon\calf=\mathbb{Z}[x_{S,M,T,i}^{\pm1},y_{S,M,T,i}]\to \mathbb{Z}[y_{S,M,T,i}]\] is a surjective ring homomorphism.

On the level of snake rings we define
 $F\calr$ to be the free  $\mathbb{Z}[y_{S,M,T,i}]$-module generated by all isomorphism classes of unions of unpunctured $F$-labeled snake graphs and unpunctured $F$-labeled band graphs, where $F$-labeled means that the labels on all edges are set equal to 1 and the face labels remain unchanged.
Let $I_{F\calr}$ be the ideal or resolutions inside the ring $F\calr$, and let $F\cals=F\calr/I_{F\calr}$.

Clearly, if $z $ is an element of the ideal $I_{\callr}$ of the ring $\callr$ then changing all edge labels in $z$ to 1 will produce an element $\zs_F(z)$ in the ideal $I_{F\calr}$ of the ring $F\calr$. Thus we obtain a commutative diagram of ring homomorphisms  with surjective vertical maps.
\[\xymatrix@C80pt{\calls\,\ar@{^(->}[r]^\varphi\ar@{>>}[d]_{\zs_F}&\calf\ar@{>>}[d]^{\zs_F} \\
F\cals\ar[r]^{\varphi_F}&\mathbb{Z}[y_{S,M,T,i}].
}
\]

 Note that in contrast to $\varphi$, the morphism $\varphi_F$ is not injective. Indeed, $\varphi_F$ maps a 2-tile snake graph with face labels $y_1,y_2$ and the corresponding 2-tile (unpunctured) band graph to the same polynomial $y_1y_2+ y_i+1$, with $i=1,2$ depending on the sign function.

Restricting to the images of the elements of a cluster algebra
we have the following result. 

\begin{thm}
 Let $z$ be an element of a cluster algebra $\cala(S,M,T)$. Then $\zs_F(\varphi(\zG(z)))$ is the $F$-polynomial of $z$.
\end{thm}
\begin{proof}
 This follows from Theorem \ref{thm 7.12} and \cite{MSW2}. 
\end{proof}

\subsubsection{The ring $\calxs$}
Let $\zs_{\calx}$ be the specialization sending each variable $x_{S,M,T,i}$ to the single variable $x$. Thus 
\[\zs_{\calx}\colon\calf=\mathbb{Z}[x_{S,M,T,i}^{\pm1},y_{S,M,T,i}]\to \mathbb{Z}[x^{\pm 1},y_{S,M,T,i}]\] is a surjective ring homomorphism.

On the level of snake rings we define
 $\calx\calr$ to be the free   $\mathbb{Z}[y_{S,M,T,i}]$-module  generated by all isomorphism classes of unions of unpunctured $\calx$-labeled snake graphs and unpunctured $\calx$-labeled band graphs, where $\calx$-labeled means that the labels on all edges are set equal to $x$ and the face labels remain unchanged.
Let $I_{\calxr}$ be the ideal of resolutions inside the ring $\calxr$, and let $\calxs=\calxr/I_{\calxr}$.

 Setting edge labels equal to $x$ does not behave as nicely as specialization to 1. For example the snake graph $\calg$ consisting of a single tile with edge labels equal to 1 and tile label equal to $y_{S,M,T,1}$ for $(S,M)$ a disk with 4 marked points on the boundary, is an element of $\calls$ since it occurs as the snake graph of the unique non-initial cluster variable in the cluster algebra of type $\mathbb{A}_1$. However, this snake graph is not an element of $\calxs$ because the edge labels are not equal to $x$. Thus the specialization map $\zs_{\calx}$ is not a well-defined map from $\calls$ to $\calxs$, not even from $\callr$ to $\calxr$.

However, restricting $\zs_\calx$ to the subring $\calls_\calx$ of $\calls$ given by all elements that have edge labels equal to  some variable $x_{S,M,T,i}$, we obtain a ring homomorphism

\[\zs_{\calx}\colon\calls_{\calx}\to\calxs.\]

 Thus we obtain a commutative diagram of ring homomorphisms
\[\xymatrix@C80pt{\calls_{\calx}\ar@{^(->}[r]^\varphi\ar@{>>}[d]_{\zs_{\calx}}&\calf\ar@{>>}[d]^{\zs_{\calx} }\\
\calxs\ar[r]^{\varphi_{\calx}}&\mathbb{Z}[x^{\pm 1},y_{S,M,T,i}].
}
\]

\subsection{The ring $\calx\caly\cals$}
Let $\zs_{\calx\caly}$ be the specialization sending each variable $x_{S,M,T,i}$ to the single variable $x$ and each $y_{S,M,T,i}$ to the single variable $y$. Thus 
\[\zs_{\calx\caly}\colon\calf=\mathbb{Z}[x_{S,M,T,i}^{\pm1},y_{S,M,T,i}]\to \mathbb{Z}[x^{\pm 1},y]\] is a surjective ring homomorphism.

On the level of snake rings we define
 $\calx\caly\calr$ to be the free $\mathbb{Z}[y]$-module generated by all isomorphism classes of unions of unpunctured $\calx\caly$-labeled snake graphs and unpunctured $\calx\caly$-labeled band graphs, where $\calx\caly$-labeled means that the labels on all edges are set equal to $x$ and the labels on each face are set equal to $y$. 
Let $I_{\calx\caly\calr}$ be the ideal of resolutions inside the ring $\calx\caly\calr$, and let $\calx\caly\cals=\calx\caly\calr/I_{\calx\caly\calr}$.

 Recall from \cite{MSW2} that the poset of perfect matchings is a graded distributive lattice whose rank function is given by the total degree of the $y$-monomial of a perfect matching, that is, if $y(P)=\prod_i y_i^{d_i}$ then the rank of $P$ is $\sum_i d_i$. We denote by $a_i(\calg)$ the number of perfect matchings of rank $i$ in the lattice of perfect matchings of $\calg$. Then $\sum_i a_i(\calg)y^i$ is the rank generating function of the lattice, see \cite[Chapter 3]{Stanley}.
\begin{thm}\label{thm 7.15}
$ \varphi\colon\calx\caly\cals\to \mathbb{Z}[x,y]$ is a ring homomorphism.   Moreover 
\begin{itemize}
\item[\textup(a)] 
if $\calg$ is a snake graph with $d$ tiles then \[\varphi(\calg) = x\sum_{i=0}^d a_i(\calg) y^i \in x\mathbb{Z}[y],\]
\item [\textup(b)] 
if $\band$ is a band graph with $d$ tiles  then \[\varphi(\band) = \sum_{i=0}^d a_i(\calg) y^i \in \mathbb{Z}[y],\]
\item [\textup(c)] 
if $M$ is a union of snake and band graphs then  the multiplicity $m_x(\varphi(M))$ of $x$ in $\varphi(M)$ is the number of snake graphs in $M$.
\end{itemize}
\end{thm}

\begin{proof} To prove that $\varphi $ is a ring homomorphism,
we only need to show that the image of $\varphi$ is in $\mathbb{Z}[x,y]$. Let $\calg$ be a snake graph with $d$ tiles and $\calx\caly$-labels. If $d=0$, then $\calg $ is a single edge and we have $x(\calg)=1$, $x(P)=x$, $y(P)=1$ for the unique matching $P$ of $\calg$. 
 This shows that $\varphi(\calg)=x\in\mathbb{Z}[x,y]$.
 
If $d>0$ then $x(\calg)=x^d$, and, since every perfect matching of $\calg$ has exactly $d+1$ edges, we have $x(P)=x^{d+1}$, for all perfect matchings $P$ of $\calg$. This shows that \[\varphi(\calg)=x\cdot \zs_\caly(F-\textup{polynomial of }\calg)\in x\mathbb{Z}[y],\]
and (a) follows.

Now let $\band$ be a band graph with $d$ tiles and $\calx\caly$-labels. By Lemma \ref{lem unpunctured}, we have $d\ge2$. Then $x(\band) = x^d$ and, since each perfect matching of $\band$ has exactly $d$ edges, $x(P)=x^d$, for all perfect matchings $P$ of $\band$.  Therefore, 
\[\varphi(\band)=\zs_\caly (F-\textup{polynomial of }\band) \in \mathbb{Z}[y],\]
and (b) follows. (c) is a consequence of (a) and (b).
\end{proof}

\begin{remark}
 Let  $\calx\caly\cals_{geo}$ denote the  geometric (punctured) version   of $\calx\caly\cals$. Thus   $\calx\caly\cals_{geo}$ is defined in the same way as   $\calx\caly\cals$ but replacing the condition ``unpunctured" with the condition ``geometric''. Then Theorem \ref{thm 7.15} still holds. Indeed, in the proof we only need to add the case of a single tile band graph $\band$, for which we also have $\varphi(\band)=y+1\in\mathbb{Z}[y].$
\end{remark}

\begin{thm} 
\begin{itemize}
\item [\textup(a)] The unpunctured  $\calx\caly\cals$  is generated by the single edge, the single tile, and all unpunctured band graphs.
\item [\textup(b)] The geometric (punctured) version  $\calx\caly\cals_{geo}$  of $\calx\caly\cals$  is generated by the single edge  and the single tile band graph.
\end{itemize}
\end{thm}

\begin{proof} Let $\calg$ be a snake graph with $\calx\caly$-labels and $d$ tiles, and suppose that $d\ge2$. Let $f$ be a sign function on $\calg$, and let $b$ be the unique edge in $\calgSW$ such that the sign of $b$ is opposite to the sign of the first interior edge of $\calg$. Let $b'$ be the unique edge in $\calgNE$ that has the same sign as $b$.
 We have the following resolution
 \[\calg=\calg^b \sqcup(\textup{single edge }b) +\calg',\]
 where $\calg^b$ is the band graph obtained from $\calg $ by identifying the edges $b$ and $b'$, and $\calg'$ is the snake graph 
 \[\calg'=\calg\setminus\pred( e) \setminus \suc (e') ,\]
 where $e$ is the first interior edge in $\calg$ that has the same sign as $b$ and $e'$ is the last interior edge in $\calg$ that has the same sign as $b$. 
 In particular, $\calg'$ has fewer tiles that $\calg$. Note that the band graph $\calg^b$ has $d$ tiles, with $d\ge 2$, and not all of its interior edges have the same sign, because of our choice of $b$. Therefore it follows from Lemma \ref{lem unpunctured} that $\calg^b$ is unpunctured and hence our resolution lies in the ideal $I_{\calx\caly\calr}$.
 Thus by induction on the number of tiles, this argument proves that every snake graph is generated by the two snake graphs consisting of a single edge and a single tile, together with all unpunctured band graphs.  This shows (a).
 
 (b) The same argument as above shows that every geometric snake graph is generated by the  snake graph consisting of a single edge, together with all geometric band graphs.
 
 Now let $\band$ be a band graph. We want to show that $\band $ is generated by our two generators. Let us assume the contrary, and suppose without loss of generality that $\band $ has a minimal number of tiles among all band graphs that are not generated.
 Suppose first that $\band$ has a subsnake graph consisting of a 4-tile straight segment. Then the second and the third tile can be viewed as two embedded single tile snake graphs, and since the 4-tile segment  is straight, this is a crossing self-overlap. Its resolution decomposes $\band$ into strictly smaller band graphs, each of which has at least two tiles. By induction, this shows that $\band$ is generated by the 2-tile band graph.

 Next suppose that $\band$ has two different subsnake graphs each consisting of a 3-tile straight segment. Then the middle tiles can be viewed as two embedded single tile snake graphs, and since the 3-tile segments  are straight, this is a crossing self-overlap. Again $\band $ decomposes into smaller band graphs and is therefore generated by the 2-tile band graph.
 
  Next suppose that all interior edges of $\band$ have the same sign. Then $\band$ is the $d$-bracelet  of the single tile band graph $\band_1$ and thus $\band=\brac_d(\band_1) =\brac_{d-1}(\band_1)\band_1 -y\, \brac_{d-2} (\band_1)$. By induction, this shows that $\band$ is generated by $\band_1$.
 
 Finally, suppose that $\band $ has no 4-tile straight segment and also no two 3-tile straight segments and that not all interior edges of $\band $ have the same sign. Thus $\band $ is not a complete zigzag graph and it follows that  $\band$ has exactly one 3-tile straight segment. This 3-tile segment has precisely two interior edges, and these two edges have opposite sign, since the segment is straight. Say the first of these two edges has sign $-$ and the second has sign $+$. Since the rest of $\band $ has is a complete zigzag, it follows that there are no other sign changes from one interior edge to the next. Thus every interior edge before the 3-tile straight segment has sign $-$ and every interior edge after the 3-tile segment has sign $+$. But this is impossible, since $\band$ is a band graph.  
\end{proof}

\begin{thm}\label{corgeo}
The rings $\calx\caly\cals_{geo}$ and $ \mathbb{Z}[x,y]$ are isomorphic. \end{thm}
\begin{proof} The isomorphism is given by $\varphi$. To see that $\varphi$ is surjective, observe that
 the image of the single edge is $x$ and if $\band$ is the single tile band graph then \[\varphi(\band -1)=(y+1)-1=y.\]
 $\varphi $ is injective, because it is injective on the generators.
\end{proof}

 \begin{example}\label{ex718}
 Let $\calg$ be the complete zigzag snake graph with $d$ tiles. Then \[\varphi(\calg)=x(1+y+y^2+\cdots +y^d).\]
\end{example}
\begin{example}\label{ex719}
 Let $\calg$ be the straight snake graph with $d$ tiles and such that the first interior edge has sign $+$. Then \[\varphi(\calg)=x\,f_d,\] where $f_d\in \mathbb{Z}[y]$  is given by the recursion $f_0=1$, $f_1=y+1$ and \[f_d=\left\{\begin{array}{ll} y f_{d-1} + f_{d-2} & \textup{if $d$ is even;}\\ f_{d-1}+y^2 f_{d-2} & \textup{if $d$ is odd.}\end{array}\right.\]
This recursion is \textup{A123245} in \cite{OEIS}. Thus for $d=0,1,\ldots 8$, the polynomials $f_d$ are
\[ 1,\  y+1, \ y^2+y+1,\ y^3+2y^2+y+1,\ y^4+2y^3+2y^2+2y+1, \  y^5+3y^4+3y^3+3y^2+2y+1,\]
\[ y^6+3 y^5+4y^4+5y^3+4y^2+3y+1, \quad y^7+4y^6+6 y^5+7y^4+7y^3+5y^2+3y+1,\] 
\[ y^8+4y^6+ 7y^6+10 y^5+11y^4+10y^3+7y^2+ 4y+1.
\]

\end{example}
\begin{remark}
 By the fundamental theorem of finite distributive lattices, every distributive lattice is isomorphic to the lattice $J(\mathcal{P})$ of order ideals of a poset $\mathcal{P}$, see \cite[Theorem 3.4.1]{Stanley}. In Example~\ref{ex718}, this poset $\mathcal{P}$ is $\{1\le 2\le \cdots\le d\} $ and in Example \ref{ex719}, the poset $\mathcal{P}$ is the fence (or zigzag) poset
 given by one of the following two Hasse diagrams or their duals.
 \[\xymatrix@R0pt@C5pt{&\bullet\ar@{-}[rdd]&&\bullet\ar@{-}[rdd]&&&&\bullet\ar@{-}[rdd] &&
\quad &&\bullet\ar@{-}[rdd]&&\bullet\ar@{-}[rdd]&&&&\bullet
 \\ &&&&&\cdots  &&&&\quad\quad\textup{or}\quad\quad&&&&&&\cdots \\
  \bullet\ar@{-}[ruu]&&\bullet\ar@{-}[ruu]&&\bullet&&\bullet\ar@{-}[ruu]&&\bullet
  && \bullet\ar@{-}[ruu]&&\bullet\ar@{-}[ruu]&&\bullet&&\bullet\ar@{-}[ruu]}\]
\end{remark}

\subsection{The ring $\caly\cals$}
Let $\zs_{\caly}$ be the specialization sending each variable $x_{S,M,T,i}$ to the constant 1 and each $y_{S,M,T,i}$ to the single variable $y$. Thus 
\[\zs_{\caly}\colon\calf=\mathbb{Z}[x_{S,M,T,i}^{\pm1},y_{S,M,T,i}]\to \mathbb{Z}[y]\] is a surjective ring homomorphism.

On the level of snake rings we define
 $\caly\calr$ to be the free   $\mathbb{Z}[y]$-module  generated by all isomorphism classes of unions of unpunctured $\caly$-labeled snake graphs and unpunctured $\caly$-labeled band graphs, where $\caly$-labeled means that the labels on all edges are set equal to 1 and the labels on each face are set equal to $y$. 
Let $I_{\caly\calr}$ be the ideal of resolutions inside the ring $\caly\calr$, and let $\caly\cals=\caly\calr/I_{\caly\calr}$.

\begin{remark}
  There is a bijection between the ideal $I_{\caly\calr}$ of resolutions in $\caly\calr$ and the ideal $I_\calr$ of resolutions in the unlabeled ring $\calr$ given by forgetting the face labels.
\end{remark}

\begin{thm}
 $\varphi_\caly\colon\caly\cals\to \mathbb{Z}[y]$ is a surjective ring homomorphism. Moreover, if $\calg$ is a snake or band graph with $d$ tiles then $\varphi_\caly(\calg)=\sum_{i=0}^d a_i(\calg)y^i$ is the rank generating function of the poset $\match \calg$.

\end{thm}
\begin{proof}
 The surjectivity of $\varphi_{\caly}$ follows from $\varphi_\caly(\calg-1)=y+1-1=y$, where $\calg $ is the single tile snake graph.
The proof of the other statements is similar to the proof of Theorem \ref{thm 7.15}.
\end{proof}
We have a commutative diagram of ring homomorphisms
\[\xymatrix@C80pt{\calls\,\ar@{^(->}[r]^\varphi\ar@{>>}[d]_{\zs_\caly}&\calf\ar@{>>}[d]^{\zs_\caly} \\
\caly\cals\ar@{>>}[r]^{\varphi_{\caly}}&\mathbb{Z}[y].
}
\]
Moreover, $\zs_\caly$ factors through $\zs_F$ 
\[\xymatrix@C40pt{\calls\ar@/^20pt/[rr]^{\zs_\caly}\ar[r]_{\zs_F}&F\cals\ar[r]_{\zs'}&\caly\cals}\]
where the map $\zs'$ is given by relabelling all faces with $y$.

\subsection{The ring $\cals$}
Finally we consider the specialization where all labels disappear.
Let $\zs$ be the specialization sending each variable $x_{S,M,T,i}$ and each $y_{S,M,T,i}$ to the constant 1.
 Thus 
\[\zs\colon\calf=\mathbb{Z}[x_{S,M,T,i}^{\pm1},y_{S,M,T,i}]\to \mathbb{Z}\] is a surjective ring homomorphism.

On the level of snake rings we define
 $\calr$ to be the free abelian
group generated by all isomorphism classes of unions of unpunctured unlabeled snake graphs and unpunctured unlabeled band graphs. 
Let $I_{\calr}$ be the ideal or resolutions inside the ring $\calr$, and let $\cals=\calr/I_{\calr}$.

\begin{thm} Let $\calg\in\calr$ be a snake or band graph.
 Then $\varphi  (\calg)$ is the number of perfect matchings of $\calg$.
\end{thm}

In our two examples, we get the following.
 \begin{example}
 If $\cals$ is the complete zigzag snake graph with $d$ tiles then $\varphi(\calg)=d+1.$
\end{example}
\begin{example}
 If $\cals$ be the straight snake graph with $d$ tiles then $\varphi(\calg)$ is the $d+2$-nd Fibonacci number. 
\end{example}

 We end this section with several observations about forgetting the labels.

\begin{thm}
  The maps $\calx\caly\cals\to \caly\cals\to \cals$ given by specializing $x$ to $1$ and $y$ to 1 respectively are ring isomorphisms.
\end{thm}
\begin{proof}
 The inverse of the first map is  given by assigning the label $x$ to each edge and the inverse of the second map is  given by assigning the label $y$ to each tile of each snake and band graph.
\end{proof}

The following lemma holds for arbitrary labeled snake graphs and not only for the geometric ones.
\begin{lem}
 If $\calg_1$ and $\calg_2$ are labeled snake graphs or band graphs with a labeled crossing overlap $\calg$ then the unlabeled versions $\sigma(\calg_1)$ and $\sigma(\calg_2)$ have a crossing overlap $\sigma(\calg)$.
\end{lem}

\begin{proof}
  The embeddings $i_j\colon \calg\to\calg_j, j=1,2$, of a crossing overlap in labeled snake or band graphs are also embeddings into the unlabeled versions of the graphs, and the crossing condition still holds. We only need to show that in the unlabeled version we also have an overlap. To do so we need to check that the maximality condition of section \ref{section2point5}(i) still holds, but this follows from the facts that it was maximal before dropping the labels and that it satisfies the crossing condition.
 \end{proof}

\begin{remark}
 The converse of the lemma is not true. The left picture in Figure \ref{figlabellem} shows an unlabeled snake graph with crossing self-overlap consisting of the first and the last tile. The  two pictures on the right show two labeled versions of this snake graph which are both unpunctured. Indeed the one in the middle can be realised in a heptagon and the one on the right in an annulus with two marked points on one boundary component and one on the other. The snake graph in the middle has distinct labels in each tile, thus it has no overlap, whereas the snake graph on the right has the same crossing overlap as the unlabeled snake graph. 
\end{remark}

\begin{figure}
\scalebox{1.3}
{\small
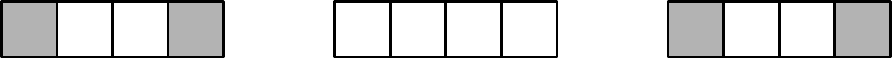}
\caption{An example of an unlabeld snake graph with crossing self-overlap (shaded) on the left, and two labeled versions of the same snake graph, one without self-overlap in the middle and one with crossing self-overlap (shaded) on the right. For simplicity, we only indicate the face labes as integers.}
\label{figlabellem}
\end{figure}

As a consequence of the lemma we get the following result.
\begin{cor}
The map $\calls\to\cals$ given by forgetting the labels is a ring homomorphism. 
\end{cor}
\begin{remark}
 This corollary also holds if we replace the unpunctured snake ring $\calls$ by the geometric or by the full snake ring.
\end{remark}

Summarizing, we have the following commutative diagram of ring homomorphisms
\[\xymatrix@R30pt@C60pt{ &\calls_\calx\ar@{>>}[d]^{\zs_\calx}\ar@{^(->}[r] & \calls\,\ar@{>>}[d]^{\zs_F}\ar@{^(->}[r]^\varphi& \calf\ar@{>>}[d]\\
\mathbb{Z}[x^{\pm1},y_{S,M,T,i}] \ar@{>>}[d] & \calx\cals\ar@{>>}[d]^{\zs'}\ar[l]_{\varphi_\calx}\ar[r]^{x=1} &F\cals\ar[r]^{\varphi_F}\ar@{>>}[d]^{\zs'} &  \mathbb{Z}[y_{S,M,T,i}]\ar@{>>}[d] \\
\mathbb{Z}[x, y] & \calx\caly\cals\ar[l]^\cong_{\varphi_{\calx\caly}}\ar[r]^{x=1}_\cong &\caly\cals\ar@{>>}[r]^{\varphi_\caly}\ar[d]_\cong^{y=1} &  \mathbb{Z}[y]\ar@{>>}[d]  \\
&&\cals\ar@{>>}[r]^{\varphi_\cals}&\mathbb{Z}
}
\]

{}

\end{document}